\documentclass[11pt]{amsart}
\usepackage{amssymb,amsmath,amsfonts,amsthm,enumerate,stmaryrd}
\usepackage{mathrsfs}
\usepackage{graphicx}

\usepackage[left=.75 in, right=.75 in,top=.75 in, bottom=.75 in]{geometry}
\numberwithin{equation}{section}

\usepackage{color}

\providecommand{\abs}[1]{\left\vert#1\right\vert}
\providecommand{\norm}[1]{\left\Vert#1\right\Vert}

\providecommand{\Rn}[1]{\mathbb{R}^{#1}}

\providecommand{\br}[1]{\langle #1 \rangle}

\providecommand{\ns}[1]{\norm{#1}^2}
\providecommand{\bs}[1]{\left[#1\right]_\ell^2}

\providecommand{\ip}[2]{\left(#1,#2\right)}

\providecommand{\jump}[1]{\left\llbracket #1 \right\rrbracket }

\def\nab{\nabla}
\def\dt{\partial_t}
\def\hal{\frac{1}{2}}

\def\ls{\lesssim}
\def\p{\partial}

\def\sg{\mathbb{D}}

\def\da{\Delta_{\mathcal{A}}}
\def\naba{\nab_{\mathcal{A}}}
\def\diva{\diverge_{\mathcal{A}}}
\def\Sa{S_{\mathcal{A}}}

\def\H1{{_0}H^1(\Omega)}

\def\sd{\mathcal{D}}
\def\se{\mathcal{E}}
\def\sdb{\mathcal{D}_{\shortparallel}}
\def\sdbb{\bar{\mathcal{D}}_{\shortparallel}}
\def\seb{\mathcal{E}_{\shortparallel}}

\providecommand{\abs}[1]{\left\vert#1\right\vert}
\providecommand{\norm}[1]{\left\Vert#1\right\Vert}

\providecommand{\ns}[1]{\norm{#1}^2}

\providecommand{\Rn}[1]{\mathbb{R}^{#1}}

\providecommand{\jump}[1]{\left\llbracket #1 \right\rrbracket }
\providecommand{\br}[1]{\langle #1 \rangle}
\providecommand{\pp}[1]{\left( \mspace{-2.0mu} \left( #1 \right) \mspace{-2.0mu} \right)}

\providecommand{\sd}[1]{\mathcal{D}_{#1}}
\providecommand{\se}[1]{\mathcal{E}_{#1}}
\providecommand{\sdb}[1]{\bar{\mathcal{D}}_{#1}}
\providecommand{\seb}[1]{\bar{\mathcal{E}}_{#1}}

\def\hal{\frac{1}{2}}

\def\ls{\lesssim}

\def\jg{\jump{\gamma}}

\def\nab{\nabla}
\def\dt{\partial_t}
\def\p{\partial}

\def\da{\Delta_{\mathcal{A}}}
\def\naba{\nab_{\mathcal{A}}}
\def\diva{\diverge_{\mathcal{A}}}
\def\Sa{S_{\mathcal{A}}}
\def\sg{\mathbb{D}}
\def\sga{\mathbb{D}_{\mathcal{A}}}

\def\Hz{{_0}H^1(\Omega)}
\def\Hzz{{_0^0}H^1(\Omega)}

\def\Lz{\mathring{H}^0(\Omega)}
\def\SH0{\mathcal{H}^0(\Omega)}
\def\SHz{{_0}\mathcal{H}^1(\Omega)}
\def\SHzz{{_0^0}\mathcal{H}^1(\Omega)}

\def\lz{ \overset{\circ}{H}{}^0}
\def\oH{\mathring{H}}
\def\oW{\mathring{W}}

\def\A{\mathcal{A}}

\def\D{\mathcal{D}}
\def\E{\mathcal{E}}

\def\G{\mathcal{G}}
\def\H{\mathcal{H}}

\def\L{\mathcal{L}}

\def\N{\mathcal{N}}
\def\P{\mathcal{P}}
\def\Q{\mathcal{Q}}
\def\R{\mathcal{R}}
\def\V{\mathcal{V}}
\def\W{\mathcal{W}}

\def\af{\mathfrak{A}}

\def\sv{\mathscr{V}}
\def\sw{\mathscr{W}}

\def\swh{\hat{\mathscr{W}}}

\def\st{\;\vert\;}

\def\XXint#1#2#3{{\setbox0=\hbox{$#1{#2#3}{\int}$ }
\vcenter{\hbox{$#2#3$ }}\kern-.6\wd0}}

\DeclareMathOperator{\sgn}{sgn}

\DeclareMathOperator{\diverge}{div}

\DeclareMathOperator{\dist}{dist}

\DeclareMathOperator{\ran}{Ran}

\newtheorem{lem}{Lemma}[section]
\newtheorem{cor}[lem]{Corollary}
\newtheorem{prop}[lem]{Proposition}
\newtheorem{thm}[lem]{Theorem}
\newtheorem{remark}[lem]{Remark}

\newtheorem{dfn}[lem]{Definition}

\title[Stability of contact lines in fluids]{Stability of contact lines in fluids: 2D Stokes Flow}

\author{Yan Guo}
\address{
Division of Applied Mathematics\\
Brown University \\
182 George St., Providence, RI 02912, USA
}
\email[Y. Guo]{guoy@dam.brown.edu}
\thanks{Y. Guo was supported in part by NSFC grant 10828103 and NSF grant DMS-grant 1209437.}

\author{Ian Tice}
\address{
Department of Mathematical Sciences\\
Carnegie Mellon University\\
Pittsburgh, PA 15213, USA
}
\email[I. Tice]{iantice@andrew.cmu.edu}

\begin{document}

\begin{abstract}
In an effort to study the stability of contact lines in fluids, we consider the dynamics of an incompressible viscous Stokes fluid evolving in a two-dimensional open-top vessel under the influence of gravity.  This is a free boundary problem: the interface between the fluid in the vessel and the air above (modeled by a trivial fluid) is free to move and experiences capillary forces.  The three-phase interface where the fluid, air, and solid vessel wall meet is known as a contact point, and the angle formed between the free interface and the vessel is called the contact angle.  We consider a model of this problem that allows for fully dynamic contact points and angles.  We develop a scheme of a priori estimates for the model, which then allow us to show that for initial data sufficiently close to equilibrium, the model admits global solutions that decay to equilibrium exponentially fast.

\end{abstract}

\maketitle

\section{Introduction }

\subsection{Formulation in Eulerian coordinates }

Consider a viscous incompressible fluid evolving in a two-dimensional open-top vessel.  We model the vessel as a bounded, connected, open set $\V \subseteq \Rn{2}$ subject to the following two assumptions. First, we assume that  
\begin{equation}
\V_{top} :=  \V \cap \{y \in \Rn{2} \st y_2 \ge 0 \} = \{y \in \Rn{2} \st -\ell < y_1 < \ell, 0 \le y_2 < L \}
\end{equation}
for some $\ell, L >0$.  This means that the ``top'' part of the vessel, $\V_{top}$, consists of a rectangular channel.  Second, we assume that $\p \V$ is $C^2$ away from the points $(\pm \ell, L)$.  We will write 
\begin{equation}
\V_{btm} :=  \V \cap \{y \in \Rn{2} \st y_2 \le 0 \} 
\end{equation}
for the ``bottom'' part of the vessel.  See Figure \ref{fig:vessel} for a cartoon of two possible vessels.

We will assume that the fluid fills the entirety of the bottom part of the vessel and partially fills the top.  More precisely, we assume that the fluid occupies the moving domain
\begin{equation}
\Omega(t) =  \V_{btm} \cup  \{ y \in \Rn{2} \st -\ell < y_1 < \ell, 0 < y_2 < \zeta(y_1,t)\},
\end{equation}
where the free surface of the fluid is given as the graph of a function $\zeta:[-\ell,\ell] \times \Rn{+} \to \Rn{}$ such that $0 < \zeta(\pm \ell,t) \le L$ for all $t \in \Rn{+}$, which guarantees that the fluid does not ``spill'' out of the vessel top.  We write $\Sigma(t) = \{(y_1,\zeta(y_1,t)) \;\vert \; \abs{y_1} < \ell\}$ for the free surface and $\Sigma_s(t) = \p \Omega(t) \backslash \Sigma(t)$ for the interface between the fluid and the fixed solid walls of the vessel.   See Figure \ref{fig:omega} for a cartoon of two possible fluid domains.

For each $t\ge0$, the fluid is described by its velocity and pressure functions $(u,P) :\Omega(t) \to \Rn{2} \times \Rn{}$.  The viscous stress tensor is determined in terms of $P$ and $u$ according to 
\begin{equation}\label{stress_def}
 S(P,u):= PI - \mu \sg u,
\end{equation}
where $I$ is the $2 \times 2$ identity matrix,  $(\mathbb{D} u)_{ij} = \partial_i u_j + \partial_j u_i$ is the symmetric gradient of $u$, and $\mu>0$ is the viscosity of the fluid.  We write $\diverge S(P,u)$ for the vector with components $\diverge S(P,u)_i = \p_j S(P,u)_{i,j}$; note that if $\diverge{u}=0$ then $\diverge S(P,u) = \nab P - \mu \Delta u$.  

Before stating the equations of motion we define a number of terms that will appear.  We will write $g>0$ for the strength of gravity, $\sigma>0$ for the surface tension coefficient along the free surface, and  $\beta > 0$ for the Navier slip friction coefficient on the vessel side walls.  The coefficients $\gamma_{sv}, \gamma_{sf} \in \Rn{}$ are a measure of the free-energy per unit length associated to the solid-vapor and solid-fluid interaction, respectively.  We set $\jg := \gamma_{sv} - \gamma_{sf}$ and assume that the classical Young relation \cite{young} holds:
\begin{equation}\label{gamma_assume}
 \frac{\abs{\jg}}{\sigma} < 1.
\end{equation}
Finally, we define the contact point velocity response function $\sv: \Rn{} \to \Rn{}$ to be a $C^2$ increasing diffeomorphism such that $\sv(0) =0$.  We will refer to its inverse as $\sw := \sv^{-1} \in C^2(\Rn{})$.

\begin{figure}
\includegraphics[width=2.5 in]{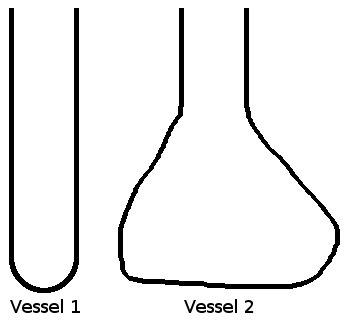}
\caption{Two possible vessels}
\label{fig:vessel}
\end{figure}

\begin{figure}
\includegraphics[width=2.5 in]{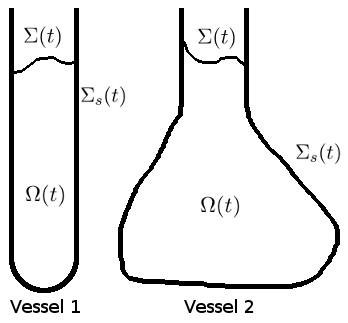}
\caption{Fluid domains}
\label{fig:omega}
\end{figure}

We require that $(u, P, \zeta)$ satisfy the gravity-driven free-boundary incompressible Stokes equations in $\Omega(t)$ for $t>0$: 
\begin{equation}\label{ns_euler}
\begin{cases}
 \diverge{S(P,u)} = \nab P - \mu \Delta u = 0 & \text{in }\Omega(t) \\ 
\diverge{u}=0 & \text{in }\Omega(t) \\ 
S(P,u) \nu = g \zeta \nu - \sigma \H(\zeta) \nu & \text{on } \Sigma(t) \\ 
(S(P,u)\nu - \beta u)\cdot \tau =0 &\text{on } \Sigma_s(t) \\
u \cdot \nu =0 &\text{on } \Sigma_s(t) \\
\partial_t \zeta = u_2 - u_1 \partial_{y_1}\zeta  &\text{on } \Sigma(t) \\ 
\dt \zeta(\pm \ell,t) = \sv\left( \jg \mp \sigma \frac{\p_1 \zeta}{\sqrt{1+\abs{\p_1 \zeta}^2}}(\pm \ell,t) \right)
\end{cases}
\end{equation}
for $\nu$ the outward-pointing unit normal, $\tau$ the associated unit tangent, and
\begin{equation}\label{H_def}
 \mathcal{H}(\zeta) := \p_1 \left( \frac{\p_1 \zeta}{\sqrt{1+\abs{\p_1 \zeta}^2}} \right)
\end{equation}
twice the mean-curvature operator.  Note that in \eqref{ns_euler} we have shifted the gravitational forcing to the boundary and eliminated the constant atmospheric pressure, $P_{atm}$, in the usual way by adjusting the actual pressure $\bar{P}$ according to $P = \bar{P} + g y_2 - P_{atm}$.  Note also that the final equation in \eqref{ns_euler} is equivalent to 
\begin{equation}
\sw( \dt \zeta(\pm \ell,t)) =  \jg \mp \sigma \frac{\p_1 \zeta}{\sqrt{1+\abs{\p_1 \zeta}^2}}(\pm \ell,t),
\end{equation}
which is a more convenient version for our analysis.

We assume that initial data are specified with the initial mass of the fluid given as
\begin{equation}
 M_0 := \abs{\Omega(0)} = \abs{\V_{btm}} +  \int_{-\ell}^\ell \zeta(y_1,0) dy_1.
\end{equation}
We may identify the second term with the mass of the fluid in the top of the vessel, 
\begin{equation}
 M_{top} := \int_{-\ell}^\ell \zeta(y_1,0) dy_1.
\end{equation}
The mass of the fluid is conserved in time since $\dt \zeta = u \cdot \nu \sqrt{1 + \abs{\p_1 \zeta}^2}$:
\begin{equation}\label{avg_prop}
\frac{d}{dt} \abs{\Omega(t)} = \frac{d}{dt}  \int_{-\ell}^\ell \zeta =  \int_{-\ell}^\ell \dt \zeta  = \int_{\Sigma(t)} u \cdot \nu = \int_{\Omega(t)} \diverge{u} = 0.
\end{equation}

We defer a deeper explanation of the system \eqref{ns_euler} to Section \ref{sec_model} and turn now to the construction of equilibrium solutions to \eqref{ns_euler}.

\subsection{Equilibrium state }

A steady state equilibrium solution to \eqref{ns_euler} corresponds to setting $u =0$, $P(y,t) = P_0 \in \Rn{}$, and $\zeta(y_1,t) = \zeta_0(y_1)$ with $\zeta_0$ and $P_0$ solving 
\begin{equation}\label{zeta0_eqn}
 \begin{cases}
 g \zeta_0 - \sigma \H(\zeta_0) = P_0 & \text{on } (-\ell,\ell) \\ 
 \sigma \frac{\p_1 \zeta_0}{\sqrt{1+\abs{\p_1 \zeta_0}^2}}(\pm \ell) = \pm \jg.
 \end{cases}
\end{equation}
Here the boundary conditions follow from the assumptions on the inverse of the velocity response function, $\sw$, which in particular require that $\sw(z) =0$ if and only if $z=0$.  A solution to \eqref{zeta0_eqn} is called an equilibrium capillary surface.

The constant $P_0$ in \eqref{zeta0_eqn} is determined through the fixed-mass condition 
\begin{equation}\label{zeta0_constraint}
 M_{top} =  \int_{-\ell}^\ell \zeta_0(y_1) dy_1.
\end{equation}
Indeed, this allows us to integrate the first equation in \eqref{zeta0_eqn} to compute $P_0$:
\begin{equation}
 2\ell P_0 = \int_{-\ell}^\ell P_0 = \int_{-\ell}^\ell g \zeta_0 - \sigma \H(\zeta_0) = g M_{top} -\sigma \left. \frac{\p_1 \zeta_0}{\sqrt{1+\abs{\p_1 \zeta_0}^2}} \right\vert_{-\ell}^\ell 
= g M_{top} -2 \jg,
\end{equation}
which means that
\begin{equation}\label{p0_def}
 P_0 = \frac{g M_{top} -2 \jg}{2\ell}.
\end{equation}

The problem \eqref{zeta0_eqn} is variational in nature.  Indeed, solutions correspond to critical points of the energy functional $\mathscr{I} : W^{1,1}(-\ell,\ell) \to \Rn{}$ given by
\begin{equation}\label{zeta0_energy}
 \mathscr{I}(\zeta) = \int_{-\ell}^\ell \frac{g}{2} \abs{\zeta}^2 + \sigma \sqrt{1 + \abs{\zeta'}^2} - \jg(\zeta(\ell) + \zeta(-\ell))
\end{equation}
subject to the mass constraint $M_{top} = \int_{-\ell}^\ell \zeta$.  The equilibrium pressure $P_0$ arises as a Lagrange multiplier associated to the constraint.  We now state a well-posedness result for \eqref{zeta0_eqn}, the proof of which we defer to Appendix \ref{app_surf}.

\begin{thm}\label{zeta0_wp}
There exists a constant $M_{min} \ge 0$ such that if $M_{top} > M_{min}$ then there exists a unique solution $\zeta_0 \in C^\infty([-\ell,\ell])$ to \eqref{zeta0_eqn} that satisfies \eqref{zeta0_constraint} with $P_0$ given by \eqref{p0_def}.  Moreover, $\min_{[-\ell,\ell]} \zeta_0 >0$, and if $\mathscr{I}$ is given by \eqref{zeta0_energy}, then $\mathscr{I}(\zeta_0) \le \mathscr{I}(\psi)$ for all $\psi \in W^{1,1}((-\ell,\ell))$ such that $\int_{-\ell}^\ell \psi = M_{top}$.
\end{thm}

\begin{remark}
Throughout the rest of the paper we make two assumptions on the parameters.  First, we assume that $M_{top} > M_{min}$ in order to have $\zeta_0$ as in Theorem \ref{zeta0_wp}.  Second we assume that the parameter $L >0$, the height of the side-walls of $\V_{top}$, satisfies the condition $\zeta_0(\pm\ell) < L$, which means that the equilibrium fluid does not spill out of $\V_{top}$.  
\end{remark}

\subsection{Discussion of the model }\label{sec_model}

We now turn to a discussion of the model \eqref{ns_euler}.  The interface between a fluid, solid, and vapor phase, known as a contact line in 3D or contact point in 2D, has been the subject of serious research for two centuries, dating to the early work of Young in 1805 and continuing to this day.   We refer to the exhaustive survey by de Gennes \cite{degennes} for a more thorough discussion.

Equilibrium configurations (given by \eqref{zeta0_energy}) were studied first by Young \cite{young}, Laplace \cite{laplace}, and Gauss \cite{gauss}.  They determined that the equilibrium contact angle $\theta_{eq}$, the angle formed between the solid wall and the fluid (see Figure \ref{fig:angle}), is determined through a variational principle.  This lead's to the well-known equation of Young: 
\begin{equation}\label{young_relat}
 \cos(\theta_{eq}) = \frac{\gamma_{sf} - \gamma_{sv}}{\sigma} = -\frac{\jg}{\sigma},
\end{equation}
which is related to the assumption \eqref{gamma_assume}.  This relation is enforced in \eqref{zeta0_eqn} in the boundary conditions.

\begin{figure}
\includegraphics[width=2.5 in]{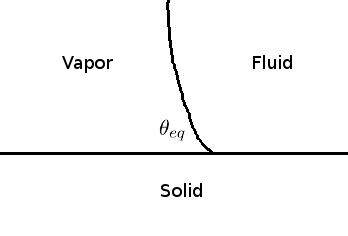}
\caption{Equilibrium angle}
\label{fig:angle}
\end{figure}

The behavior of a contact line or point in a dynamical setting is a much more complicated issue.  The basic problem stems from the incompatibility of the standard no-slip boundary conditions for viscous fluids ($u=0$ at the fluid-solid interface) and the free boundary kinematics ($\dt \zeta = u\cdot \nu\sqrt{1 + \abs{\nab \zeta}^2}$).  Combining the two at the contact point shows that the fluid cannot move along the solid, which is clearly in gross contradiction with the actual behavior of fluids in solid vessels.  This suggests that the no-slip condition is inappropriate for modeling fluid-solid-vapor junctions and that slip must be introduced into the model.

Much work has gone into the study of contact line motion: we refer to the surveys of Drussan \cite{drussan} and Blake \cite{blake} for a thorough discussion of theoretical and experimental studies.  The general picture that has emerged is that the contact line moves as a result of the deviation of the dynamic contact angle $\theta_{dyn}$ from the equilibrium angle $\theta_{eq}$ determined by \eqref{young_relat}.  More precisely, these quantities are related via 
 \begin{equation}\label{cl_motion}
  V_{cl} = F( \cos(\theta_{eq}) - \cos(\theta_{dyn}) ),
 \end{equation}
where $V_{cl}$ is the contact line normal velocity and $F$ is some increasing function such that $F(0)=0$.  These assumptions on $F$ show that the slip of the contact line acts to restore the equilibrium angle.  Equations of the form \eqref{cl_motion} have been derived in a number of ways.  Blake-Haynes \cite{blake_haynes} combined thermodynamic and molecular kinetics arguments to arrive at $F(z) = A \sinh(Bz)$ for material constants $A,B>0$.  Cox \cite{cox} used matched asymptotic analysis and hydrodynamic arguments to derive \eqref{cl_motion} with a different $F$ but of the same general form.  Ren-E \cite{ren_e} performed molecular dynamics simulations to probe the physics near the contact line and also found an equation of the form \eqref{cl_motion}.  Ren-E \cite{ren_e_deriv} also derived \eqref{cl_motion} from constitutive equations and thermodynamic principles.

The above studies confirm that an appropriate model for contact line dynamics involves an equation of the form \eqref{cl_motion} and slip.  The molecular dynamics simulations of Ren-E \cite{ren_e} found that the slip of the  fluid along the solid obeys the well-known Navier-slip condition 
\begin{equation}\label{navier_slip}
u \cdot \nu =0 \text{ and }  S(P,u) \nu \cdot \tau = \beta u \cdot \tau
\end{equation}
for some parameter $\beta >0$.  The former is a no-penetration condition, which prevents the fluid from flowing into the solid, while the latter states that the fluid experiences a tangential stress due to the slipping of the fluid along the solid wall.  Note that the limiting case $\beta=\infty$ corresponds to the no-slip condition.  

The equations \eqref{ns_euler} studied in this article combine the Navier-slip boundary conditions \eqref{navier_slip} with a general form of the contact point equation \eqref{cl_motion}.  Indeed, the last equation in \eqref{ns_euler} may be rewritten as 
\begin{equation}
V_{cl} =  \dt \zeta =  \sv\left( \jg \mp \sigma \frac{\p_1 \zeta}{\sqrt{1+\abs{\p_1 \zeta}^2}}(\pm \ell,t) \right)  = \sv( \sigma (\cos(\theta_{eq}) - \cos(\theta_{dyn} )),
\end{equation}
which is clearly of the form \eqref{cl_motion}. Our assumptions on $\sv$ then correspond to the assumptions on $F$ mentioned above.  Note that in principle we only need $\sv : [-2\sigma,2\sigma] \to \Rn{}$ to be increasing with $\sv(0)=0$, but we can easily make a $C^2$ extension so that $\sv : \Rn{} \to \Rn{}$ is a $C^2$ diffeomorphism.  Thus, it is no loss of generality to assume that $\sv$ is defined on all of $\Rn{}$.  We do this because this makes it simpler to work with $\sw = \sv^{-1}$.  

We use the Stokes equations for the bulk fluid mechanics in \eqref{ns_euler} in order to somewhat simplify the problem and focus on the contact point dynamics.  Thus in some sense the problem \eqref{ns_euler} is quasi-stationary:  the problem is time-dependent and the free boundary moves, but the fluid dynamics are stationary at each time.  In future work, based on the techniques we develop in this paper, we will address the full Navier-Stokes equations coupled to the boundary conditions in \eqref{ns_euler}.

Much work has been devoted to studying contact lines and points in simplified thin-film models; we will not attempt to enumerate these results here and instead refer to the survey by Bertozzi \cite{bertozzi}.  By contrast, there are relatively few results in the literature related to models in which the full fluid mechanics are considered, and to the best of our knowledge none that allow for both dynamic contact point and dynamic contact angle.  Schweizer \cite{schweizer} studied a 2D Navier-Stokes problem with a fixed contact angle of $\pi/2$.  Bodea studied a similar problem with fixed $\pi/2$ contact angle in 3D channels with periodicity in one direction.  Kn\"upfer-Masmoudi studied the dynamics of a 2D drop with fixed contact angle when the fluid is assumed to be governed by Darcy's law.  Related analysis of the fully stationary Navier-Stokes system with free, but unmoving boundary, was carried out in 2D by Solonnikov \cite{solonnikov} with contact angle fixed at $\pi$, by Jin \cite{jin} in 3D with angle $\pi/2$,  and by Socolowsky \cite{socolowsky} for 2D coating problems with fixed contact angles.

Given that our choice of boundary conditions can be derived in all the different ways mentioned above, we believe that the system \eqref{ns_euler} is a good general model for the dynamics of a viscous fluid with dynamic contact points and contact angles.  The purpose of our work is to show that the problem \eqref{ns_euler} is globally well-posed for data sufficiently close to the equilibrium state and that these solutions return to equilibrium exponentially fast.  In other words, we will prove that the equilibrium capillary surfaces are asymptotically stable in this model.  This provides further evidence for the validity of the model.

Let us now turn to a simple motivation for why we should expect solutions to \eqref{ns_euler} to exist globally and decay to equilibrium.  Solutions to the system \eqref{ns_euler} satisfy the following energy-dissipation equality:
\begin{equation}\label{fund_en_evolve}
 \frac{d}{dt} \mathscr{I}( \zeta(\cdot,t)) + \int_{\Omega(t)} \frac{\mu}{2} \abs{\sg u(\cdot,t)}^2 + \int_{\Sigma_s(t)} \frac{\beta}{2} \abs{u(\cdot,t) \cdot \tau} + \sum_{a=\pm 1} \dt \zeta(a \ell,t) \sw(\dt \zeta(a \ell,t))= 0,
\end{equation}
where $\mathscr{I}$ is the energy functional given by \eqref{zeta0_energy} and $\sw = \sv^{-1}$.  Indeed, this can be shown in the usual way by taking the dot product of the first equation in \eqref{ns_euler} with $u$, and integrating by parts over $\Omega(t)$, employing all of the other equations in \eqref{ns_euler}.  Rather than prove the result we refer to Theorem \ref{linear_energy}, which provides the proof of a similar result.

Since $\sw$ is increasing and vanishes precisely at the origin we have that $z \sw(z) >0$ for $z \neq 0$.  This means that the third term on the left of \eqref{fund_en_evolve} provides positive definite control of $\dt \zeta$ at the fluid-solid-vapor contact point.  Thus the latter three terms on the left of \eqref{fund_en_evolve} serve as a positive-definite dissipation functional, and so we can use \eqref{fund_en_evolve} as the basis of a nonlinear energy method.  We also deduce from \eqref{fund_en_evolve} that  $\mathscr{I}(\zeta(\cdot,t))$ is non-increasing, and since $\zeta_0$ is the unique minimizer of $\mathscr{I}$ this suggests that solutions return to equilibrium as $t \to \infty$.  However, the exponential rate of decay to equilibrium is not obvious from \eqref{fund_en_evolve} and requires deeper analysis.

\subsection{Reformulation of \eqref{ns_euler}}

Let $\zeta_0 \in C^\infty[-\ell,\ell]$ be the equilibrium capillary surface given by Theorem \ref{zeta0_wp}.  We then define the equilibrium domain $\Omega \subseteq \Rn{2}$ by
\begin{equation}\label{omega_def}
 \Omega := \V_{btm} \cup  \{ x \in \Rn{2} \;\vert\; -\ell < x_1 < \ell \text{ and } 0 < x_2 < \zeta_0(x_1) \}.
\end{equation}
We write $\p \Omega = \Sigma \sqcup \Sigma_s$, where
\begin{equation}\label{sigma_def}
 \Sigma := \{ x \in \Rn{2} \;\vert\; -\ell < x_1 < \ell \text{ and } x_2 = \zeta_0(x_1) \} \text{ and }
 \Sigma_s := \p \Omega \backslash \Sigma.
\end{equation}
Here $\Sigma$ is the equilibrium free surface, and $\Sigma_s$ denotes the ``sides'' of the equilibrium fluid configuration.  We will write $x \in \Omega$ as the spatial coordinate in the equilibrium domain.

Let us now assume that the free surface to $\Omega(t)$ is given as a perturbation of $\zeta_0$, i.e. we assume that
\begin{equation}
 \zeta(x_1,t) = \zeta_0(x_1) + \eta(x_1,t) \text{ for } \eta:(-\ell,\ell) \times \Rn{+} \to \Rn{}.
\end{equation}
We define 
\begin{equation}\label{extension_def}
 \bar{\eta}(x,t) = \P E \eta(x_1,x_2 - \zeta_0(x_1),t), 
\end{equation}
where $E:H^s(-\ell,\ell) \to H^s(\Rn{})$ is a bounded extension operator for all $0 \le s \le 3$ and $\P$ is the lower Poisson extension given by
\begin{equation}\label{poisson_def}
 \mathcal{P}f(x_1,x_2) = \int_{\Rn{}} \hat{f}(\xi) e^{2\pi \abs{\xi}x_2} e^{2\pi i x_1 \xi} d\xi.
\end{equation}
Let $\phi \in C^\infty(\Rn{})$ be such that $\phi(z) =0$ for $z \le \frac{1}{4} \min \zeta_0$ and $\phi(z) =z$ for $z \ge \hal \min \zeta_0$.  The extension $\bar{\eta}$ allows us to map the equilibrium domain to the moving domain $\Omega(t)$ via the mapping
\begin{equation}\label{mapping_def}
 \Omega \ni x \mapsto   \left( x_1,x_2 +  \frac{\phi(x_2)}{\zeta_0(x_1)} \bar{\eta}(x,t) \right) := \Phi(x,t) = (y_1,y_2) \in \Omega(t).
\end{equation}
Note that 
\begin{equation}
\begin{split}
  \Phi(x_1,\zeta_0(x_1),t) &= (x_1, \zeta_0(x_1) + \eta(x_1,t)) = (x_1,\zeta(x_1,t))  \Rightarrow \Phi(\Sigma,t) = \Sigma(t) \\
  \Phi(\V_{btm},t) &= \V_{btm} \\
  \Phi(\pm \ell, x_2,t) &= (\pm \ell, x_2+ \phi(x_2)\bar{\eta}(\pm\ell ,x_2)/\zeta_0(\pm \ell)) \\
&\Rightarrow  \Phi(\Sigma_s \cap \{x_1 = \pm \ell, x_2 \ge 0\},t) = \Sigma_s(t) \cap \{y_1 = \pm \ell, y_2 \ge 0\}.
\end{split}
\end{equation}
If $\eta$ is sufficiently small (in an appropriate Sobolev space), then the mapping $\Phi$ is a $C^1$ diffeomorphism of $\Omega$ onto $\Omega(t)$ that maps the components of $\p \Omega$ to the corresponding components of $\p \Omega(t)$.

We have
\begin{equation}\label{A_def}
 \nab \Phi = 
\begin{pmatrix}
 1 & 0  \\
 A & J 
\end{pmatrix}
\text{ and }
 \mathcal{A} := (\nab \Phi^{-1})^T = 
\begin{pmatrix}
 1 & -A K \\
 0 & K
\end{pmatrix}
\end{equation}
for 
\begin{equation}\label{AJK_def}
 W = \frac{\phi}{\zeta_0}, \quad
 A = W \p_1 \bar{\eta} - \frac{W}{\zeta_0} \p_1 \zeta_0 \bar{\eta}, \quad 
 J = 1 + W \p_2 \bar{\eta} + \frac{\phi' \bar{\eta}}{\zeta_0}, \quad
 K = J^{-1}.
\end{equation}
Here $J = \det{\nab \Phi}$ is the Jacobian of the coordinate transformation.

We will assume that in fact $\Phi$ is a diffeomorphism.  This allows us to transform the problem \eqref{ns_euler} to one on the fixed spatial domain $\Omega$ for $t \ge 0$.  In the new coordinates, the PDE \eqref{ns_euler} becomes
\begin{equation}\label{geometric_full}
 \begin{cases}
  \diva \Sa(P,u) =  -\mu \da u + \naba P     =0 & \text{in } \Omega \\
 \diva u = 0 & \text{in }\Omega \\
 \Sa(P,u) \N = g\zeta \N -\sigma \H(\zeta) \N & \text{on } \Sigma \\
 (\Sa(P,u)\nu - \beta u)\cdot \tau =0 &\text{on }\Sigma_s \\
 u\cdot \nu =0 &\text{on }\Sigma_s \\
 \dt \zeta = u \cdot \N & \text{on } \Sigma \\
 \sw(\dt \zeta(\pm \ell,t)) =  \jg \mp \sigma \frac{\p_1 \zeta}{\sqrt{1+\abs{\p_1 \zeta}^2}}(\pm \ell,t)   \\
 u(x,0) = u_0(x), \zeta(x_1,0) = \zeta_0(x_1) + \eta_0(x_1).
 \end{cases}
\end{equation}
Here we have written the differential operators $\naba$, $\diva$, and $\da$ with their actions given by $(\naba f)_i := \A_{ij} \p_j f$, $\diva X := \A_{ij}\p_j X_i$, and $\da f = \diva \naba f$ for appropriate $f$ and $X$; for $u\cdot \naba u$ we mean $(u \cdot \naba u)_i := u_j \A_{jk} \p_k u_i$.  We have also written  $\N := -\p_1 \zeta e_1  + e_2$ for the non-unit normal to $\Sigma(t)$,  and we write $\Sa(P,u)  = (P I  - \mu \sg_{\A} u)$ for the stress tensor, where $I$ the $2 \times 2$ identity matrix and $(\sg_{\A} u)_{ij} = \A_{ik} \p_k u_j + \A_{jk} \p_k u_i$ is the symmetric $\A-$gradient.  Note that if we extend $\diva$ to act on symmetric tensors in the natural way, then $\diva \Sa(P,u) = \naba P - \mu \da u$ for vector fields satisfying $\diva u=0$.

Recall that $\A$ is determined by $\eta$ through the relation \eqref{A_def}.  This means that all of the differential operators in \eqref{geometric_full} are connected to $\eta$, and hence to the geometry of the free surface.  This geometric structure is essential to our analysis, as it allows us to control high-order derivatives that would otherwise be out of reach.

\subsection{Perturbation }

We want to consider solutions as perturbations around the equilibrium state $(0,P_0,\zeta_0)$ given by Theorem \ref{zeta0_wp}, i.e. we assume that  $u =0 + u$, $P = P_0 + p$, $\zeta = \zeta_0 + \eta$ for new unknowns $(u,p,\eta)$.  We will now reformulate the equations \eqref{geometric_full} in terms of the perturbed unknowns.

To begin, we use a Taylor expansion in $z$ to write
\begin{equation}
 \frac{y+z}{(1+\abs{y+z}^2)^{1/2}} = \frac{y}{(1+\abs{y}^2)^{1/2}} + \frac{z}{(1+\abs{y}^2)^{3/2}} +   \R(y,z),
\end{equation}
where $\R \in C^\infty(\Rn{2})$ is given by
\begin{equation}\label{R_def}
 \R(y,z) =  \int_0^z 3 \frac{(s-z)(s+y)}{(1+ \abs{y+s}^2)^{5/2}}  ds.
\end{equation}

Then
\begin{equation}
 \frac{\p_1 \zeta}{(1+\abs{\p_1 \zeta}^2)^{1/2}} = \frac{\p_1 \zeta_0}{(1+\abs{\p_1 \zeta_0}^2)^{1/2}} + \frac{\p_1 \eta}{(1+\abs{\p_1 \zeta_0}^2)^{3/2}} +   \R(\p_1 \zeta_0,\p_1 \eta),
\end{equation}
which allows us to use \eqref{zeta0_eqn} to compute
\begin{multline}\label{pert_comp_1}
 g \zeta - \sigma \H(\zeta) = \left( g \zeta_0 - \sigma \H(\zeta_0)\right) 
+ g \eta - \sigma \p_1 \left(\frac{\p_1 \eta}{(1+\abs{\p_1 \zeta_0}^2)^{3/2}}\right) 
-\sigma \p_1 \left(  \R(\p_1 \zeta_0,\p_1 \eta) \right) \\
= P_0  + g \eta - \sigma \p_1 \left(\frac{\p_1 \eta}{(1+\abs{\p_1 \zeta_0}^2)^{3/2}}\right) 
-\sigma \p_1 \left(  \R(\p_1 \zeta_0,\p_1 \eta) \right)
\end{multline}
and 
\begin{multline}
 \jg \mp \sigma \frac{\p_1 \zeta}{\sqrt{1+\abs{\p_1 \zeta}^2}}(\pm \ell,t) =
\jg \mp \frac{\sigma \p_1 \zeta_0}{(1+\abs{\p_1 \zeta_0}^2)^{1/2}}(\pm \ell) \mp \frac{\sigma \p_1 \eta}{(1+\abs{\p_1 \zeta_0}^2)^{3/2}}(\pm \ell,t) \\
\mp   \sigma \R(\p_1 \zeta_0,\p_1 \eta)(\pm \ell,t) 
=  \mp \frac{\sigma \p_1 \eta}{(1+\abs{\p_1 \zeta_0}^2)^{3/2}}(\pm \ell,t) \mp \sigma  \R(\p_1 \zeta_0,\p_1 \eta)(\pm \ell,t).
\end{multline}
On the other hand, 
\begin{multline}\label{pert_comp_2}
\diva \Sa(P,u) = \diva \Sa(p,u) \text{ in } \Omega, \; \Sa(P,u) \N = \Sa(p,u) \N + P_0 \N \text{ on }\Sigma, \\ 
\text{and } \Sa(P,u)\nu \cdot \tau = \Sa(p,u)\nu \cdot \tau \text{ on }\Sigma_s. 
\end{multline}

Next we expand the velocity response function inverse $\sw \in C^2(\Rn{})$.  Since $\sw$ is increasing, we may set 
\begin{equation}\label{kappa_def}
 \kappa = \sw'(0) >0. 
\end{equation}
We then define the perturbation $\swh \in C^2(\Rn{})$ as 
\begin{equation}\label{V_pert}
 \swh(z) = \frac{1}{\kappa} \sw(z) - z.
\end{equation}

We now plug \eqref{pert_comp_1}--\eqref{pert_comp_2} and \eqref{V_pert} into \eqref{geometric_full} to see that $(u,p,\eta)$ solve
\begin{equation}\label{geometric}
 \begin{cases}
 \diva \Sa(p,u)  = -\mu \da u + \naba p     =0 & \text{in } \Omega \\
 \diva u = 0 & \text{in }\Omega \\
 \Sa(p,u) \N = g\eta \N  - \sigma \p_1 \left(\frac{\p_1 \eta}{(1+\abs{\p_1 \zeta_0}^2)^{3/2}} +  \R(\p_1 \zeta_0,\p_1 \eta) \right)\N & \text{on } \Sigma \\
 (\Sa(p,u)\nu - \beta u)\cdot \tau =0 &\text{on }\Sigma_s \\
 u\cdot \nu =0 &\text{on }\Sigma_s \\
 \dt \eta = u \cdot \N & \text{on } \Sigma \\
 \kappa \dt \eta(\pm \ell,t) + \kappa \swh(\dt \eta(\pm \ell,t)) = \mp \sigma \left( \frac{\p_1 \eta}{(1+\abs{\p_1 \zeta_0}^2)^{3/2}} +   \R(\p_1 \zeta_0,\p_1 \eta)\right)(\pm \ell,t)  \\
 u(x,0) = u_0(x), \eta(x_1,0) =  \eta_0(x_1).
 \end{cases}
\end{equation}
Here we still have that $\N$, $\A$, etc are determined in terms of $\zeta = \zeta_0 + \eta$.  Throughout the paper we will write 
\begin{equation}\label{n_0_def}
 \N_0 = -\p_1 \zeta_0 e_1 + e_2
\end{equation}
for the non-unit normal associated to the equilibrium surface.  Then 
\begin{equation}
 \N = \N_0 - \p_1 \eta e_1.
\end{equation}

\section{Main results and discussion}

\subsection{Main results }

In order to state our main results we must first define a number of energy and dissipation functionals.  We define the basic or ``parallel'' (since temporal derivatives are the only ones parallel to the boundary) energy as 
\begin{equation}\label{ed_def_1}
\seb =  \sum_{j=0}^2  \ns{\dt^j \eta}_{H^1((-\ell,\ell))}, 
\end{equation}
and we define the basic dissipation as
\begin{equation}\label{ed_def_2}
 \sdbb = \sum_{j=0}^2 \ns{\dt^j u}_{H^1(\Omega)} +\ns{\dt^j u}_{H^0(\Sigma_s)} + \bs{\dt^j u \cdot \N},
\end{equation}
where we have written
\begin{equation}
 \bs{f} = [f,f]_\ell.
\end{equation}
We also define the improved basic dissipation as
\begin{equation}\label{ed_def_3}
 \sdb = \sdbb +  \sum_{j=0}^2     \ns{\dt^j p}_{H^0(\Omega)} +   \ns{\dt^j \eta}_{H^{3/2}((-\ell,\ell))}.
\end{equation}
The basic energy and dissipation arise through a version of the energy-dissipation equation \eqref{fund_en_evolve}.  However, once we control these terms we are then able to control much more.  This extra control is encoded in the full energy and dissipation, which are defined as follows:
\begin{equation}\label{ed_def_4}
 \se = \seb +  \ns{\eta}_{W^{5/2}_\delta(\Omega)} + \ns{\dt \eta}_{H^{3/2}((-\ell,\ell))}  + \ns{u}_{W^{2}_\delta(\Omega)} + \ns{\dt u}_{H^1(\Omega)} + \ns{p}_{\oW^{1}_\delta(\Omega)} + \ns{\dt p}_{H^0(\Omega)},   
\end{equation}
and
\begin{equation}\label{ed_def_5}
 \sd = \sdb + \ns{\eta}_{W^{5/2}_\delta(\Omega)} + \ns{\dt \eta}_{W^{5/2}_\delta(\Omega)} + \ns{\dt^3 \eta}_{W^{1/2}_\delta(\Omega)}   + \ns{u}_{W^{2}_\delta(\Omega)} + \ns{\dt u}_{W^{2}_\delta(\Omega)} + \ns{p}_{\oW^{1}_\delta(\Omega)} + \ns{\dt p}_{\oW^{1}_\delta(\Omega)}.
\end{equation}
In \eqref{ed_def_4} and \eqref{ed_def_5} the spaces $W^{r}_\delta$ are weighted Sobolev spaces, as defined in Appendix \ref{app_weight}, for a fixed weight parameter $\delta \in (0,1)$.

We now state our main results.  In these we make reference to the corner angles of $\Omega$, which we call $\omega \in (0,\pi)$.  These are related to the equilibrium contact angle, $\theta_{eq} \in (0,\pi)$ given by \eqref{young_relat}, via 
\begin{equation}
 \theta_{eq} + \omega = \pi,
\end{equation}
i.e. $\omega$ is supplementary to $\theta_{eq}$.  Thus we have have that $\omega$ may be computed via 
\begin{equation}
\cos(\omega) = \frac{\zeta_0'(-\ell)}{\sqrt{1+\abs{\zeta_0'(-\ell)}^2}}.
\end{equation}
We now state our a priori estimates for solutions to \eqref{geometric}.

\begin{thm}\label{main_apriori}
Let $\omega \in (0,\pi)$ be the angle formed by $\zeta_0$ at the corners of $\Omega$,  $\delta_\omega \in [0,1)$ be given by \eqref{crit_wt}, and $\delta \in (\delta_\omega,1)$. There exists a universal constant $\gamma >0$ such that if a solution to \eqref{geometric} exists on the temporal interval $[0,T]$ and obeys the estimate
\begin{equation}
 \sup_{0 \le t \le T} \E(t) + \int_0^T \D(t) dt \le \gamma,
\end{equation}
then there exist universal constants $\lambda,C >0$ such that
\begin{multline}
 \sup_{0\le t \le T} \left( \E(t) + e^{\lambda t} \left[ \seb(t) +  \ns{u(t)}_{H^1(\Omega)} + \ns{u(t)\cdot \tau}_{H^0(\Sigma_s)} + \bs{u\cdot \N(t)} + \ns{p(t)}_{H^0(\Omega)} \right]    \right) \\
 + \int_0^T \D(t) dt \le C \E(0).
\end{multline}
\end{thm}
\begin{proof}
 The result is proved later in Theorems \ref{ap_decay} and \ref{ap_bound}.
\end{proof}

The a priori estimates of Theorem \ref{main_apriori} couple with a local existence theory, which we will develop in a companion paper, to prove the existence of global decaying solutions.

\begin{thm}\label{main_gwp} 
Let $\omega \in (0,\pi)$ be the angle formed by $\zeta_0$ at the corners of $\Omega$,  $\delta_\omega \in [0,1)$ be given by \eqref{crit_wt}, and $\delta \in (\delta_\omega,1)$.  There exists a universal smallness parameter $\gamma >0$ such that if 
\begin{equation}
 \E(0) \le \gamma,
\end{equation}
then there exists a unique global solution triple $(u,p,\eta)$ to \eqref{geometric} such that 
\begin{multline}
 \sup_{t \ge 0} \left( \E(t) + e^{\lambda t} \left[ \seb(t) +  \ns{u(t)}_{H^1(\Omega)} + \ns{u(t)\cdot \tau}_{H^0(\Sigma_s)} + \bs{u\cdot \N(t)} + \ns{p(t)}_{H^0(\Omega)} \right]    \right) \\
 + \int_0^\infty \D(t) dt \le C \E(0).
\end{multline}
where $\lambda,C >0$ are  universal constants.
\end{thm}
\begin{proof}
 The result is proved later in Theorem \ref{gwp}.
\end{proof}

In Theorems \ref{main_apriori} and \ref{main_gwp} we have stated the estimates in terms of a mixture of standard and weighted Sobolev norms.  The weighted norms do provide control of standard norms.  For instance, we find in Appendix \ref{app_weight} that 
\begin{equation}
 W^2_\delta \hookrightarrow H^s \text{ and } W^1_\delta \hookrightarrow H^{s-1}
\end{equation}
where our choice of $\delta$ in the theorems allows for the choice of any any 
\begin{equation}
 1 < s < \min\{\frac{\pi}{\omega},2\}.
\end{equation}
We also have that 
\begin{equation}
 W^{5/2}_\delta((-\ell,\ell)) \hookrightarrow H^{s+1/2}((-\ell,\ell))
\end{equation}
so that we can control $\ns{\eta}_{H^{s+1/2}} + \ns{\dt \eta}_{H^{s+1/2}}$.

Theorem \ref{main_gwp} and the above embeddings tell us that 
\begin{equation}
 \sup_{t \ge 0} \left[ \ns{u(t)}_{H^s} + \ns{p(t)}_{H^{s-1}} + \ns{\eta(t)}_{H^{s+1/2}} + \ns{\dt \eta(t)}_{H^{3/2}}  \right] \le C \E(0)
\end{equation}
and that
\begin{equation}
 \sup_{t \ge 0} e^{\lambda t}\left[ \ns{u(t)}_{H^1} + \ns{p(t)}_{H^0} + \ns{\eta(t)}_{H^1} + \ns{\dt \eta(t)}_{H^1}  \right] \le C \E(0).
\end{equation}
We may interpolate between these two estimates to deduce the following.

\begin{cor}
Under the assumptions of Theorem \ref{main_gwp} we have that the following hold.  For any $1 < r < s$ there exists $\lambda_r >0$ such that
\begin{equation}
 \sup_{t \ge 0} e^{\lambda_r t}\left[ \ns{u(t)}_{H^r} + \ns{p(t)}_{H^0}  \right] \le C \E(0).
\end{equation}
For any $1 < r < s+1/2$ there exists $\beta_r >0$ such that
\begin{equation}
 \sup_{t \ge 0} e^{\lambda_r t}\ns{\eta(t)}_{H^r}  \le C \E(0).
\end{equation}
 For any $1 < r < 3/2$ there exists $\beta_r >0$ such that
\begin{equation}
 \sup_{t \ge 0} e^{\lambda_r t}  \ns{\dt \eta(t)}_{H^r}  \le C \E(0).
\end{equation}
\end{cor}

\subsection{Sketch of proof and summary of methods }

We now provide a summary of the principal difficulties encountered in the analysis of \eqref{geometric} and our techniques for overcoming them.  We employ a nonlinear energy method that combines energy estimates, enhanced dissipation estimates, and elliptic estimates in weighted Sobolev spaces into a closed scheme of a priori estimates.

\textbf{Energy estimates:}  The starting point for our analysis is the basic energy-dissipation equation satisfied by solutions to \eqref{geometric}, which comes in essentially the form of a perturbation of \eqref{fund_en_evolve}.  The control provided by the terms in this basic energy-dissipation equation is insufficient for closing a scheme of a priori estimates, so we must move to a higher-regularity context.  In order to employ the energy-dissipation equality to this end, we can only apply differential operators to \eqref{geometric} that are compatible with the boundary conditions.  This leads us to apply temporal derivatives to \eqref{geometric} since they are indeed compatible with the boundary conditions.  Upon doing so and summing we arrive at an equation of the form 
\begin{equation}\label{sum_1}
 \frac{d}{dt} \seb + \sdbb = \mathscr{N},
\end{equation}
where $\seb$ and $\sdbb$ are given by \eqref{ed_def_1} and \eqref{ed_def_2}, and where $\mathscr{N}$ denotes nonlinear interaction terms.   Of course, temporal derivatives do not commute with the other differential operators in \eqref{geometric}, so the nonlinear terms $\mathscr{N}$ become more complicated as the number of applied derivatives grows.

The identity \eqref{sum_1}, which we derive more precisely in Theorem \ref{linear_energy}, forms the basis of our scheme of a priori estimates.  The key term in $\sdbb$ is the third term, which comes from the linearization of $\sw$ in \eqref{fund_en_evolve}.  It provides dissipative control of the contact line velocity $u\cdot \N = \dt \eta$ and its temporal derivatives.  This is essential in our analysis due to the second-to-last equation in \eqref{geometric}, which provides a Neumann-type boundary condition for $\dt^j \eta$ that is compatible with the linearized mean curvature operator appearing in the third equation.  As we describe later, we crucially exploit this in order to gain higher-regularity control of $\dt^j \eta$ in terms of the dissipation.

The goal of a nonlinear energy method is to control the nonlinear term $\mathscr{N}$ in such a way that it can be absorbed onto the left side.  Roughly speaking, we aim to show that 
\begin{equation}\label{sum_3}
 \abs{\mathscr{N}} \le C \seb^\theta \sdbb \text{ for some } \theta >0,
\end{equation}
which when coupled with a bound of the form $\seb \le \delta$ for $\delta>0$ sufficiently small allows for the nonlinear term to be absorbed onto the left side of \eqref{sum_1}, leading to an inequality of the form 
\begin{equation}\label{sum_4}
 \frac{d}{dt} \seb + \hal \sdbb \le 0.
\end{equation}
This requires two crucial ingredients.  First, the nonlinear terms must not exceed the level of regularity controlled by the energy or dissipation.  Second, the nonlinear terms must obey the structured estimates of the form \eqref{sum_3}; for instance, the estimate $\abs{\mathscr{N}} \le C \sdbb^\theta \seb$ cannot be used to derive \eqref{sum_4}.  For the problem \eqref{geometric} neither of these ingredients is available for the energy and dissipation coming directly from the equations.  This dictates that we seek to augment the control provided by $\seb$ and $\sdbb$ by appealing to auxiliary estimates.  Even with these auxiliary estimates in hand, delicate care is still required to show that these ingredients are available.  Our choice of the coordinate system and of the form of the differential operators in \eqref{geometric} are made for this reason, as they have already proven successful in this regard in the analysis of other free boundary problems \cite{gt_hor,gt_inf,jtw}.

The regularity demands needed to prove an estimate of the form \eqref{sum_3} dictate that we move beyond the estimates available strictly through energy-type estimates.  For example, we could not prove \eqref{sum_3} with weighted estimates for first derivatives of $u$.  To control higher-order derivatives we need elliptic estimate, but the ones we would use in smooth domains are unavailable due to the presence of the corners.  Consequently, we are forced to employ weighted Sobolev estimates, which work well in domains with corners.  However, the estimates for the boundary terms are not available in the basic energy or dissipation, so we must first prove some enhance dissipation estimates in order to form a bridge between the basic energy-dissipation estimate and the weighted elliptic theory.

\textbf{Enhanced dissipation estimates:}  Notice that the basic dissipation $\sdbb$ provides no control of either $\eta$ or $p$.  However, by appealing to the structure of the PDEs in \eqref{geometric} we can achieve such control and thereby derive enhanced dissipation estimates.  Our control of the pressure is based on the technique of viewing it as a Lagrange multiplier associated to the diverge-free condition.  Here we adapt this argument to the context of a function space appropriate for weak solutions to \eqref{geometric}.  Interestingly, the estimate for the pressure decouples from $\eta$, and so the control of $u$ provided by $\sdbb$ is sufficient to control $\ns{\dt^j p}_{H^0(\Omega)}$ for $j=0,1,2$.  We prove this in Theorem \ref{pressure_est}.

With the above control of $u$ and $p$ we formally expect from the third equation in \eqref{geometric} that $\p_1^2 \eta \in H^{-1/2}$ and hence that $\eta \in H^{3/2}$.  It turns out that this formal derivative counting can be made rigorous by using the system \eqref{geometric} to derive a weak elliptic problem for $\eta$.  Here the control of the contact line velocity in the basic dissipation is essential, as it provides control of a Neumann boundary condition for $\eta$.  Thus in Theorem \ref{xi_est} we are able to use elliptic estimates to gain control of $\ns{\dt^j \eta}_{H^{3/2}((-\ell,\ell))}$ for $j=0,1,2$.  Interestingly, these estimates decouple from $p$ and are determined only by $u$ terms in $\sdb$.

The decoupling of these estimates is ultimately related to the fact that \eqref{geometric} involves the Stokes problem for the fluid mechanics rather than the Navier-Stokes problem.  Note, though, that we would be able to achieve the same decoupling with the stationary Navier-Stokes problem.   The control provided by the above estimates allows us to move from the dissipation functional $\sdbb$ to the functional $\sdb$, as defined by \eqref{ed_def_3}.  Roughly speaking, this means that \eqref{sum_1} also holds with $\sdb$ replacing $\sdbb$, at the price of introducing more terms to $\N$.

\textbf{Elliptic estimates in weighted Sobolev spaces:}  The enhanced dissipation estimates for $(u,p,\eta)$ are what we would expect for weak solutions to the $\eta-$coupled Stokes problem in \eqref{geometric}.  Here $\eta-$coupling means that $\eta$ is treated as an unknown with an elliptic operator on the boundary rather than as a forcing term (compare \eqref{A_stokes_stress} vs. \eqref{A_stokes_beta}).  The trade-off for this coupling to a new unknown is that we must consider an extra boundary condition on $\Sigma$; indeed, in \eqref{A_stokes_stress} there are three scalar boundary conditions on $\Sigma$, whereas in \eqref{A_stokes_beta} there are only two because there is one fewer unknown. 

The above suggests that we might be able to invoke the higher-order elliptic regularity results of Agmon-Douglis-Nirenberg \cite{adn_2}.  However, the equilibrium domain (and even the dynamic domain at any time) is piecewise $C^2$ but only globally Lipschitz because of the corners formed at the contact points.  It is well-known that the usual elliptic regularity theory fails in domains with corners, as the corners allow for weak singularities, and the results of \cite{adn_2} are inapplicable.  This means that while we can derive the weak-formulation estimates for $(u,p,\eta) \in H^1 \times H^0 \times H^{3/2}$,  we cannot in general expect even $u \in H^2$.  Moreover, the analysis of \cite{adn_2} only applies if we view $\eta$ as a forcing term, and is thus unsuitable for the $\eta-$coupled problem.

The extension of elliptic regularity theory to domains with corner singularities, which originated in the work of \`{E}skin \cite{eskin}, Lopatinski\u{i} \cite{lopatinskii}, and Kondrat'ev \cite{kondra}, is achieved by replacing the standard Sobolev spaces with their weighted counterparts.  The weights are tuned to cancel out the singular behavior near the corners.  It is now well-understood that the relationship between the weights, the corner angles, and the choice of boundary conditions is determined by the eigenvalues of certain operator pencils: we refer to the books of Grisvard \cite{grisvard,grisvard_2} and Kozlov-Maz'ya-Rossman \cite{kmr_1,kmr_2,kmr_3} for exhaustive surveys.  

The particular choice of boundary conditions we use in \eqref{geometric} does not seem to be directly available in the literature, so we have to develop the theory here.  This is the content of Section \ref{sec_elliptics}, which culminates in Theorem \ref{A_stokes_stress_solve}, the weighted $\eta-$coupled Stokes regularity theory for the triple $(u,p, \eta)$.  Fortunately, the boundary conditions in \eqref{geometric} essentially amount to a compact perturbation of boundary conditions for which the operator pencil is well-understood thanks to the work of Orlt-S\"{a}ndig \cite{orlt_sandig}.  Combining these ingredients with existing estimates in \cite{kmr_1,kmr_2,kmr_3} then allows us to prove the theorem.  It is worth noting that the weighted Sobolev estimates imply estimates in the standard Sobolev spaces $W^{k,p}$,  for $1 \le p < 2$, but the weighted estimates are actually sharper.  Our analysis actually depends crucially on the weight terms, and so we cannot simplify our approach by working directly with the derived $W^{k,p}$ estimates.

In order to invoke the weighted elliptic estimates of Theorem \ref{A_stokes_stress_solve} we must have control of the forcing terms.  It is here that the velocity response function plays another essential role.  The control it provides in $\sdbb$ leads to the $\dt^j \eta \in H^{3/2}$ estimates in $\sdb$, which in turn provide exactly the right level of regularity needed to employ the weighted elliptic estimates.  More precisely, in order to get weighted estimates for $(\dt^j u, \dt^j p, \dt^j \eta)$ we must control $\dt^{j+1} \eta \in H^{3/2}$.  This dictates the count $j=0,1,2$ that we use in $\seb$ and $\sdbb$, as we need weighted estimates for $(u,p,\eta)$ and $(\dt u,\dt p,\dt \eta)$ in order to close our a priori estimates.  This extra control then allows us to move from $\sdb$ to $\D$, as defined by \eqref{ed_def_5}.

\textbf{A priori estimates:}  We combine the above ingredients into a scheme of a priori estimates for solutions to \eqref{geometric}.  Our method is a nonlinear energy method based on the full dissipation $\D$ given by \eqref{ed_def_5}, and the full energy $\E$ given by \eqref{ed_def_4}.  In the above analysis we have not mentioned any improvements of $\seb$ from enhanced or elliptic estimates.  These appear to be unavailable, and so we are only able to enhance the control of the energy functional by integrating the dissipation functional in time.  This leads to the extra terms controlled in $\E$.  Thus for \eqref{geometric} we see that the dissipation functional plays the principal role in our nonlinear energy method.  This is in stark contrast with the nonlinear energy methods we have employed in \cite{gt_hor,gt_inf,jtw}, for which the energy and dissipation each play an essential role.  

In Section \ref{sec_nlin_en} we develop the estimates of the nonlinearities that appear in the energy-dissipation equality, i.e. the term $\mathscr{N}$ in \eqref{sum_1}.  The goal is to prove a variant of \eqref{sum_3}, namely an estimate of the form $\abs{\mathscr{N}} \le C \E^\theta \D$.  Here we make very frequent use of the weighted estimates, which is why we choose to work with them directly, and we employ a number of product estimates from Appendix \ref{app_prods}.  The choice of operators and coordinates is also important here, as it keeps the derivative count of the nonlinear terms low enough.  

In Section \ref{sec_nlin_ell} we develop the estimates for the nonlinear terms appearing in the weighted elliptic estimates.  We prove that (roughly speaking)
\begin{equation}
 \D \ls \sdb + \E^\theta \D.
\end{equation}
This structure is again essential in order to work with an absorbing argument.

Finally, in Section \ref{sec_apriori} we complete the a priori estimates by combining the energy-dissipation equality and the nonlinear estimates.  This results in an estimate of the form 
\begin{equation}
 \frac{d}{dt} \seb + \D \ls \E^\theta \D,
\end{equation}
which then shows that if we know a priori that the solution obeys the estimate $\E \le \delta$ for $\delta$ some universal constant, then 
\begin{equation}\label{sum_5}
 \frac{d}{dt} \seb + \hal \D \le 0. 
\end{equation}
From \eqref{sum_5} we then close our a priori estimates by integrating in time to see that 
\begin{equation}
 \sup_{0 \le t\le T} \E(t) + \int_0^T \D(t) dt \ls \E(0).
\end{equation}
The coercivity of the dissipation over the energy, $\seb \ls \D$, combines with  \eqref{sum_5} to allow us to prove decay: (again, roughly)
\begin{equation}
 \sup_{0 \le t \le T} e^{\beta t} \seb(t) \ls \E(0)
\end{equation}
for some universal $\beta >0$

\subsection{Plan of paper}

The paper is organized as follows.  In Section \ref{sec_weakform} we discuss the weak formulation of \eqref{geometric}.  In Section \ref{sec_basic_ests} we develop the basics on which our nonlinear energy method is based.  This includes the energy-dissipation equation as well as the ingredients needed for the enhanced dissipation estimates.  In Section \ref{sec_elliptics} we develop the weighted Sobolev elliptic regularity theory for \eqref{geometric}.  Section \ref{sec_nlin_en} provides estimates for the nonlinear terms that arise in the energy-dissipation equality.  Section \ref{sec_nlin_ell} develops the estimates for the nonlinearities in the elliptic problems.  In Section \ref{sec_apriori} we complete our a priori estimates and record the proofs of our main results.  Appendix \ref{app_nonlin_form} records the precise form of various nonlinearities that appear in our analysis.  Appendix \ref{app_r} records some estimates for the perturbation function $\R$.  Appendices \ref{app_weight} and \ref{app_prods} provide some key details about weighted Sobolev spaces.  Appendix \ref{app_coeff} records some coefficient estimates. Finally, Appendix \ref{app_surf} records useful properties about equilibrium capillary surfaces.

\subsection{Notation and terminology} 

We now mention some of the notational conventions that we will use throughout the paper.

\textbf{Einstein summation and constants:}  We will employ the Einstein convention of summing over  repeated indices for vector and tensor operations.  Throughout the paper $C>0$ will denote a generic constant that can depend $\Omega$, or any of the parameters of the problem.  We refer to such constants as ``universal.''  They are allowed to change from one inequality to the next.   We will employ the notation $a \ls b$ to mean that $a \le C b$ for a universal constant $C>0$.

\textbf{Norms:} We write $H^r(\Omega)$, $H^r(\Sigma)$, and $H^r(\Sigma_s)$ with $r \in \Rn{}$ for the usual Sobolev spaces.   We will typically write $H^0 = L^2$.  To avoid notational clutter, we will often avoid writing $H^r(\Omega)$, $H^r(\Sigma)$, or $H^r(\Sigma_s)$ in our norms and typically write only $\norm{\cdot}_{r}$.  When we need to refer to norms on the space $L^r$ we will explicitly write $\norm{\cdot}_{L^r}$.  Since we will do this for functions defined on $\Omega$,$\Sigma$, and $\Sigma_s$ this presents some ambiguity.  We avoid this by adopting two conventions.  First, we assume that functions have natural spaces on which they ``live.''  For example, the functions $u$,  $p$, and $\bar{\eta}$ live on $\Omega$, while $\eta$ lives on $\Sigma$. Second, whenever the norm of a function is computed on a space different from the one in which it lives, we will explicitly write the space.  This typically arises when computing norms of traces onto $\Sigma$ of functions that live on $\Omega$.

\section{Weak formulation }\label{sec_weakform}

The purpose of this section is to define a number of useful function spaces and to give a weak formulation of the problem \eqref{geometric}.

\subsection{Some spaces and bilinear forms }

Suppose that $\eta$ is given and that $\A$, $J$, $\N$, etc are determined in terms of it.   Let us define 
\begin{equation}
 \pp{u,v} := \int_\Omega \frac{\mu}{2} \sg_{\A} u : \sg_{\A} v J + \int_{\Sigma_s} \beta (u\cdot \tau) (v\cdot \tau) J.
\end{equation}
We also define 
\begin{equation}
 (\phi,\psi)_{1,\Sigma} := \int_{-\ell}^\ell g \phi \psi + \sigma \frac{\p_1 \phi \p_1 \psi }{(1+\abs{\p_1 \zeta_0}^2)^{3/2}}
\end{equation}
and
\begin{equation}
 [a,b]_\ell := \kappa\left( a(\ell)b(\ell) + a(-\ell)b(-\ell)\right).
\end{equation}

We define the time-dependent spaces
\begin{equation}
 \SH0 := \{ u: \Omega \to \Rn{2} \;\vert\; \sqrt{J} u \in H^0(\Omega)  \},
\end{equation}
\begin{equation}
 \SHz := \{ u: \Omega \to \Rn{2} \;\vert\; \pp{u,u}<\infty, u\cdot \nu=0 \text{ on }\Sigma_s \},
\end{equation}
and
\begin{equation}
 \SHzz := \{ u \in \SHz \;\vert\;  u\cdot \N=0 \text{ on }\Sigma  \}.
\end{equation}
Here we suppress the time-dependence in the notation, though each space depends on $t$ through the dependence of $J,\A,\N$ on $t$.

We also define the time-independent spaces
\begin{equation}
 \Lz = \{ q \in H^0(\Omega) \st \int_{\Omega} q =0\},
\end{equation}
\begin{equation}
 \lz(-\ell,\ell) = \{ g \in H^0((-\ell,\ell)) \st \int_{-\ell}^\ell g =0\},
\end{equation}
\begin{equation}
 W = \{ v \in \Hz \;\vert\; u \cdot \N_0 \in H^1(-\ell,\ell) \cap \lz(-\ell,\ell)  \},
\end{equation}
where $\N_0$ is given by \eqref{n_0_def}, and 
\begin{equation}
 V = \{ v \in W \;\vert\; \diverge{v}=0\}.
\end{equation}
We endow both with the natural inner-product on $W$:
\begin{equation}
 (u,v)_W = \int_\Omega \frac{\mu}{2} \sg u: \sg v + \int_{\Sigma_s} \beta (u\cdot \tau)(v\cdot \tau) + (u\cdot \N_0, v\cdot \N_0)_{1,\Sigma}.
\end{equation}
We then define the space
\begin{equation}
  \W(t):= \{ w \in \Hz \;\vert\; v \cdot \N \in H^1(-\ell,\ell) \cap \lz(-\ell,\ell)   \},
\end{equation}
which we endow with the inner-product
\begin{equation}
 (v,w)_{\W} = \pp{v,w} + (v\cdot \N,w\cdot \N)_{1,\Sigma}.
\end{equation}
We also define the subspace 
\begin{equation}
 \V(t) := \{v \in \W(t) \;\vert\; \diva v = 0\},
\end{equation}

\subsection{Formulation }

We now aim to justify a weak formulation of \eqref{geometric}.  Suppose that $\zeta=\zeta_0 + \eta$ and  $\A$ and $\N$ are determined in terms of $\zeta$.  Then suppose that $(v,q,\xi)$ satisfy
\begin{equation}\label{linear_geometric}
 \begin{cases}
   \diva\Sa(q,v)  =F^1 & \text{in } \Omega \\
 \diva v = F^2 & \text{in }\Omega \\
 \Sa(q,v) \N = g\xi \N -\sigma \p_1\left( \frac{\p_1 \xi}{(1+\abs{\p_1 \zeta_0}^2)^{3/2}} + F^3\right) \N  + F^4& \text{on } \Sigma \\
 (\Sa(q,v)\nu - \beta v)\cdot \tau =F^5 &\text{on }\Sigma_s \\
 v\cdot \nu =0 &\text{on }\Sigma_s \\
 \dt \xi = v \cdot \N  + F^6& \text{on } \Sigma \\
 \kappa \dt \xi(\pm \ell,t) = \mp \sigma \left(\frac{\p_1 \xi}{(1+\abs{\p_1 \zeta_0}^2)^{3/2}} + F^3 \right)(\pm \ell,t) - \kappa F^7(\pm\ell,t).
 \end{cases}
\end{equation}

We have the following integral identity  $(v,q,\xi)$.

\begin{lem}\label{geometric_evolution}
Suppose that $u$ and $\zeta$ are given as above and that $(v,q,\xi)$ are sufficiently regular and satisfy \eqref{linear_geometric}.    Suppose that $w \in \W(t)$.  Then
\begin{multline}\label{ge_0}
  \pp{v,w} - (p,\diva w)_0 + (\xi,w\cdot \N)_{1,\Sigma} + [v\cdot \N,w\cdot \N]_\ell  
 = \int_\Omega F^1 \cdot w J  \\- \int_{\Sigma_s} J (w\cdot \tau)F^5 
- \int_{-\ell}^\ell \sigma  F^3   \p_1 (w \cdot \N) + F^4 \cdot w  
  - [w\cdot \N, F^6 + F^7]_\ell.
\end{multline}
\end{lem}
\begin{proof}
 
We take the dot product of the first equation  in \eqref{linear_geometric} with $J w$ and integrate over $\Omega$ to find that 
\begin{equation}\label{ge_1}
\int_\Omega \diva \Sa(q,v) \cdot w J =:  I  = II := \int_\Omega F^1 \cdot w J.
\end{equation}
We will compute $I$  and then plug into \eqref{ge_1} to deduce \eqref{ge_0}.  In computing, we will utilize the geometric identity $\p_k(J \A_{jk}) =0$ for each $j$.  We will also employ the identity
\begin{equation}\label{ge_4}
 J \A \nu = 
\begin{cases}
 J \nu &\text{on }\Sigma_s \\
 \N/\sqrt{1 +\abs{\p_1 \zeta_0}^2} &\text{on }\Sigma,
\end{cases}
\end{equation}
which follows from a straightforward computation.

Now we turn to the computation of $I$.  The geometric identity and an integration by parts allow us to rewrite
\begin{equation}\label{ge_7}
 I = \int_\Omega \p_k (J \A_{jk} \Sa(q,v)_{ij}) w_i = \int_\Omega -J \A_{jk} \p_k w_i \Sa(q,v)_{ij} + \int_{\p \Omega} (J \A \nu) \cdot (\Sa(q,v) w) 
:= I_1 + I_2.
\end{equation}
We may use the definition of $\Sa(q,v)$ to compute
\begin{equation}
 I_1 = \int_\Omega \frac{\mu}{2} \sga v: \sga w J - q \diva{w} J.
\end{equation}
On the other hand, the first equality in \eqref{ge_4} allows us to compute 
\begin{multline}
 \int_{\Sigma_s} (J \A \nu) \cdot (\Sa(q,v) w) = \int_{\Sigma_s} J \nu \cdot (\Sa(q,v) w) = 
\int_{\Sigma_s} J w \cdot (\Sa(q,v) \nu) \\
= \int_{\Sigma_s} J \left(\beta (v \cdot \tau) (w \cdot \tau) + w\cdot \tau F^5\right). 
\end{multline}
Similarly, the second equality in \eqref{ge_4} shows that
\begin{multline}
 \int_{\Sigma} (J \A \nu) \cdot (\Sa(q,v) w) \\ = 
\int_{-\ell}^\ell (\Sa(q,v) \N)\cdot w = \int_{-\ell}^\ell g \xi (w \cdot \N) - \sigma \p_1 \left( \frac{\p_1 \xi }{(1+\abs{\p_1 \zeta_0}^2)^{3/2}} +F^3\right)w\cdot  \N + F^4 \cdot w.
\end{multline}
We then compute the  second of these terms by using the equations in \eqref{linear_geometric}:
\begin{multline}
 \int_{-\ell}^\ell - \sigma \p_1 \left( \frac{\p_1 \xi }{(1+\abs{\p_1 \zeta_0}^2)^{3/2}} +F^3 \right)w\cdot \N  \\
 =
\int_{-\ell}^\ell \left( \frac{\p_1 \xi }{(1+\abs{\p_1 \zeta_0}^2)^{3/2}} +F^3 \right)\p_1 (w\cdot \N)
- \sigma \left. \left( \frac{\p_1 \xi }{(1+\abs{\p_1 \zeta_0}^2)^{3/2}} +F^3 \right) (w\cdot \N) \right\vert_{-\ell}^\ell 
\end{multline}
and 
\begin{multline}\label{ge_8}
- \sigma \left.  \left( \frac{\p_1 \xi }{(1+\abs{\p_1 \zeta_0}^2)^{3/2}} +F^3 \right)  (w\cdot \N) \right\vert_{-\ell}^\ell =
 \sum_{a=\pm 1} \left(\kappa \dt\xi(a\ell) + \kappa F^7(a\ell)  \right)(w\cdot \N(a\ell) ) \\
 = 
 \sum_{a=\pm 1} \kappa (v \cdot \N(a\ell))(w\cdot \N(a\ell) ) + \kappa (w\cdot\N(a\ell)) ( F^6(a\ell) + F^7(a\ell) ).
\end{multline}

Combining \eqref{ge_7}--\eqref{ge_8}, we find that
\begin{multline}\label{ge_9}
 I =   \int_\Omega \frac{\mu}{2} \sga v: \sga w J + \int_{\Sigma_s} J \beta (v \cdot \tau)(w \cdot \tau)  + \int_{\ell}^\ell g \xi (w\cdot \N) + \sigma  \frac{\p_1 \xi }{(1+\abs{\p_1 \zeta_0}^2)^{3/2}} \p_1(w \cdot \N)  \\
+ \sum_{a=\pm 1} \kappa (v \cdot \N(a\ell)) (w \cdot \N(a\ell)) 
- \int_\Omega q \diva{w} J + \int_{\Sigma_s} J (w\cdot \tau)F^5 
+ \int_{-\ell}^\ell \sigma  F^3  \p_1( w \cdot \N) + F^4 \cdot w  \\
  + [w\cdot \N, F^6 + F^7]_\ell.
\end{multline}
Plugging this into \eqref{ge_1} yields \eqref{ge_0}.

\end{proof}

The lemma motivates our definition of a weak solution.

\begin{dfn}
A weak solution to \eqref{linear_geometric} is a triple $(v,q,\xi)$
that satisfies
\begin{multline}\label{weak_form}
  \pp{v,w} - (q,\diva w)_0 + (\xi,w\cdot \N)_{1,\Sigma} + [v\cdot \N,w\cdot \N]_\ell  
 = \int_\Omega F^1 \cdot w J  \\- \int_{\Sigma_s} J (w\cdot \tau)F^5 
- \int_{-\ell}^\ell \sigma  F^3   \p_1(w \cdot \N) + F^4 \cdot w  
  - [w\cdot \N,F^6 +F^7 ]_\ell.
\end{multline}
 for a.e. $t \in [0,T]$ and for every $w \in \W(t)$.
\end{dfn}

\section{Basic estimates }\label{sec_basic_ests}
 
\subsection{The energy estimate }

We have the following equation for the evolution of the energy of $(v,q,\xi)$.

\begin{thm}\label{linear_energy}
Suppose that $\zeta=\zeta_0 + \eta$ is given and  $\A$ and $\N$ are determined in terms of $\zeta$.  Suppose that $(v,q,\xi)$ satisfy \eqref{linear_geometric}.    Then
\begin{multline} \label{linear_energy_0}
 \dt \left(   \int_{-\ell}^\ell \frac{g}{2} \abs{\xi}^2 + \frac{\sigma}{2} \frac{\abs{\p_1 \xi}^2}{(1+\abs{\p_1 \zeta_0}^2)^{3/2}} \right) + \frac{\mu}{2} \int_\Omega \abs{\sga v}^2 J 
+\int_{\Sigma_s} \beta J \abs{v \cdot \tau}^2 + [v\cdot \N,v\cdot \N]_\ell  \\
= \int_\Omega F^1 \cdot v J + q F^2 J - \int_{\Sigma_s}  J (v \cdot \tau)F^5 \\
-  \int_{-\ell}^\ell \sigma  F^3  \p_1(v \cdot \N) + F^4 \cdot v  -  g \xi F^6 -  \sigma \frac{\p_1 \xi \p_1 F^6}{(1+\abs{\p_1 \zeta_0}^2)^{3/2}}
  - [v\cdot \N, F^6 + F^7]_\ell.
\end{multline}
\end{thm} 
\begin{proof}
We use $v$ as a test function in Lemma \ref{geometric_evolution} to see that 
\begin{multline}\label{linear_energy_1}
  \pp{v,v} - (p,\diva v)_0 + (\xi,v\cdot \N)_{1,\Sigma} + [v\cdot \N,v\cdot \N]_\ell  
 = \int_\Omega F^1 \cdot v J  \\
 - \int_{\Sigma_s} J (v\cdot \tau)F^5 
- \int_{-\ell}^\ell \sigma F^3  \p_1(v \cdot \N) + F^4 \cdot v  
  - [v\cdot \N, F^6 + F^7]_\ell.
\end{multline}
We then compute
\begin{multline}
(\xi,v\cdot \N)_{1,\Sigma} =  (\xi,\dt \xi - F^6)_{1,\Sigma} \\
 =  \dt \left(   \int_{-\ell}^\ell \frac{g}{2} \abs{\xi}^2 + \frac{\sigma}{2} \frac{\abs{\p_1 \xi}^2}{(1+\abs{\p_1 \zeta_0}^2)^{3/2}} \right) - \int_{-\ell}^\ell  g \xi F^6 +  \sigma \frac{\p_1 \xi \p_1 F^6}{(1+\abs{\p_1 \zeta_0}^2)^{3/2}}
\end{multline}
and plug into \eqref{linear_energy_1} to deduce \eqref{linear_energy_0}

\end{proof}

We now apply this to solutions to \eqref{geometric}.  First we define
\begin{equation}\label{Q_def}
 \Q(y,z) := \int_0^z  \R(y,r) dr \Rightarrow \frac{\p \Q}{\p z}(y,z) =  \R(y,z),
\end{equation}
where $\R$ is defined by \eqref{R_def}.

\begin{cor}\label{basic_energy}
 Suppose that $(u,p,\eta)$ solve \eqref{geometric}.  Let $\Q$ be given by \eqref{Q_def}.  Then
\begin{multline}\label{be_0}
 \dt \left( \int_{-\ell}^\ell \frac{g}{2} \abs{\eta}^2 + \frac{\sigma}{2} \frac{\abs{\p_1 \eta}^2}{(1+\abs{\p_1 \zeta_0}^2)^{3/2}} \right) + \frac{\mu}{2} \int_\Omega \abs{\sga u}^2 J 
+\int_{\Sigma_s} \beta J \abs{u \cdot \tau}^2 + \sum_{a=\pm 1} \kappa \abs{u\cdot \N(a\ell)}^2
\\
 =- \dt \left( \int_{-\ell}^\ell \sigma \Q(\p_1 \zeta_0,\p_1 \eta) \right) - [u\cdot \N,\swh(\dt \eta)]_\ell .
\end{multline}
\end{cor}
\begin{proof}
According to \eqref{geometric}, $v =u$, $q =p$, and $\xi = \eta$ solve \eqref{linear_geometric} with $F^i=0$ for $i\neq 3,7$ and $F^3 =   \R(\p_1 \zeta_0,\p_1 \eta)$, $F^7 = \swh(\dt \eta)$.  The  equality  \eqref{be_0} then follows directly from Theorem \ref{linear_energy} and a simple computation with $\Q$.
\end{proof}

\subsection{The pressure estimate }

The basic energy estimate does not control the pressure.  We get this control from another estimate.  Recall that 
\begin{equation}
 \Hzz =\{ u \in H^1(\Omega) \st u\cdot \nu =0 \text{ on } \Sigma_s \text{ and } u \cdot \N_0 =0 \text{  on } \Sigma\}.
\end{equation}

\begin{prop}\label{pressure_v}
Let $p \in \Lz$.  Then there exists $v \in \Hzz$ such that $\diverge v = p$ and 
\begin{equation}
 \ns{v}_1 \ls \ns{p}_0.
\end{equation}
\end{prop}
\begin{proof}
Consider the elliptic problem 
\begin{equation}
\begin{cases}
-\Delta \varphi =p &\text{in }\Omega \\
\nab \varphi \cdot \nu =0 & \text{on } \Sigma_s \\
\nab \varphi \cdot \N_0 =0 & \text{on } \Sigma. 
\end{cases}
\end{equation}
By the Neumann problem analysis in planar domains with convex corners (see for instance \cite{tice_neumann}), we know that there exists a unique solution $\varphi \in H^2(\Omega) \cap \Lz$ and that $\ns{\varphi}_2 \ls \ns{p}_0$.  Set $v=\nab \varphi$.
\end{proof}

We can parlay this into a solution of $\diva v =p$.  We need a preliminary result first.  Consider the matrix
\begin{equation}\label{M_def}
 M = K \nab \Phi = (J \A^T)^{-1}. 
\end{equation}
We view multiplication by $M$ as a linear operator.  The following lemma summarizes some of the properties of this operator.

\begin{prop}\label{M_properties}
Assume that $\eta \in H^r$ for $r >3/2$.  Let $M$ be defined by \eqref{M_def}.  The following hold.
\begin{enumerate}
 \item $M: H^0(\Omega) \to \SH0$ is a bounded linear isomorphism, and 
\begin{equation}
 \norm{M u}_{\SH0} \ls (1+ \norm{\eta}_{r-1/2}) \norm{u}_0.
\end{equation}

 \item $M : \Hz \to \SHz$ is a bounded linear isomorphism, and 
\begin{equation}
 \norm{M u}_{\SHz} \ls (1+ \norm{\eta}_{r}) \norm{u}_1.
\end{equation}

 \item $M : \Hzz \to \SHzz$ is a bounded linear isomorphism.

 \item Let $u \in H^1(\Omega)$.  Then $\diverge u =p$ if and only if $\diva(Mu)=K p$.

 \item $M : W \to \W(t)$ and $M: V \to \V(t)$ are bounded linear isomorphisms.
\end{enumerate}
\end{prop}
\begin{proof}
The boundedness of $M$ on $\SH0$ and on $\SHz$ follows directly from standard product estimates in Sobolev spaces.  This proves the first item.  To complete the proof of the second we note that 
a straightforward computation reveals that
\begin{equation}
 K\nab \Phi^T \nu = K\nu \text{ on } \{ x \in \p \Omega \;\vert \;  x_1 = \pm \ell, x_2 \ge 0  \} \text{ and } K\nab \Phi^T \nu = \nu \text{ on } \{x \in \p \Omega \; \vert \; x_2 < 0 \}.
\end{equation}
Hence, on $\Sigma_s$ we have that
\begin{equation}
 M u \cdot \nu =0 \Leftrightarrow  u \cdot (K \nab \Phi^T \nu) =0 \Leftrightarrow u \cdot \nu =0.
\end{equation}
This then proves the second item.

We also have that on $\Sigma$
\begin{equation}
 J \A \N_0 = \N \Rightarrow \N_0 = K (\A)^{-1} \N = K \nab \Phi^T \N,
\end{equation}
which implies that
\begin{equation}
 u \cdot \N_0 = u \cdot K \nab \Phi^T \N = K \nab \Phi u \cdot \N = M u \cdot \N.
\end{equation}
This then proves the third item.

To prove the fourth item we compute
\begin{equation}
 \diverge(M^{-1} v) = \p_j( J\A_{ij}v_i) = J \A_{ij} \p_j v_i = J \diva{v}.
\end{equation}
Thus if $Mu = v$ we have that
\begin{equation}
 \diverge{u} = p \text{ if and only if } \diva(Mu) = Kp.
\end{equation}

The fifth item follows easily from the previous four.

\end{proof}

We now combine Propositions \ref{pressure_v} and \ref{M_properties} to deduce an orthogonal decomposition of $\W(t)$.  Using the Riesz representation theorem we may define the operator $\Q^1_t: \Lz \to \W(t)$ via
\begin{equation}
\ip{\Q^1_t p }{w}_{\W(t)} = \int_{\Omega} p \diva{w} J. 
\end{equation}
We may estimate
\begin{equation}
 \ns{\Q^1_t p}_{\W(t)} = \int_{\Omega} p \diva{\Q^1_t p} J \ls \norm{p}_{0}\norm{\Q^1_t p}_{\W(t)} \Rightarrow \norm{\Q^1_t p}_{\W(t)} \ls \norm{p}_0.
\end{equation}
On the other hand, for $p \in \Lz$ we use Proposition \ref{pressure_v} to find $u \in \Hzz \subset W$ such that $\diverge{u} =p$ and $\ns{u}_1 \ls \ns{p}_0$.  Proposition \ref{M_properties} then implies that if we set $v = Mu \in \W(t)$ we have that $J \diva{v} = p$ and $\ns{v}_{\W(t)} \ls \ns{p}_0$.   Then 
\begin{equation}
 \int_\Omega \abs{p}^2 = \int_\Omega p \diva{v}J = \ip{\Q^1_t p }{v}_{\W(t)} \ls \norm{\Q^1_t p}_{\W(t)} \norm{v}_{\W(t)} \ls \norm{\Q^1_t p}_{\W(t)} \norm{p}_0 \Rightarrow \norm{p}_0 \ls \norm{\Q^1_t p}_{\W(t)}.
\end{equation}
Hence
\begin{equation}\label{Qt_closed}
 \norm{p}_0 \ls \norm{\Q^1_t p}_{\W(t)} \ls \norm{p}_0.
\end{equation}

We deduce from \eqref{Qt_closed} that the range of $\Q^1_t$ is closed in $\W(t)$ and hence that $\Q^1_t$ is an isomorphism from $\Lz$ to $\ran(\Q^1_t) \subseteq \W(t)$.  Now we argue as per usual (see for instance \cite{GT_lwp}) to deduce that 
\begin{equation}\label{Qt_decomp}
 (\ran(\Q^1_t))^\bot = \V(t) \text{ and hence } \W(t) = \V(t) \oplus_{\W(t)} \ran(\Q^1_t).
\end{equation}
Indeed, the inclusion $\V(t) \subseteq  (\ran(\Q^1_t))^\bot$ is trivial.  To prove the opposite inclusion we suppose $w \in (\ran(\Q^1_t))^\bot$.  Then 
\begin{equation}
 0 = \int_{\Omega} p \diva{w} J \text{ for every } p \in \Lz \Rightarrow \diva{w} J = C
\end{equation}
for some constant $C \in \Rn{}$.  However, 
\begin{equation}
 C \abs{\Omega} = \int_\Omega \diva{w} J = \int_{\Sigma_s} J w \cdot \nu + \int_\Sigma \frac{\N}{\sqrt{1+\abs{\p_1 \zeta_0}^2}}\cdot w = 0 + \int_{-\ell}^\ell \N \cdot w =0
\end{equation}
so $C=0$ and $w \in \V(t)$.

Now we use \eqref{Qt_decomp} to deduce the following.

\begin{thm}\label{pressure_lagrange}
 Suppose that $\Lambda \in (\W(t))^\ast$ and that $\Lambda(v) =0$ for all $v \in \V(t)$.  Then there exists a unique $p \in \Lz$ such that 
\begin{equation}
 \Lambda(w) = \ip{\Q^1_t p}{w}_{\W(t)} = \int_\Omega p \diva{w} J \text{ for all }w \in \W(t).
\end{equation}
Moreover, 
\begin{equation}
 \norm{p}_0 \ls \norm{\Lambda}_{(\W(t))^\ast}.
\end{equation}
\end{thm}

We can also use this to deduce the following theorem.

\begin{thm}\label{pressure_est}
If $(v,\xi)$ satisfy 
\begin{multline}
  \pp{v,w} + (\xi,w\cdot \N)_{1,\Sigma} + [v\cdot \N,w\cdot \N]_\ell  
 = \int_\Omega F^1 \cdot w J   - \int_{\Sigma_s} J (w\cdot \tau)F^5 \\
- \int_{-\ell}^\ell \sigma  F^3 \p_1  (w \cdot \N) + F^4 \cdot w  
  -  [w\cdot \N, F^6 + F^7]_\ell
\end{multline}
for all $w \in \V(t)$, then there exists a unique $q \in \Lz$ such that \eqref{weak_form} is satisfied.  Moreover, 
\begin{equation}\label{pressure_est_0}
 \norm{q}_0 \ls \norm{v}_{1} + \norm{\mathcal{F}}_{(H^1)^\ast},
\end{equation}
where $\mathcal{F} \in (H^1)^\ast$ is given by 
\begin{equation}
 \br{\mathcal{F},w} = \int_\Omega F^1 \cdot w J - \int_{\Sigma_s} J (w\cdot \tau)F^5.
\end{equation}
\end{thm}
\begin{proof}
Define $\Lambda \in (\W(t))^\ast$ via 
\begin{multline}
 \Lambda(w) = 
  -\pp{v,w} - (\xi,w\cdot \N)_{1,\Sigma} - [v\cdot \N,w\cdot \N]_\ell  
 + \int_\Omega F^1 \cdot w J  \\- \int_{\Sigma_s} J (w\cdot \tau)F^5 
- \int_{-\ell}^\ell \sigma  F^3 \p_1  (w \cdot \N) + F^4 \cdot w  
  - [w\cdot \N, F^6 + F^7]_\ell.
\end{multline}
Then $\Lambda(v) =0$ for all $v\in \V(t)$, so Theorem \ref{pressure_lagrange} yields $q$ and shows that \eqref{weak_form} is satisfied.

Using Propositions \ref{pressure_v} and \ref{M_properties} we can find $w \in \SHzz \subseteq \W(t)$ such that $J \diva w = q$.  Then
\begin{multline}
\ns{q}_0 =   -\pp{v,w}   + \int_\Omega F^1 \cdot w J  - \int_{\Sigma_s} J (w\cdot \tau)F^5  \\
\ls \norm{w}_{\W(t)} \left( \norm{v}_{1} + \norm{\mathcal{F}}_{(H^1)^\ast} \right) 
\ls \norm{q}_{0} \left( \norm{v}_{1} + \norm{\mathcal{F}}_{(H^1)^\ast} \right),
\end{multline}
which yields the estimate \eqref{pressure_est_0}.

\end{proof}

\subsection{The $\eta$ dissipation estimate  }

In what follows we write
\begin{equation}
 \oH^s(-\ell,\ell) = H^s(-\ell,\ell) \cap \oH^0(-\ell,\ell)
\end{equation}
for $s \ge 0$. Two important facts about these spaces are recorded below.

\begin{prop}\label{ohs_interp}
Let $(X,Y)_{s,q}$ denote the (real) interpolant of Banach spaces $X,Y$ with parameters $s\in (0,1)$, $q \in [1,\infty]$ (see for instance Triebel's book \cite{triebel} for definitions).  We have that 
\begin{equation}
 (\oH^0,\oH^1)_{s,2} = \oH^s
\end{equation}
and 
\begin{equation}
 ((\oH^0)^\ast,(\oH^1)^\ast)_{s,2}  = (\oH^s)^\ast.
\end{equation}
\end{prop}
\begin{proof}
 The former follows from Theorem 1.17.1/1 of \cite{triebel} and the latter follows from Theorem 1.11.2/1 of \cite{triebel}.
\end{proof}

\begin{remark}
 Here we have used real interpolation, but analogous results hold for complex interpolation.
\end{remark}

We can now use these and usual elliptic estimates to get an interpolated elliptic estimate.


\begin{thm}\label{eta_elliptic}
Suppose that $\xi \in \oH^1(-\ell,\ell)$ is the unique solution to  
\begin{equation}\label{eta_elliptic_01}
 (\xi,\theta)_{1,\Sigma} + [h,\theta]_\ell = \br{F,\theta}
\end{equation}
for all $\theta \in \oH^{1}(-\ell,\ell)$.  If $F \in  (\oH^s(-\ell,\ell))^\ast$ for $s \in [0,1]$ then $\xi \in \oH^{2-s}(-\ell,\ell)$ and 
\begin{equation}\label{eta_elliptic_0}
 \ns{\xi}_{\oH^{2-s}} \ls [h]_\ell^2 + \ns{F}_{(\oH^s)^\ast}.
\end{equation}
\end{thm}
\begin{proof}
We sketch the proof in several steps.

\emph{Step 1 - Elliptics at the endpoints with $h=0$}

Suppose for now that $h=0$.  We get the following estimates.  If $F \in (\oH^0)^\ast = \oH^0$, then $\xi \in \oH^2$ and 
\begin{equation}
 \ns{\xi}_{\oH^2} \ls  \ns{F}_{(\oH^0)^\ast}.
\end{equation}
On the other hand, if  $F \in (\oH^1)^\ast$ then we get no improvement of regularity for $\xi$, but we have the estimate 
\begin{equation}
 \ns{\xi}_{\oH^1} \ls  \ns{F}_{(\oH^1)^\ast}. 
\end{equation}

\emph{Step 2 - Elliptics  with $h=0$}

Again assume $h=0$.  We may then apply Proposition \ref{ohs_interp} and the usual the usual theory of operator interpolation to deduce that 
if $F \in (\oH^s)^\ast$, then $\xi \in \oH^{2-s}$ and 
\begin{equation}
 \ns{\xi}_{\oH^{2-s}} \ls  \ns{F}_{(\oH^s)^\ast}.
\end{equation}

\emph{Step 3 - Elliptics  with $F =0$}

Now we consider the case $F=0$.  In this case a standard argument reveals that  $\xi \in C^\infty$ and
\begin{equation}
 \ns{\xi}_{\oH^k} \le C_k [h]_\ell^2
\end{equation}
for any $k \ge 0$, where $C_k >0$ does not depend on the solution or $h$.

\emph{Step 4 - synthesis}

We now combine Steps 3 and 4 to deduce that \eqref{eta_elliptic_0} holds.

\end{proof}

The derivative operator $\p_1$ is a bounded operator from $H^1(-\ell,\ell)$ to $L^2(-\ell,\ell)$ and from $L^2(-\ell,\ell)$ to  $(H_0^1(-\ell,\ell))^\ast$, and so the usual theory of interpolation guarantees that 
\begin{equation}
\p_1 : H^{1/2}(-\ell,\ell) \to [L^2(-\ell,\ell),H_0^1(-\ell,\ell)]_{1/2}^\ast = (H_{00}^{1/2}(-\ell,\ell))^\ast.
\end{equation}
is a bounded linear operator.  The trouble is that $H_{00}^{1/2}(-\ell,\ell) \subset H^{1/2}(-\ell,\ell)$ (for a proof of this and precise definitions we refer to \cite{lions_magenes_1}), and so the previous theorem is not ideal for dealing with $F$ of the from 
\begin{equation}
 \br{F,\theta} = \int_{-\ell}^\ell f \p_1 \theta.
\end{equation}
As a result, we need the following variant.

\begin{thm}\label{eta_elliptic_var}
Suppose that $\xi \in \oH^1(-\ell,\ell)$ is the unique solution to  
\begin{equation}\label{eta_elliptic_var_01}
 (\xi,\theta)_{1,\Sigma}   =  \int_{-\ell}^\ell f \p_1 \theta
\end{equation}
for all $\theta \in \oH^{1}(-\ell,\ell)$, where $f \in H^{1/2}(-\ell,\ell)$.  Then $\xi \in H^{3/2}(-\ell,\ell)$ and 
\begin{equation}\label{eta_elliptic_var_02}
 \ns{\xi}_{3/2} \ls \ns{f}_{1/2}.
\end{equation}
\end{thm}
\begin{proof}
Using $\theta = \xi$ as a test function in \eqref{eta_elliptic_var_01} provides us with the estimate 
\begin{equation}\label{eta_elliptic_var_1}
 \norm{\xi}_1 \ls \norm{f}_0.
\end{equation}

Now let $\varphi \in C_c^\infty(-\ell,\ell)$ and let $\bar{\varphi} = \int_{-\ell}^\ell \varphi$.  Then $\varphi - \bar{\varphi} \in \oH^1(-\ell,\ell)$ and hence 
\begin{equation}
 \int_{-\ell}^\ell \sigma \frac{\p_1 \xi}{\sqrt{1+\abs{\p_1 \zeta_0}^2}} \p_1 (\varphi -\bar{\varphi})  + g \xi (\varphi -\bar{\varphi}) = \int_{-\ell}^\ell f \p_1 (\varphi -\bar{\varphi}).
\end{equation}
Since 
\begin{equation}
 \p_1 (\varphi -\bar{\varphi}) = \p_1 \varphi \text{ and } \int_{-\ell}^\ell g \xi  \bar{\varphi}  = g \bar{\varphi} \int_{-\ell}^\ell \xi =0 
\end{equation}
we then have that 
\begin{equation}
 \int_{-\ell}^\ell \sigma \frac{\p_1 \xi}{\sqrt{1+\abs{\p_1 \zeta_0}^2}} \p_1 \varphi   + g \xi \varphi  = \int_{-\ell}^\ell f \p_1 \varphi 
\end{equation}
for every $\varphi \in C_c^\infty(-\ell,\ell)$.  From this we immediately deduce that $\chi:= \sigma \p_1 \xi (1+\abs{\p_1 \zeta_0}^2)^{-1/2} -f$ is weakly differentiable with 
\begin{equation}
\p_1 \chi = g \xi \in H^1(-\ell,\ell).
\end{equation}
Thus $\chi  \in H^2(-\ell,\ell)$ and
\begin{equation}\label{eta_elliptic_var_2}
 \norm{\chi}_{2} \le \norm{\chi}_0 + \norm{\p_1 \chi}_{1}   \ls \norm{\xi}_1 + \norm{f}_0 + \norm{g \xi}_1 \ls \norm{f}_0,
\end{equation}
where in the last inequality we have used \eqref{eta_elliptic_var_1}.

Now we may estimate 
\begin{equation}
 \norm{\p_1 \xi}_{1/2} = \norm{\sqrt{1+\abs{\p_1 \zeta_0}^2} (f + \chi ) }_{1/2} \ls \norm{f+\chi}_{1/2} \ls \norm{f}_{1/2} + \norm{\chi}_{1/2} \ls \norm{f}_{1/2}.
\end{equation}
From this estimate we immediately deduce \eqref{eta_elliptic_var_02}.

\end{proof}

Now we use Theorems \ref{eta_elliptic} and \ref{eta_elliptic_var} to get the $\eta$ dissipation estimate.

\begin{thm}\label{xi_est}
Suppose that $(v,\xi)$ satisfy 
\begin{multline} 
  \pp{v,w}  + (\xi,w\cdot \N)_{1,\Sigma} + [v\cdot \N,w\cdot \N]_\ell  
 = \int_\Omega F^1 \cdot w J  \\- \int_{\Sigma_s} J (w\cdot \tau)F^5 
- \int_{-\ell}^\ell \sigma F^3 \p_1  (w \cdot \N) + F^4 \cdot w  
  - [w\cdot \N, F^6 + F^7]_\ell 
\end{multline}
for all $w \in \V(t)$.  Then for each  $\theta \in \oH^1(-\ell,\ell)$ then there exists $w[\theta] \in \V(t)$ such that the following hold:
\begin{enumerate}
 \item $w[\theta]$ depends linearly on $\theta$,
 \item $w[\theta]\cdot \N =\theta$ on $\Sigma$,
 \item we have the estimates
\begin{equation}\label{xi_est_02}
 \ns{w[\theta]}_{1} \ls \ns{\theta}_{\oH^{1/2}} \text{ and } \ns{w[\theta]}_{\W(t)} \ls \ns{\theta}_{\oH^{1}}, 
\end{equation}
 \item we have the identity
\begin{equation}\label{xi_est_00} 
 (\xi,\theta)_{1,\Sigma} + [h,\theta]_\ell = \br{\mathcal{G},\theta} - \int_{-\ell}^\ell \sigma  F^3 \p_1 \theta,
\end{equation}
where $\mathcal{G}$ and $h$ are defined as follows.  First,  $h$ is given by 
\begin{equation}
 [h,\theta]_\ell = [v\cdot \N,\theta]_\ell -  [ F^6 +F^7 ,\theta]_\ell.
\end{equation}
Second,  $\mathcal{G} \in (\oH^{1/2})^\ast$ is defined via
\begin{equation}
 \br{\mathcal{G},w[\theta]} =  - \pp{v,w[\theta]} + \br{\mathcal{F},w[\theta]},
\end{equation}
with  $\mathcal{F} \in (H^1)^\ast$ given by 
\begin{equation}
 \br{\mathcal{F},w} =    \int_\Omega F^1 \cdot w J - \int_{\Sigma_s} J (w\cdot \tau)F^5  - \int_{-\ell}^\ell F^4 \cdot w. 
\end{equation}
\end{enumerate}
Consequently,  $\xi$ satisfies 
\begin{equation}\label{xi_est_01}
  \ns{\xi}_{\oH^{3/2}} \ls  \ns{v}_1 +[v\cdot \N]_\ell^2 + \ns{\mathcal{F}}_{(H^1)^\ast} +  \ns{F_3}_{1/2} + [ F^6 +F^7 ]_\ell^2.
\end{equation}
\end{thm}
\begin{proof}
Let $\theta \in \oH^1$.  We may again employ the Neumann problem analysis (see for example \cite{tice_neumann}) to find $\varphi \in H^2(\Omega)$ such that $w = M \nab \varphi \in \V(t)$ (with $M$ as in \eqref{M_def}) such that 
 \begin{equation}
\begin{cases}
\diva w = 0 &\text{in }\Omega \\
w \cdot \N = \theta &\text{on } \Sigma \\
w\cdot \nu =0 &\text{on }\Sigma_s
\end{cases}
\end{equation}
and \eqref{xi_est_02} holds.  Let us write $w[\theta]$ to denote this function.  We then have that 
\begin{equation}
 (\xi,\theta)_{1,\Sigma} + [h,\theta]_\ell = \br{\mathcal{G},\theta}  - \int_{-\ell}^\ell \sigma  F^3 \p_1 \theta
\end{equation}
for all $\theta \in \oH^1$, which is \eqref{xi_est_00}.  We may decompose $\xi = \xi_1 + \xi_2$, where 
\begin{equation}
 (\xi_1,\theta)_{1,\Sigma} + [h,\theta]_\ell = \br{F,\theta}  
\end{equation}
and
\begin{equation}
 (\xi_2,\theta)_{1,\Sigma} =  - \int_{-\ell}^\ell \sigma  F^3 \p_1 \theta.
\end{equation}
We then apply  Theorem \ref{eta_elliptic} with $s=1/2$ to $\xi_1$ and Theorem \ref{eta_elliptic_var} to $\xi_2$ in order to arrive at  \eqref{xi_est_01}.

\end{proof}

\section{Elliptic theory for the Stokes problem }\label{sec_elliptics}

\subsection{Analysis in cones }

Consider the cone of opening angle $\omega \in (0,\pi)$ given by 
\begin{equation}\label{cone_def}
 K_\omega = \{x \in \Rn{2} \st  r>0 \text{ and } \theta \in (-\pi/2,-\pi/2 + \omega)   \},
\end{equation}
where $(r,\theta)$ are standard polar coordinates in $\Rn{2}$ (i.e. $\theta=0$ corresponds to the positive $x_1$ axis). We write
\begin{equation}
 \Gamma_- = \{ x \in \Rn{2} \st r>0 \text{ and } \theta =-\pi/2  \} \text{ and } \Gamma_+ = \{x \in \Rn{2} \st r>0 \text{ and } \theta =-\pi/2 + \omega  \}
\end{equation}
for the lower and upper boundaries of $K_\omega$.  For a given $\omega \in (0,\pi)$ we will often need to refer to the critical weight 
\begin{equation}\label{crit_wt}
 \delta_\omega := \max\{0,2-\pi/\omega\} \in [0,1). 
\end{equation}

Next we introduce a special matrix-valued function.  Suppose that $\af: K_\omega \to \Rn{2\times 2}$ is a map satisfying the following four properties.  First, $\af$ is smooth on $K_\omega$ and  $\af$ extends to a smooth function on $\bar{K}_\omega \backslash \{0\}$ and a continuous function on $\bar{K}_\omega$.  Second, $\af$ satisfies the following for all $a,b \in \mathbb{N}$:
\begin{equation}\label{frak_A_assump}
\begin{split}
& \lim_{r\to 0} \sup_{\theta \in [-\pi/2,-\pi/2 + \omega]} \abs{  (r \p_r)^a \p_\theta^b [ \af(r,\theta) \af^T(r,\theta) - I  ]    } =0  \\
& \lim_{r\to 0} \sup_{\theta \in [-\pi/2,-\pi/2 + \omega]} \abs{  (r \p_r)^a \p_\theta^b [ \af_{ij}(r,\theta)\p_j \af_{ik}(r,\theta)  ]    } =0 \text{ for  }k \in \{1,2\} \\
& \lim_{r\to 0} \sup_{\theta \in [-\pi/2,-\pi/2 + \omega]} \abs{  (r \p_r)^a \p_\theta^b [ \af(r,\theta)  - I  ]    } =0  \\
& \lim_{r\to 0}   (r \p_r)^a  [ \af(r,\theta_0)\nu   - \nu  ]     =0    \text{ for } \theta_0 =-\pi/2,-\pi/2 + \omega \\
& \lim_{r\to 0}  (r \p_r)^a  \left[ \left(\af\nu \otimes \af^T (\af \nu)^\bot + (\af\nu)^\bot \otimes \af^T (\af\nu)\right)(r,\theta_0)   - I  \right]     =0    \text{ for } \theta_0 =-\pi/2,-\pi/2 + \omega \\
\end{split}
\end{equation}
where $(r,\theta)$ denote the standard polar coordinates and $(z_1,z_2)^\bot = (z_2,-z_1)$. Third, the matrix $\af \af^T$ is uniformly elliptic on $K_\omega$.  Fourth,  $\det \af =1$ and 
\begin{equation}
\p_j( \af_{ij}) = 0 \text{ for }i=1,2.
\end{equation}

We now concern ourselves with solving the $\af-$Stokes problem in the cone $K_\omega$:
\begin{equation}\label{af_stokes_cone}
\begin{cases}
 \diverge_\af S_\af(q,v) = G^1 &\text{in } K_\omega \\
 \diverge_\af v = G^2 &\text{in } K_\omega \\
 v \cdot \af \nu = G^3_\pm &\text{on } \Gamma_\pm \\
 \mu \sg_\af v \af \nu \cdot (\af \nu)^\bot = G^4_\pm &\text{on } \Gamma_\pm,
\end{cases}
\end{equation}
where here the operators $\diverge_\af$ and $S_\af$ are defined in the same way as $\diva$ and $S_\A$.  Note that in the case that $\af = I_{2\times 2}$, the system  \eqref{af_stokes_cone} is the standard Stokes problem 
\begin{equation}\label{stokes_cone}
\begin{cases}
 \diverge S(q,v) = G^1 &\text{in } K_\omega \\
 \diverge v = G^2 &\text{in } K_\omega \\
 v \cdot  \nu = G^3_\pm &\text{on } \Gamma_\pm \\
 \mu \sg v  \nu \cdot  \tau = G^4_\pm &\text{on } \Gamma_\pm.
\end{cases}
\end{equation}
We note that the assumptions in \eqref{frak_A_assump} are needed to show that the operators appearing in \eqref{af_stokes_cone} behave like the operators in \eqref{stokes_cone} near $0 \in \bar{K}_\omega$.

Following \cite{kmr_1}, for $k \in \mathbb{N}$ and $\delta >0$ we  define the weighted Sobolev spaces
\begin{equation}
 W^k_\delta(K_\omega) = \{ u \st \norm{u}_{W^k_\delta} < \infty\},
\end{equation}
where 
\begin{equation}
 \ns{u}_{W^k_\delta} = \sum_{\abs{\alpha} \le k} \int_{K_\omega} \abs{x}^{2\delta} \abs{\p^\alpha u(x)}^2 dx.
\end{equation}
We then define the trace spaces $W^{k-1/2}_\delta(\Gamma_\pm)$ as in \cite{kmr_1}.

\begin{thm}\label{cone_solve}
Let $\omega \in (0,\pi)$, $\delta_\omega \in [0,1)$ be given by \eqref{crit_wt}, and $\delta \in (\delta_\omega,1)$.  Suppose that $\af$ satisfies the four properties stated above.  Assume that the data $G^1,G^2, G^3_\pm,G^4_\pm$ for the problem \eqref{stokes_cone} satisfy 
\begin{equation}\label{cone_solve_01}
G^1 \in W^0_{\delta}(K_\omega), G^2 \in W^1_{\delta}(K_\omega), G^3_\pm \in W^{3/2}_{\delta}(\Gamma_\pm), G^4_\pm \in W^{1/2}_{\delta}(\Gamma_\pm)
\end{equation}
as well as the compatibility condition
\begin{equation}\label{cone_solve_cc}
 \int_{K_\omega} G^2  = \int_{\Gamma_+} G^3_+  + \int_{\Gamma_-} G^3_- .
\end{equation}
Suppose that $(v,q) \in H^1(K_\omega) \times H^0(K_\omega)$ satisfy  $\diverge_\af v = G^2$, $v\cdot \af \nu = G^3_\pm$ on $\Gamma_\pm$, and  
\begin{equation}
 \int_{K_\omega} \frac{\mu}{2} \sg_\af v : \sg_\af w  - q \diverge_\af w   = \int_{K_\omega} G^1  \cdot w      + \int_{\Gamma_+} \G^4_+ w\cdot \frac{(\af \nu)^\bot}{\abs{\af \nu}} + \int_{\Gamma_-} \G^4_-  w\cdot \frac{(\af \nu)^\bot}{\abs{\af \nu}} 
\end{equation}
for all $w \in \{ w \in H^1(K_\omega) \st w\cdot (\af \nu) =0 \text{ on } \Gamma_\pm\}$.   Finally, suppose that $v,q$ and all of the data $G^i$ are supported in $\bar{K}_\omega \cap B[0,1]$.  Then $D^2 v, \nab q \in W^0_{\delta}(K_\omega)$ and 
\begin{multline}\label{cone_solve_02}
 \ns{D^2 v}_{W^0_{\delta}} + \ns{\nab q}_{W^0_{\delta}} \ls  \ns{G^1}_{W^0_{\delta} } +  \ns{G^2}_{W^1_{\delta}} +  \ns{G^3_-}_{W^{3/2}_{\delta} } +  \ns{G^3_+}_{W^{3/2}_{\delta}} 
 +  \ns{G^4_-}_{W^{1/2}_{\delta}} +  \ns{G^4_+}_{W^{1/2}_{\delta}}.
\end{multline}
\end{thm}
\begin{proof}
In the case $\af = I$ the result is essentially proved in Theorem 9.4.5 in \cite{kmr_3} except that there the results are stated in a three-dimensional dihedral angle.  However, the analysis begins with the problem in two dimensions and is easily adaptable to the $\af$-Stokes problem \eqref{af_stokes_cone}.  The key to the proof is an application of Theorem 8.2.1 of \cite{kmr_1}, which characterizes the solvability of elliptic systems in terms of the eigenvalues of an associated operator pencil.  The assumptions on $\af$, in particular \eqref{frak_A_assump}, guarantee that the ``leading operators'' (in the terminology of \cite{kmr_1}) associated to \eqref{af_stokes_cone} are exactly the operators appearing in \eqref{stokes_cone}, and hence the problems \eqref{af_stokes_cone} and \eqref{stokes_cone} give rise to the same associated operator pencil.  The eigenvalues of the pencil associated to \eqref{stokes_cone} may be found in the ``G-G eigenvalue computations'' of \cite{orlt_sandig} (with $\chi_1 = \chi_2 = \pi/2$).  Indeed, the latter guarantees that the strip 
\begin{equation}
 \{ z \in \mathbb{C} \st 0 \le \Re(z) \le 1- \delta\}
\end{equation}
contains no eigenvalues of the operator pencil associated to the Stokes problem \eqref{stokes_cone} in the cone $K_\omega$, which are $\pm 1 + n \pi/\omega$ for $n \in \mathbb{Z}$.  Thus we may use Theorem 8.2.1 of \cite{kmr_1} on \eqref{af_stokes_cone} and then argue as in Theorem 9.4.5 in \cite{kmr_3}.  

\end{proof}

\subsection{The Stokes problem in $\Omega$}

We now turn to the study of the Stokes problem in $\Omega$:
\begin{equation}\label{stokes_omega}
\begin{cases}
\diverge S(q,v)  = G^1 &\text{in } \Omega \\
 \diverge v = G^2 &\text{in } \Omega \\
 v \cdot \nu = G^3_+ &\text{on } \Sigma \\
 \mu \sg v \nu \cdot \tau = G^4_+ &\text{on } \Sigma\\
 v \cdot \nu = G^3_- &\text{on } \Sigma_s \\
 \mu \sg v \nu \cdot \tau = G^4_- &\text{on } \Sigma_s.
\end{cases}
\end{equation}
In what follows we will work with the spaces $W^k_\delta(\Omega)$, $W^{k-1/2}_\delta(\p \Omega)$, and $\oW^k_\delta(\Omega)$  as defined in Appendix \ref{app_weight}.

Next we  define  $\mathfrak{X}_\delta$ for $0 < \delta < 1$ to be the space of $6-$tuples
\begin{equation}
 (G^1,G^2,G^3_+,G^3_-,G^4_+,G^4_-) \in W^{0}_\delta(\Omega) \times W^1_\delta(\Omega) \times W^{3/2}_\delta(\Sigma  )\times W^{3/2}_\delta(\Sigma_s  ) \times  W^{1/2}_\delta(\Sigma) \times W^{1/2}_\delta(\Sigma_s)             
\end{equation}
such that 
\begin{equation}
 \int_{\Omega} G^2 = \int_{\Sigma} G^3_+ + \int_{\Sigma_s} G^3_-.
\end{equation}

We will now formulate a definition of weak solution to \eqref{stokes_omega} for data in this space.

\begin{dfn}
Assume that $(G^1,G^2,G^3_+,G^3_-,G^4_+,G^4_-)  \in \mathfrak{X}_\delta$ for some $0 < \delta < 1$.  We say that a pair $(v,q) \in H^1(\Omega) \times \oH^0(\Omega)$ such that  $\diverge v = G^2$, $v\cdot \nu = G^3$ on $\p \Omega$, and  
\begin{equation}\label{stokes_om_weak_form}
 \int_{\Omega} \frac{\mu}{2} \sg v : \sg w - q \diverge w = \int_{\Omega} G^1  \cdot w   + \int_{\Sigma} G^4_+ (w\cdot \tau) + \int_{\Sigma_s} G^4_- (w \cdot \tau) 
\end{equation}
for all $w \in \{ w \in H^1(\Omega) \st w\cdot \nu =0 \text{ on } \p \Omega\}$ is a weak solution to \eqref{stokes_omega}.  Note that the integrals on the right side of \eqref{stokes_om_weak_form} are well-defined by virtue of \eqref{hardy_embed} and \eqref{hardy_embed_2}.  Also $G^2 \in H^0(\Omega)$ and $G^3 \in H^{1/2}(\p \Omega)$ for the same reason.
\end{dfn}

We have the following weak existence result.  

\begin{thm}\label{stokes_om_weak}
Let $(G^1,G^2,G^3_+,G^3_-,G^4_+,G^4_-)   \in \mathfrak{X}_\delta$ for some $0 < \delta < 1$.  Then there exist a unique pair $(v,q) \in H^1(\Omega) \times \oH^0(\Omega)$ that is a weak solution to \eqref{stokes_omega}.   Moreover, 
\begin{equation}\label{stokes_weak_0}
 \ns{v}_1 + \ns{q}_0 \ls \ns{G^1}_{W^0_\delta} + \ns{G^2}_{W^1_\delta} + \ns{G^3}_{W^{3/2}_\delta} + \ns{G^4}_{W^{1/2}_\delta}.
\end{equation}

\end{thm}
\begin{proof}
We first use \eqref{hardy_embed} to see that $G^2 \in H^0(\Omega)$ and $G^3 \in H^{1/2}(\p \Omega)$.  Choose $\bar{v} \in W^2_\delta(\Omega)$ such that $\bar{v}\vert_{\Sigma} = G^3_+$ and $\bar{v}\vert_{\Sigma_s} = G^3_-$ with $\norm{\bar{v}}_{W^2_\delta} \ls \norm{G_3}_{W^{3/2}_\delta}$.  Using, for instance, the analysis in \cite{tice_neumann}, we may find $\varphi \in H^2(\Omega)$ solving
\begin{equation}
\begin{cases}
 -\Delta \varphi = G^2 - \diverge \bar{v} &\text{in }\Omega \\
 \nab \varphi \cdot \nu = G^3 &\text{on } \p \Omega
\end{cases}
\end{equation}
with 
\begin{equation}
\ns{\varphi}_{2} \ls \ns{G^2}_0 + \ns{G^3}_{1/2}   \ls \ns{G^2}_{W^1_\delta} + \ns{G^3}_{W^{3/2}_\delta}.
\end{equation}

Next we find $u \in \Hzz$ with $\diverge u =0$ such that 
\begin{equation}
 \int_{\Omega} \frac{\mu}{2} \sg u : \sg w  = \int_{\Omega} G^1  \cdot w  - \frac{\mu}{2} \sg (\nab \varphi + \bar{v}) : \sg w  + \int_{\p \Omega} G^4 (w\cdot \tau)  
\end{equation}
for all $w \in \Hzz$ such that $\diverge w =0$.  This is readily done with the Riesz representation theorem, and we find that 
\begin{equation}
 \ns{u}_1 \ls \ns{G^1}_{W^0_\delta} + \ns{G^2}_{W^1_\delta} + \ns{G^3}_{W^{3/2}_\delta} + \ns{G^4}_{W^{1/2}_\delta}.
\end{equation}

Finally, we use Theorem \ref{pressure_lagrange} (with $\eta =0$ so that $\A = I$, etc) to find $q \in \oH^0(\Omega)$ such that \eqref{stokes_om_weak_form} holds with $v = u + \bar{v}+ \nab \varphi$.  We then easily deduce the estimate \eqref{stokes_weak_0}, which in turn implies the uniqueness claim.
\end{proof}

Next we turn to the issue of second-order regularity.  To develop this theory we will first need the following technical result, which constructs a special diffeomorphism.

\begin{prop}\label{wedge_diffeo}
Let $K_\omega \subset \Rn{2}$ be the cone of opening $\omega \in (0,\pi)$ defined by \eqref{cone_def}, where $\omega$ is the angle of $\Omega$ near the corners, and let $0 < r < \min\{\ell,\zeta_0(-\ell)/2\}$.  Then there exists a smooth diffeomorphism $\Psi : K_\omega \to \Psi(K_\omega) \subset \Rn{2}$ satisfying the following properties.  
\begin{enumerate}
 \item $\Psi$ is smooth up to $\bar{K}_\omega$.     
 \item $\Gamma_- = \Psi^{-1}(\{x \in \Rn{2} \st x_1=-\ell, x_2 < \zeta_0(-\ell)\})$.
 \item We have that  $\Psi^{-1}( \Sigma \cap B((-\ell,\zeta_0(\ell)),r)) \subseteq \Gamma_+ \cap B(0,R)$ and  $\Psi^{-1}( \Omega \cap B((-\ell,\zeta_0(\ell)),r) ) \subseteq K_\omega \cap B(0,R)$ for $R = \sqrt{2r^2 + 2r^4\ns{\zeta_0}_{C^2}}$.

 \item The matrix function $\af(x) = (D\Psi(x))^{-T}$ is smooth on $\bar{K}_\omega$, and all its derivatives are bounded.  Moreover, $\af$ satisfies the four properties listed near \eqref{frak_A_assump}.
\end{enumerate}
\end{prop}
\begin{proof}

Let $\chi \in C^\infty(\Rn{})$ be such that $\chi(s) =1$ for $s \le r$ and $\chi(s) =0$ for $s \ge 2r$.  Let $\alpha = \zeta_0'(-\ell)$, which is related to $\omega$ via $-\text{cotan}(\omega) = \alpha$.  Define  $\zeta : [0,\infty) \to \Rn{}$  by
\begin{equation}
 \zeta(s) = \chi(s)  \zeta_0(-\ell +s)  + (1-\chi(s)) \alpha s, 
\end{equation}
which is  well-defined for all $s \in (0,\infty)$ since $2r < 2\ell$ and hence $\zeta_0(-\ell+s)$ is defined on the support of $\chi$.   It's easy to see that $\zeta$ is smooth, $\zeta(0) = \zeta_0(-\ell)$, and  $\zeta'(0) = \alpha$.  Also, $\zeta(s) - \alpha s$ is compactly supported in $[0,\infty)$.  We also define the open set 
\begin{equation}
 G_\zeta = \{x \in \Rn{2} \st x_1 >-\ell \text{ and } x_2 < \zeta(x_1)\}
\end{equation}
and note that 
\begin{equation}\label{wedge_diffeo_1}
 G_\zeta \cap B((-\ell,\zeta_0(-\ell)),r) =  \Omega \cap B((-\ell,\zeta_0(-\ell)),r)
\end{equation}
since $\zeta(s) = \zeta_0(-\ell+s)$ for $s \in [0, r]$.

Next we define the map $\Psi : K_\omega \to \Rn{2}$ via 
\begin{equation}
 \Psi(x) = (x_1 - \ell,   x_2 -\alpha x_1  + \zeta(x_1)).
\end{equation}
It is a trivial matter to see that $\Psi$ is smooth on $\bar{K}_\omega$ and that $\Psi$ is a smooth diffeomorphism from $K_\omega$ to $\G_\zeta$ with inverse given by 
\begin{equation}\label{wedge_diffeo_2}
 \Psi^{-1}(y_1,y_2) = (y_1+\ell, y_2 - \zeta(y_1+\ell) + \alpha (y_1+\ell)).
\end{equation}
This proves the first item, and the second item follows trivially.

To prove the third item we first note that  $\Psi(\Gamma_+)  = \{x \in \Rn{2} \st x_1 >0, x_2 = \zeta(x_1))\}$.  From \eqref{wedge_diffeo_1} and \eqref{wedge_diffeo_2} we find that if $y \in \Omega \cap B((-\ell,\zeta_0(-\ell)),r)$ then 
\begin{equation}
 \abs{\Psi^{-1}(y)}^2\le (y_1+\ell)^2 + 2(y_2 - \zeta_0(-\ell))^2 + 2[\zeta_0(-\ell) - \zeta_0(y_1) + \alpha(y+\ell)  ]^2 
 \le 2 r^2 + 2 r^4 \ns{\zeta_0}_{C^2}.
\end{equation}
A similar calculation works for $y \in \Sigma \cap B((-\ell,\zeta_0(-\ell)),r)$, completing the proof of the third item.

We now turn to the proof of the fourth item.  The matrix $\af(x) = (D\Psi(x))^{-T}$ is given by 
\begin{equation}\label{wedge_diffeo_3}
 \af(x) = 
\begin{pmatrix}
 1 & \alpha - \zeta'(x_1) \\
 0 & 1 
\end{pmatrix}.
\end{equation}  
From this we easily deduce that $\af$ is smooth with derivatives of all order bounded in $\bar{K}_\omega$.  The equality \eqref{wedge_diffeo_3} implies that $\af$ satisfies \eqref{frak_A_assump}.  The fact that $\alpha - \zeta'(x_1)$ is compactly supported in $(0,\infty)$ then implies that $\af \af^T$ is uniformly elliptic; indeed, it is easily verified that 
\begin{equation}
 (\af(x) \af^T(x))_{ij} \xi_i \xi_j \ge \gamma \abs{\xi}^2 
\end{equation}
for all $x \in \bar{K}_\omega$, where
\begin{equation}
\gamma = 1 + \frac{\norm{\alpha - \zeta'}_{L^\infty}^2 -\sqrt{4 \norm{\alpha - \zeta'}_{L^\infty}^2 + \norm{\alpha - \zeta'}_{L^\infty}^4} }{2} > 0.
\end{equation}
Finally, we note that $\p_j(\af_{ij})=0$ for $i=1,2$, which follows by direct computation.  This completes the proof of the fourth item.

\end{proof}

We may now proceed to the proof of second-order regularity.

\begin{thm}\label{stokes_om_reg}
Let $\omega \in (0,\pi)$  be the angle formed by $\zeta_0$ at the corners of $\Omega$,  $\delta_\omega \in [0,1)$ be given by \eqref{crit_wt}, and $\delta \in (\delta_\omega,1)$.  Let $(G^1,G^2,G^3_+,G^3_-,G^4_+,G^4_-)   \in \mathfrak{X}_{\delta}$, and let $(v,q) \in H^1(\Omega) \times \oH^0(\Omega)$ be the weak solution to \eqref{stokes_omega} constructed in Theorem \ref{stokes_om_weak}.   Then $v \in W^2_{\delta}(\Omega)$, $q \in W^1_{\delta}(\Omega)$, and 
\begin{equation}\label{stokes_om_reg_0}
 \ns{v}_{W^2_{\delta}} + \ns{q}_{\oW^1_{\delta}} \ls  \ns{G^1}_{W^0_{\delta} } +  \ns{G^2}_{W^1_{\delta}} +  \ns{G^3}_{W^{3/2}_{\delta} }  
 +  \ns{G^4}_{W^{1/2}_{\delta}} .
\end{equation}
\end{thm}
\begin{proof}
For the sake of brevity we will only sketch the proof.  The omitted details may be filled in readily using standard argument.

\emph{Step 1 -- Estimates away from the corners }

Away from the corners we know that $\p \Omega$ is $C^2$, so we may apply the standard elliptic regularity theory (see for example \cite{adn_2}) for the Stokes problem with boundary conditions as in \eqref{stokes_omega} to deduce that if $V \subset \Omega$ is an open set with a $C^2$ boundary whose boundary agrees with $\p \Omega$ except near the corners, then $(v,q) \in H^2(V) \times H^1(V)$ and
\begin{equation}\label{stokes_om_reg_1}
\ns{v}_{H^2(V)} + \ns{q}_{H^1(V)} \le C(V) \left(  \ns{G^1}_{W^0_\delta} + \ns{G^2}_{W^1_\delta} + \ns{G^3}_{W^{3/2}_\delta} + \ns{G^4}_{W^{1/2}_\delta} \right).
\end{equation}
Here we have used the fact that $V$ avoids the corners to trivially estimate 
\begin{multline}
 \ns{G^1}_{H^0(W) } + \ns{G^2}_{H^1(W)} + \ns{G^3}_{H^{3/2}(\bar{W} \cap \p \Omega) } + \ns{G^4}_{H^{1/2}(\bar{W} \cap \p \Omega) } 
\\
\ls   \ns{G^1}_{W^0_\delta} + \ns{G^2}_{W^1_\delta} + \ns{G^3}_{W^{3/2}_\delta} + \ns{G^4}_{W^{1/2}_\delta},
\end{multline}
where $V \subset W \subset \Omega$ is another open set that avoids the corners of $\Omega$.

\emph{Step 2 -- Estimates near the corners }

The key step is to get weighted estimates for the solution near the corners of the domain.  To this end we introduce a small parameter 
\begin{equation}\label{stokes_om_reg_2}
 0 < r < \min\left\{ \ell, \frac{\zeta_0(-\ell)}{2}, \frac{-1 + \sqrt{1 + 2\ns{\zeta_0}_{C^2}} }{2 \ns{\zeta_0}_{C^2} }   \right\}
\end{equation}
and consider $U_r = \Omega \cap B((-\ell,\zeta_0(- \ell)),r)$.   We choose a cutoff function $\psi \in C^\infty_c(B((-\ell,\zeta_0(- \ell)),r))$ such that $\psi\ge 0$ and $\psi  =1$ on $B((-\ell,\zeta_0(- \ell)),r/2)$.  By using $\psi v$ as a test function in the weak formulation and integrating by parts (which is justified by Step 1 since $\nab \psi$ is supported away from the corner) we find that $(\tilde{v} ,\tilde{q}) =   (v \psi, q \psi)$ is a weak solution to \eqref{stokes_omega} with $G^i$ replaced by $\tilde{G}^i$, for 
\begin{equation}
\begin{split}
\tilde{G}^1 & = \psi G^1 -  \mu \sg v \nab \psi - \mu \diverge(v \otimes \nab \psi + \nab \psi \otimes v)  + q \nab \psi \\
\tilde{G}^2 & = \psi G^2 + v\cdot \nab \psi \\
\tilde{G}^3 & = \psi G^3 \\
\tilde{G}^4 & = \psi G^4 + \mu (v \otimes \nab \psi + \nab \psi \otimes v )\nu \cdot \tau.
\end{split}
\end{equation}
It's clear that $(\tilde{G}^1,\tilde{G}^2,\tilde{G}^3,\tilde{G}^4)  \in \mathfrak{X}_{\delta}$ and that $(\tilde{v} ,\tilde{q}) \in H^1(\Omega) \times H^0(\Omega)$.

Next we note that because of the assumption \eqref{stokes_om_reg_2} we know that  $\bar{U}_r \cap \p \Omega$ is actually smooth away from the corner point $(-\ell,\zeta_0(-\ell))$; indeed, the upper boundary is the graph of $\zeta_0$, which is smooth, and the side boundary is a straight line.   We then employ the diffeomorphism $\Psi^{-1}$ constructed in Proposition \ref{wedge_diffeo} to map $U_r$ to $\Psi^{-1}(U_r) \subseteq K_\omega$, where $K_\omega$ is a cone of opening angle $\omega$.  


Let $(w,\theta)$ and $\G^i$  denote the composition of $(\tilde{v} ,\tilde{q})$ and $\tilde{G}^i$, respectively, with $\Psi^{-1}$.  It is then a simple matter to verify that $(w,\theta) \in H^1(K_\omega) \times H^0(K_\omega)$ and that 
\begin{equation}
\G^1 \in W^0_{\delta}(K_\omega), \G^2 \in W^1_{\delta}(K_\omega), \G^3_\pm \in W^{3/2}_{\delta}(\Gamma_\pm), \G^4_\pm \in W^{1/2}_{\delta}(\Gamma_\pm)
\end{equation}
where $\Gamma_\pm$ denote the top and bottom sides of the cone $K_\omega$.  Moreover, $(w,\theta)$ and the $\G^i$ are all  supported in $\bar{K}_\omega \cap B[0,1]$  due to the third item of Proposition \ref{wedge_diffeo} since \eqref{stokes_om_reg_2} guarantees that $2r^2 + 2r^4 \ns{\zeta_0}_{C^2} \le 1$.

Next we use the diffeomorphism to change variables in the  weak formulation to derive a new identities for $(w,\theta)$: $\diverge_\af w = \G^2$ in $K_\omega$, $w \cdot \af \nu = \G^3 \abs{\af \nu}$ on $\Gamma_\pm$, and
\begin{equation}
 \int_{K_\omega} \frac{\mu}{2} \sg_\af w : \sg_\af \Upsilon - \theta \diverge_\af \Upsilon = \int_{K_\omega} \G^1  \cdot \Upsilon   + \int_{\Gamma_+} \G^4_+ \Upsilon \cdot \frac{(\af \nu)^\bot}{\abs{\af \nu}} + \int_{\Gamma_-} \G^4_-  \Upsilon \cdot \frac{(\af \nu)^\bot}{\abs{\af \nu}} 
\end{equation}
for all $\Upsilon \in H^1(K_\omega)$ such that $\Upsilon\cdot (\af \nu) =0$ on $\Gamma_\pm$.  This means that $(w,\theta)$ is a weak solution to the problem \eqref{af_stokes_cone} with $G^1,G^2$ replaced by $\G^1,\G^2$ and $G^3_\pm,G^4_\pm$ replaced by $\G^3_\pm\abs{\af \nu}  ,\G^4_\pm\abs{\af\nu}$.  The properties of $\Psi$ given in Proposition \ref{wedge_diffeo} guarantee that Theorem \ref{cone_solve} is applicable and we then arrive at the inclusion $(w,\theta) \in W^2_{\delta} \times W^1_{\delta}$ and the estimate
\begin{equation}\label{stokes_om_reg_3}
\ns{w}_{W^2_{\delta}} + \ns{\theta}_{W^1_{\delta}}  \ls   \ns{\G^1}_{W^0_{\delta} } +  \ns{\G^2}_{W^1_{\delta}} +  \ns{\G^3_-}_{W^{3/2}_{\delta} } +  \ns{\G^3_+}_{W^{3/2}_{\delta}} 
 +  \ns{\G^4_-}_{W^{1/2}_{\delta}} +  \ns{\G^4_+}_{W^{1/2}_{\delta}} .
\end{equation}
Upon changing coordinates back to $\Omega$ we then find that 
\begin{multline}\label{stokes_om_reg_4}
\ns{\tilde{v}}_{W^2_{\delta}(U_r) } + \ns{\tilde{q}}_{W^1_{\delta}( U_r) }   \ls   \ns{\tilde{G}^1}_{W^0_{\delta} } +  \ns{\tilde{G}^2}_{W^1_{\delta}} +  \ns{\tilde{G}^3}_{W^{3/2}_{\delta} } +  \ns{\tilde{G}^4}_{W^{1/2}_{\delta}} \\
\ls 
\ns{G^1}_{W^0_{\delta} } +  \ns{G^2}_{W^1_{\delta}} +  \ns{G^3}_{W^{3/2}_{\delta} }  + \ns{G^4}_{W^{1/2}_{\delta}} + \ns{u}_{H^1} + \ns{p}_{H^0}.
\end{multline}

A similar argument provides us with an estimate analogous to \eqref{stokes_om_reg_4} near the right corner of $\Omega$, namely the point $(\ell,\zeta_0(\ell))$.  In this case we must employ a reflection of $\Omega$ across the $x_2$ axis in order to use the diffeomorphism from Proposition \ref{wedge_diffeo}, but this does not change any of the essential properties of the diffeomorphism, and so the analysis proceeds as above.  Writing $U_r^\pm$ for the neighborhoods of the left corner $(-)$ and the right corner $(+)$, noting that our cutoff functions are unity on $U_{r/2}$, and employing the estimate \eqref{stokes_weak_0}, we then find that 
\begin{multline}\label{stokes_om_reg_5}
\ns{v}_{W^2_{\delta}(U^+_{r/2}) } + \ns{q}_{W^1_{\delta}( U^-_{r/2}) }    +
\ns{v}_{W^2_{\delta}(U^+_{r/2}) } + \ns{q}_{W^1_{\delta}( U^-_{r/2}) }    
\\
\le C(r) \left(  
\ns{G^1}_{W^0_{\delta} } +  \ns{G^2}_{W^1_{\delta}} +  \ns{G^3}_{W^{3/2}_{\delta} }  + \ns{G^4}_{W^{1/2}_{\delta}}\right).
\end{multline}

\emph{Step 3 -- Synthesis }

To conclude we simply sum \eqref{stokes_om_reg_1} and \eqref{stokes_om_reg_5} with an appropriate choice of $V$ and $r$  to deduce that \eqref{stokes_om_reg_0} holds.

\end{proof}

In what follows it will be useful to rephrase Theorem \ref{stokes_om_reg} as follows.  For $0 < \delta < 1$ we define the operator 
\begin{equation}\label{stokes_om_iso_def1}
 T_\delta : W^2_\delta(\Omega) \times  \oW^1_\delta(\Omega) \to \mathfrak{X}_\delta
\end{equation}
via
\begin{equation}\label{stokes_om_iso_def2}
 T_\delta(v,q) = (\diverge S(q,v), \diverge v, v\cdot n \vert_{\Sigma},v\cdot n \vert_{\Sigma_s}, \mu \sg v n \cdot \tau  \vert_{\Sigma}, \mu \sg v n \cdot \tau  \vert_{\Sigma_s}).  
\end{equation}
We may then deduce the following from Theorems \ref{stokes_om_weak} and \ref{stokes_om_reg}.

\begin{cor}\label{stokes_om_iso}
Let  $\delta \in (\delta_\omega,1)$ be as in Theorem \ref{stokes_om_reg}.  Then the operator $T_{\delta}$ defined by \eqref{stokes_om_iso_def1} and \eqref{stokes_om_iso_def2} is an isomorphism.
\end{cor}

\subsection{The $\A$-Stokes problem in $\Omega$}


We now assume that $\eta \in W^{5/2}_\delta$ is a given function with $\delta \in (0,1)$, which in turn determines $\A, J$, etc, and we consider the problem
\begin{equation}\label{A_stokes}
\begin{cases}
\diva S_\A(q,v) = G^1 & \text{in }\Omega \\
\diva v = G^2 & \text{in } \Omega \\
v\cdot \N = G^3_+ &\text{on } \Sigma \\
\mu \sg_\A v \N \cdot \mathcal{T} = G^4_+ &\text{on } \Sigma \\
v\cdot \nu = G^3_- &\text{on } \Sigma_s \\
\mu \sg_\A v \nu \cdot \tau = G^4_- &\text{on } \Sigma_s.
\end{cases}
\end{equation}
Note here that $\N = N - \p_1 \eta e_1$ for $N = -\p_1\zeta_0 e_1 + e_2$ an outward normal vector on $\Sigma$ and $\mathcal{T} = T + \p_1 \eta e_2$ for $T = e_1 + \p_1 \zeta_0 e_2$ the associated  tangent vector.

We now show that under a smallness assumption on $\eta$, the problem \eqref{A_stokes} is solvable in weighted Sobolev spaces.  We begin by introducing the operator 
\begin{equation}\label{A_stokes_om_iso_def1}
 T_\delta[\eta] : W^2_\delta(\Omega) \times  \oW^1_\delta(\Omega) \to \mathfrak{X}_\delta
\end{equation}
given by 
\begin{equation}\label{A_stokes_om_iso_def2}
 T_\delta[\eta](v,q) = (\diva S_\A(q,v), \diva v, v\cdot \N \vert_{\Sigma}, v\cdot \nu \vert_{\Sigma_s}, \mu \sg v \N \cdot \mathcal{T}  \vert_{\Sigma}, \mu \sg v \nu \cdot \tau  \vert_{\Sigma_s}).  
\end{equation}

\begin{prop}\label{A_stokes_T_well_def}
Suppose that  $\eta \in W^{5/2}_\delta$ is a given function that determines $\A, J$, etc.  Then the map $T_\delta[\eta]$ defined above is well-defined and bounded.
\end{prop}
\begin{proof}
Proposition \ref{weighted_embed} implies that 
\begin{equation}
  W^2_\delta(\Omega) \hookrightarrow W^{1,r}(\Omega)
\text{ for }
 1 \le r < \frac{2}{\delta}.
\end{equation} 
We similarly find that $\eta \in H^{s+1/2}$ for each $1 < s < \min\{\pi/\omega,2\}$. This, the usual Sobolev embeddings, the weighted Sobolev embeddings of Appendix \ref{app_weight}, the product estimates of Appendix \ref{app_prods}, and trace theory then imply that the map $T_\delta[\eta]$ is  well-defined from  $W^2_\delta(\Omega) \times  W^1_\delta(\Omega)$ to  $\mathfrak{X}_\delta$.  
\end{proof}

In fact, the map $T_\delta[\eta]$ is an isomorphism for some values of $\delta$ under a smallness assumption on $\eta$.

\begin{thm} \label{A_stokes_om_iso}
Let $\delta \in (\delta_\omega,1)$ be as in Theorem \ref{stokes_om_reg}.   There exists a $\gamma >0$ such that if $\ns{\eta}_{W^{5/2}_{\delta} } < \gamma$, then  the operator $T_{\delta}[\eta]$ defined by \eqref{A_stokes_om_iso_def1} and \eqref{A_stokes_om_iso_def2} is an isomorphism.
\end{thm}
\begin{proof}
Assume initially that $\gamma < 1$ is as small as in Lemma \ref{eta_small}.  We can rewrite \eqref{A_stokes} as 
\begin{equation}\label{A_stokes_om_iso_1}
 T_\delta(v,q) = (\G^1(v,q), \G^2(v), \G^3_+(v), \G^4_+(v), \G^3_-, \G^4_-(v)) =: \G(v,q),
\end{equation}
where $T_\delta$ is defined by \eqref{stokes_om_iso} and
\begin{equation}
\begin{split}
\G^1(v,q) &= G^1 + \diverge_{I-A} S_{\A}(q,v) - \diverge \mu \sg_{I-\A}(v)  \\
\G^2(v) &= G^2 + \diverge_{I-\A} v  \\
\G^3_+(v)  &= (1+(\p_1\zeta_0)^2)^{-1/2}[G^3_- + \p_1 \eta v_1  ]    \\
\G^4_+(v) &= (1+(\p_1\zeta_0)^2)^{-1} [G^4_+ + \mu \sg_{I-\A} v N\cdot T -\mu \p_1 \eta (\sg_\A v N \cdot e_2 - \sg_\A v e_1 \cdot T  ) -\mu (\p_1 \eta)^2 \sg_\A v e_1 \cdot e_2    ] \\
\G^3_- & = G^3_- \\
\G^4_- & = G^4_- + \mu \sg_{I-\A} v \nu \cdot \tau.
\end{split}
\end{equation} 
A variant of the argument used in Proposition \ref{A_stokes_T_well_def} shows that $\G:  W^2_\delta(\Omega) \times  \oW^1_\delta(\Omega) \to \mathfrak{X}_\delta$ and that we have the estimates
\begin{equation}\label{A_stokes_om_iso_2}
\begin{split}
&\norm{\G(v,q) }_{\mathfrak{X}_\delta} \le P(\norm{\eta}_{W^{5/2}_\delta} ) \left( \norm{(G^1,G^2,G^3_+,G^3_-,G^4_+,G^4_-)}_{\mathfrak{X}_\delta } + \norm{v}_{W^2_\delta} + \norm{q}_{\oW^1_\delta}  \right) 
\\
&\norm{\G(v_1,q_1)  - \G(v_2,q_2)}_{\mathfrak{X}_\delta} \le P(\norm{\eta}_{W^{5/2}_\delta} ) \left( \norm{v_1 - v_2}_{W^2_\delta} + \norm{q_1 - q_2 }_{\oW^1_\delta}   \right),
\end{split}
\end{equation}
where $P$ is a polynomial with non-negative coefficients such that $P(0) = 0$.  The coefficients depend on $\Omega$ and the parameters of the problem but not on $v,q, \eta$ or the data.
 
Since $\delta \in (\delta_\omega,1)$ as in Theorem \ref{stokes_om_reg}, we know that $T_{\delta}$ is an isomorphism.  Consequently, \eqref{A_stokes_om_iso_1} is equivalent to the fixed point problem
\begin{equation}\label{A_stokes_om_iso_3}
 (v,q) = T_{\delta}^{-1} \G(v,q)
\end{equation}
on $W^2_{\delta}(\Omega) \times  \oW^1_{\delta}(\Omega)$.  The fact that $T_{\delta}$ is an isomorphism and the estimate \eqref{A_stokes_om_iso_2} then imply that if $\gamma$ is sufficiently small, then 
\begin{equation}\label{A_stokes_om_iso_4}
P(\norm{\eta}_{W^{5/2}_\delta} )  \norm{T^{-1}_{\delta}}_{\mathfrak{X}_{\delta} \to W^2_{\delta} \times  \oW^1_{\delta} } \le 1/2
\end{equation}
and so the map $(v,q) \mapsto T_{\delta}^{-1} \G(v,q)$ is a contraction.  Hence \eqref{A_stokes_om_iso_3} admits a unique solution $(v,q) \in W^2_{\delta}(\Omega) \times  \oW^1_{\delta}(\Omega)$, which in turn implies that \eqref{A_stokes} is uniquely solvable for every $6-$tuple  $(G^1,\dotsc,G^4_-) \in \mathfrak{X}_{\delta}$.  This, the first estimate in \eqref{A_stokes_om_iso_2}, and \eqref{A_stokes_om_iso_4} then imply that  $T_{\delta}[\eta]$ is an isomorphism with this choice of $\gamma$. 
\end{proof}

\subsection{The $\A$-Stokes problem in $\Omega$ with $\beta \neq 0$}

Previously we considered the $\A-$stokes problem \eqref{A_stokes} with the boundary condition $\sg_\A v \nu \cdot \tau = G^4_-$ on $\Sigma_s$.  Now we consider the problem with the boundary condition $\sg_\A v \nu \cdot \tau + \beta v\cdot \tau = G^4_-$ on $\Sigma_s$:
\begin{equation}\label{A_stokes_beta}
\begin{cases}
\diva S_\A(q,v) = G^1 & \text{in }\Omega \\
\diva v = G^2 & \text{in } \Omega \\
v\cdot \N = G^3_+ &\text{on } \Sigma \\
\mu \sg_\A v \N \cdot \mathcal{T} = G^4_+ &\text{on } \Sigma \\
v\cdot \nu = G^3_- &\text{on } \Sigma_s \\
\mu \sg_\A v \nu \cdot \tau + \beta v\cdot \tau = G^4_- &\text{on } \Sigma_s,
\end{cases}
\end{equation}
where $\beta>0$ is the Navier slip friction coefficient on the vessel walls.

\begin{thm}\label{A_stokes_beta_solve}
Let  $\delta \in (\delta_\omega,1)$ be as in Theorem \ref{stokes_om_reg}.  Suppose that  $\ns{\eta}_{W^{5/2}_{\delta} } < \gamma$, where $\gamma$ is as in Theorem \ref{A_stokes_om_iso}.  If  $(G^1,G^2,G^3_+,G^3_-,G^4_+,G^4_-)   \in \mathfrak{X}_{\delta}$ then there exists a unique $(v,q) \in W^2_{\delta}(\Omega) \times  \oW^1_{\delta}(\Omega)$ solving \eqref{A_stokes_beta}.  Moreover, the solution obeys the estimate
\begin{equation}\label{A_stokes_beta_solve_0}
 \ns{v}_{W^2_{\delta}} + \ns{q}_{\oW^1_{\delta}} \ls  \ns{G^1}_{W^0_{\delta} } +  \ns{G^2}_{W^1_{\delta}} + \ns{G^3_+}_{W^{3/2}_{\delta} } +\ns{G^3_-}_{W^{3/2}_{\delta} }  
 +  \ns{G^4_+}_{W^{1/2}_{\delta}} +  \ns{G^4}_{W^{1/2}_{\delta}}. 
\end{equation}
\end{thm}
\begin{proof}

For $0 < \delta < 1$ define the operator $R: W^2_\delta(\Omega) \times  \oW^1_\delta(\Omega) \to \mathfrak{X}_\delta$ via 
\begin{equation}
 R(v,q) = (0,0,0,0,0,\beta v\cdot \nu\vert_{\Sigma_s}),
\end{equation}
which is bounded and well-defined since $v\cdot \nu \in W^{3/2}_\delta(\Sigma_s)$.  In fact, the embedding $W^{3/2}_\delta(\Sigma_s) \hookrightarrow W^{1/2}_\delta(\Sigma_s)$ is compact, so $R$ is a compact operator.  

Theorem \ref{A_stokes_om_iso} tells us that the operator $T_{\delta}[\eta]$ is an isomorphism from $W^2_\delta(\Omega) \times  \oW^1_\delta(\Omega)$ to  $\mathfrak{X}_\delta$.  Since $R$ is compact we have that $T_{\delta}[\eta] + R$ is a Fredholm operator.  

We claim that $T_{\delta}[\eta] + R$ is injective.  To see this we assume that $T_{\delta}[\eta](v,q) + R(v,q) =0$, which is equivalent to \eqref{A_stokes_beta} with vanishing $G^i$ data.  Multiplying the first equation in \eqref{A_stokes_beta} by $J v$ and integrating by parts as in Lemma \ref{geometric_evolution}, we find that 
\begin{equation}
\int_\Omega \frac{\mu}{2} \abs{\sg_\A v}^2 J + \int_{\Sigma_s} \beta\abs{v\cdot \tau}^2 J =0
\end{equation}
and thus that $v=0$.  Then $0 = \nab_\A q = \A \nab q=0$, which implies, since $\A$ is invertible (via Lemma \ref{eta_small}), that $q$ is constant.  Since $q \in \oW^1_{\delta}$ we then have that $q =0$.  This proves the claim.

We now know that $T_{\delta}[\eta] + R$ is injective, so the Fredholm alternative guarantees that it is also surjective and hence is an isomorphism.  From this we deduce that \eqref{A_stokes_beta} is uniquely solvable for any choice of data in $\mathfrak{X}_{\delta}$ and that the estimate \eqref{A_stokes_beta_solve_0} holds.

\end{proof}

\subsection{The $\A$-Stokes problem in $\Omega$ with a boundary equations for $\xi$}
 
We now consider another version of the $\A-$Stokes system in $\Omega$ with boundary conditions on $\Sigma$ involving a new unknown $\xi$: 
\begin{equation}\label{A_stokes_stress}
\begin{cases}
\diva S_\A(q,v) = G^1 & \text{in }\Omega \\
\diva v = G^2 & \text{in } \Omega \\
v\cdot \N = G^3_+ &\text{on } \Sigma \\
S_\A(q,v) \N =  \left[ g\xi  -\sigma \p_1\left( \frac{\p_1 \xi}{(1+\abs{\p_1 \zeta_0}^2)^{3/2}} + G^6 \right)  \right] \N   + G^4_+\frac{\mathcal{T}}{\abs{\mathcal{T}}^2} + G^5\frac{\N}{\abs{\N}^2}  &\text{on } \Sigma \\
v\cdot \nu = G^3_- &\text{on } \Sigma_s \\
 (\Sa(q,v)\nu - \beta v)\cdot \tau = G^4_- &\text{on } \Sigma_s \\
 \mp \sigma \frac{\p_1 \xi}{(1+\abs{\p_1 \zeta_0}^2)^{3/2}} (\pm \ell) = G^7_\pm. 
\end{cases}
\end{equation}

We now construct solutions to \eqref{A_stokes_stress}.

\begin{thm}\label{A_stokes_stress_solve}
Let $\delta \in (\delta_\omega,1)$ be as in Theorem \ref{stokes_om_reg}.  Suppose that  $\ns{\eta}_{W^{5/2}_{\delta} } < \gamma$, where $\gamma$ is as in Theorem \ref{A_stokes_om_iso}.  If  
\begin{equation}
(G^1,G^2,G^3_+,G^3_-,G^4_+,G^4_-)   \in \mathfrak{X}_{\delta}, 
\end{equation}
$G^5,\p_1 G^6 \in W^{1/2}_{\delta}$, and $G^7_\pm \in \R$, then there exists a unique triple $(v,q,\xi) \in W^2_{\delta}(\Omega) \times  \oW^1_{\delta}(\Omega) \times W^{5/2}_{\delta}$ solving \eqref{A_stokes_stress}.  Moreover, the solution obeys the estimate
\begin{multline}\label{A_stokes_stress_0}
 \ns{v}_{W^2_{\delta}} + \ns{q}_{\oW^1_{\delta}} +  \ns{\xi}_{W^{5/2}_{\delta}}\ls  \ns{G^1}_{W^0_{\delta} } +  \ns{G^2}_{W^1_{\delta}} + \ns{G^3_+}_{W^{3/2}_{\delta} } +\ns{G^3_-}_{W^{3/2}_{\delta} }  
 \\
 +  \ns{G^4_+}_{W^{1/2}_{\delta}} +  \ns{G^4_-}_{W^{1/2}_{\delta}}
 + \ns{G^5}_{W^{1/2}_{\delta}} + \ns{\p_1 G^6}_{W^{1/2}_{\delta}} + [G^7]_\ell^2.
\end{multline}
\end{thm}
\begin{proof}
We employ Theorem \ref{A_stokes_beta_solve} to find $(v,q) \in W^2_{\delta}(\Omega) \times  \oW^1_{\delta}(\Omega)$ solving \eqref{A_stokes_beta} and obeying the estimates \eqref{A_stokes_beta_solve_0}.  With this $(v,q)$ in hand we then have a solution to \eqref{A_stokes_stress} as soon as we find $\xi$ solving 
\begin{equation}\label{A_stokes_stress_solve_1}
g\xi  -\sigma \p_1\left( \frac{\p_1 \xi}{(1+\abs{\p_1 \zeta_0}^2)^{3/2}} \right) = S_\A(q,v) \N\cdot \frac{\N}{\abs{\N}^2}     +\sigma \p_1 G^6   -  \frac{G^5}{\abs{\N}^2}
\end{equation}
on $\Sigma$ subject to the boundary conditions
\begin{equation}\label{A_stokes_stress_solve_2}
  \mp \sigma \left(\frac{\p_1 \xi}{(1+\abs{\p_1 \zeta_0}^2)^{3/2}} + F^3 \right)(\pm \ell) = G^7_\pm.
\end{equation}
The estimate \eqref{A_stokes_beta_solve_0} guarantees that $S_\A(q,v) \N\cdot \frac{\N}{\abs{\N}^2} \in W^{1/2}_{\delta}(\Sigma)$, so the usual weighted elliptic theory implies that there exists a unique $\xi \in W^{5/2}_{\delta}(\Sigma)$ satisfying \eqref{A_stokes_stress_solve_1} and \eqref{A_stokes_stress_solve_2} and obeying the estimate 
\begin{multline}\label{A_stokes_stress_solve_3}
\ns{\xi}_{W^{5/2}_{\delta}} \ls \ns{S_\A(q,v) \N\cdot \frac{\N}{\abs{\N}^2}}_{W^{1/2}_{\delta}} + \ns{\p_1 G^6}_{W^{1/2}_{\delta}} + \ns{\frac{G^5}{\abs{\N}^2}}_{W^{1/2}_{\delta}} + [G^7]_\ell^2 \\
\ls \ns{v}_{W^2_{\delta}} + \ns{q}_{\oW^1_{\delta}}  + \ns{\p_1 G^6}_{W^{1/2}_{\delta}} + \ns{G^5}_{W^{1/2}_{\delta}} + [G^7]_\ell^2.
\end{multline}
Then \eqref{A_stokes_stress_0} follow by combining \eqref{A_stokes_beta_solve_0} and \eqref{A_stokes_stress_solve_3}.
\end{proof}

\section{Energy estimate terms }\label{sec_nlin_en}

We will employ the basic energy estimate of Theorem \ref{linear_energy} as the starting point for our a priori estimates.  In order for this to be effective we must be able to estimate the interaction terms appearing on the right side of \eqref{linear_energy_0} when the $F^i$ terms are given as in Appendices \ref{fi_dt1} and \ref{fi_dt2}.  For the sake of brevity we will only present these estimates when the $F^i$ terms are given for the twice temporally differentiated problem, i.e. when $F^i$ are given by \eqref{dt2_f1}--\eqref{dt2_f6}.  The corresponding estimates for the once temporally differentiated problem follow from similar, though often simpler, arguments.  When possible we will present our estimates in the most general form, as estimates for general functionals generated by the $F^i$ terms.  It is only for a few essential terms that we must resort to employing the special structure of the interaction terms in order to close our estimates

In all of the subsequent estimates we abbreviate $d = \dist(\cdot,M)$, where $M =\{(-\ell,\zeta_0(-\ell)), (\ell, \zeta_0(\ell))\}$ is the set of corner points of $\p \Omega$.  Throughout this section we will repeatedly make use of the following simple lemma, whose trivial proof we omit.

\begin{lem}
 Suppose that $d = \dist(\cdot, M)$.  Let $0 < \delta < 1$.  Then $d^{-\delta} \in L^r(\Omega)$ for $1 \le r < \delta/2$.  
\end{lem}

Note also that we will assume throughout the entirety of Section \ref{sec_nlin_ell} that $\eta$ is given and satisfies 
\begin{equation}\label{eta_assume}
 \sup_{0 \le t \le T} \left(  \seb(t) +  \ns{\eta(t)}_{W^{5/2}_\delta(\Omega)} + \ns{\dt \eta(t)}_{H^{3/2}((-\ell,\ell))}  \right) \le \gamma < 1,
\end{equation}
where $\gamma \in (0,1)$ is as in Lemma \ref{eta_small}. For the sake of brevity we will not explicitly state this in each result's hypotheses.

\subsection{Generic functional estimates: velocity term}

On the right side of \eqref{linear_energy_0} we find an interaction term of the form 
\begin{equation}
\br{\mathcal{F},w} = \int_\Omega F^1 \cdot w J 
-  \int_{-\ell}^\ell  F^4 \cdot w 
- \int_{\Sigma_s}  J (w \cdot \tau)F^5 
\end{equation}
for $w \in H^1(\Omega)$.  Our goal now is to prove estimates for this functional.  We will estimate each term separately and then synthesize the estimates.   We begin with an analysis of the $F^1$ term.

\begin{prop}\label{ee_f1}
Let $F^1$ be given by \eqref{dt1_f1} or \eqref{dt2_f1}.  Then we have the estimate  
\begin{equation}\label{ee_f1_0}
 \abs{\int_\Omega J w\cdot F^1  } \ls \norm{w}_1 (\sqrt{\E} + \E) \sqrt{\D}
\end{equation}
for each $w \in H^1(\Omega)$.
\end{prop}

\begin{proof}
We will prove the result only when $F^1$ is given by \eqref{dt2_f2}, i.e. 
\begin{multline}\label{ee_f1_1}
 F^1 = - 2\diverge_{\dt \A} S_\A(\dt p,\dt u) + 2\mu \diva \sg_{\dt \A} \dt u  \\
- \diverge_{\dt^2 \A} S_\A(p,u) + 2 \mu \diverge_{\dt \A} \sg_{\dt \A} u + \mu \diva \sg_{\dt^2 \A} u.
\end{multline}
The result when $F^1$ is of the form \eqref{dt1_f1} follows from a similar but simpler argument.  We will examine each of the terms in \eqref{ee_f1_1} separately.  The estimate \eqref{ee_f1_0}  follow by combining the subsequent estimates of each term.

\textbf{TERM:  $- 2\diverge_{\dt \A} S_\A(\dt p,\dt u)$.}  We begin by estimating 
\begin{multline}
 \abs{\int_\Omega J w \cdot (- 2\diverge_{\dt \A} S_\A(\dt p,\dt u)) } \ls \int_\Omega \abs{w} \abs{\dt \nab \bar{\eta}} (\abs{\nab \dt p} + \abs{\nab^2 \dt u} ) \\
 + \int_\Omega \abs{w} \abs{\dt \nab \bar{\eta}} \abs{\nab^2 \bar{\eta}} (\abs{\dt p} + \abs{\nab \dt u}) =: I + II.
\end{multline}
For $I$ we choose $q\in [1,\infty)$ and $2 < r < 2/\delta$ such that $2/q + 1/r = 1/2$ and estimate 
\begin{multline}
 I \le \norm{w}_{L^q} \norm{\dt \nab \bar{\eta}}_{L^q}  \norm{ d^{-\delta} }_{L^r}  \left(  \norm{ d^\delta \nab \dt p  }_{L^2} + \norm{ d^\delta \nab^2 \dt u  }_{L^2} \right) 
\\
\ls \norm{w}_{1} \norm{\dt \nab \bar{\eta}}_{1}  \left( \norm{\dt p}_{\oW^1_\delta} + \norm{\dt u}_{W^2_\delta}  \right) 
\ls \norm{w}_{1} \norm{\dt \eta }_{3/2}  \left( \norm{\dt p}_{\oW^1_\delta} + \norm{\dt u}_{W^2_\delta}  \right) \ls \norm{w}_1 \sqrt{\E} \sqrt{\D}.
\end{multline}
For $II$ we choose $m = 2/(2-s)$, $2 < r < \delta/2$ such that $1/m + 1/r < 1$, which is possible since $\delta < 1 < s$.  We then choose $q \in [1,\infty)$ such that $3/q + 1/m + 1/r =1$.  This allows us to estimate 
\begin{multline}
 II \le \norm{w}_{L^q} \norm{\dt \nab \bar{\eta}}_{L^q} \norm{\nab^2 \bar{\eta}}_{L^m} \norm{d^{-\delta}}_{L^r}  \left(\norm{ d^\delta \dt p}_{L^q} + \norm{d^\delta \nab \dt u}_{L^q}  \right) 
\\
\ls \norm{w}_1 \norm{\dt \eta}_{3/2} \norm{\nab^2 \bar{\eta}}_{s-1} \left(\norm{  \dt p}_{\W^1_\delta} + \norm{ \nab \dt u}_{W^1_\delta}  \right) \ls \norm{w}_1 \norm{\dt \eta}_{3/2} \norm{\eta}_{s+1/2  } \left(\norm{  \dt p}_{\oW^1_\delta} + \norm{\dt u}_{W^2_\delta}  \right)
\\
\ls \norm{w}_1 \sqrt{\E} \sqrt{\E} \sqrt{\D} = \norm{w}_1 \E \sqrt{\D}.
\end{multline}

\textbf{TERM: $2\mu \diva \sg_{\dt \A} \dt u$.} We begin by estimating 
\begin{equation}
 \abs{\int_\Omega J w \cdot (2\mu \diva \sg_{\dt \A} \dt u )} \ls \int_\Omega \abs{w} \abs{\nab^2 \dt \bar{\eta}} \abs{\nab \dt u} + \int_\Omega \abs{w}\abs{\dt \nab \bar{\eta}} \abs{\nab^2 \dt u} =: I + II.
\end{equation}
For $I$ we let $2 < r < 2/\delta$ and choose $q \in [1,\infty)$ such that $2/q + 1/r =1/2$.  We then estimate 
\begin{multline}
 I \le \norm{w}_{L^q} \norm{\nab^2 \dt \bar{\eta}}_{L^2} \norm{d^{-\delta}}_{L^r} \norm{d^\delta \nab \dt u}_{L^q} \ls 
  \norm{w}_{1} \norm{\dt \bar{\eta}}_{2}  \norm{\nab \dt u}_{W^1_\delta} \\
  \ls  \norm{w}_{1} \norm{\dt \eta}_{3/2}  \norm{\dt u}_{W^2_\delta} \ls \norm{w}_1 \sqrt{\E} \sqrt{\D}. 
\end{multline}
For $II$ we choose the same $r,q$ as for $I$ to estimate 
\begin{equation}
 II \le \norm{w}_{L^q} \norm{\dt \bar{\nab}\eta}_{L^q}  \norm{d^{-\delta} }_{L^r} \norm{d^\delta \nab^2 \dt u}_{L^2} \ls \norm{w}_{1} \norm{\dt \eta}_{3/2} \norm{\dt u}_{W^2_\delta} \ls \norm{w}_{1} \sqrt{\E} \sqrt{\D}.
\end{equation}

\textbf{TERM: $- \diverge_{\dt^2 \A} S_\A(p,u)$.} We start by bounding 
\begin{equation}
 \abs{\int_\Omega J w \cdot (- \diverge_{\dt^2 \A} S_\A(p,u))  } \ls \int_\Omega \abs{w} \abs{\nab \dt^2 \bar{\eta}} (\abs{\nab p} + \abs{\nab^2 u}) + \int_\Omega \abs{w} \abs{\nab \dt^2 \bar{\eta}} \abs{\nab^2 \bar{\eta}} \abs{\nab u} =: I + II.
\end{equation}
For $I$ we choose $2 < r < 2/\delta$ and $q \in [1,\infty)$ such that $2/q + 1/r =1/2$.  We then bound
\begin{multline}
 I \le \norm{w}_{L^q} \norm{\nab \dt^2 \bar{\eta}}_{L^q}  \norm{d^{-\delta}}_{L^r} \left( \norm{d^\delta \nab p }_{L^2}  +   \norm{d^\delta \nab^2 u}_{L^2} \right) \ls \norm{w}_1 \norm{\nab \dt^2 \bar{\eta}}_{1} \left( \norm{p}_{\oW^1_\delta} + \norm{u}_{W^2_\delta}\right) \\
\ls \norm{w}_1 \norm{\dt^2\eta}_{3/2} \left( \norm{p}_{\oW^1_\delta} + \norm{u}_{W^2_\delta}\right)  \ls \norm{w}_1 \sqrt{\D} \sqrt{\E}.
\end{multline}
For $II$ we choose  $m = 2/(2-s)$, $2 < r < \delta/2$ such that $1/m + 1/r < 1$, which is possible since $\delta < 1 < s$.  We then choose $q \in [1,\infty)$ such that $3/q + 1/m + 1/r =1$.  Then 
\begin{multline}
 II \ls \norm{w}_{L^q} \norm{\nab \dt^2 \bar{\eta}}_{L^q} \norm{\nab^2 \bar{\eta}}_{L^m}  \norm{d^{-\delta}}_{L^r} \norm{d^\delta \nab u}_{L^q} \ls \norm{w}_1 \norm{\nab \dt^2 \bar{\eta}}_{1} \norm{\nab^2 \bar{\eta}}_{s-1} \norm{\nab u}_{W^1_\delta} \\
\ls \norm{w}_1 \norm{\dt^2 \eta}_{3/2} \norm{\eta}_{s+1/2} \norm{u}_{W^2_\delta} \ls \norm{w}_1 \sqrt{\D} \sqrt{\E} \sqrt{\E} = \norm{w}_1 \sqrt{\D}\E.
\end{multline}

\textbf{TERM: $2 \mu \diverge_{\dt \A} \sg_{\dt \A} u$.}  We first bound 
\begin{equation}
 \abs{\int_\Omega J w \cdot (2 \mu \diverge_{\dt \A} \sg_{\dt \A} u) } \ls \int_\Omega \abs{w} \abs{\nab \dt \bar{\eta}}^2 \abs{\nab^2 u} + \int_\Omega \abs{w} \abs{\nab \dt \bar{\eta}} \abs{\nab^2 \dt \bar{\eta}} \abs{\nab u} =: I + II.
\end{equation}
For $I$ we let $2 < r < 2/\delta$ and choose $q \in [1,\infty)$ such that $3/q + 1/r = 1/2$.  Then 
\begin{multline}
 I \le \norm{w}_{L^q} \ns{\nab \dt \bar{\eta}}_{L^q} \norm{d^{-\delta}}_{L^r} \norm{d^\delta \nab^2 u}_{L^2} \ls \norm{w}_1 \ns{\nab \dt \bar{\eta}}_1 \norm{u}_{W^2_\delta } \\
 \ls \norm{w}_1 \ns{\dt \eta}_{3/2}  \norm{u}_{W^2_\delta} \ls \norm{w}_1 \E \sqrt{\D}.
\end{multline}
For $II$ we let $m = 2/(2-s)$ and choose $2 < r < 2/\delta$ such that $1/m + 1/r < 1$, which is possible since $\delta < 1 < s$.  We then choose $q \in [1,\infty)$ such that $3/q + 1/m + 1/r =1$.  Then 
\begin{multline}
 II \ls \norm{w}_{L^q} \norm{\nab \dt \bar{\eta}}_{L^q} \norm{\nab^2 \dt \bar{\eta}}_{L^m} \norm{d^{-\delta}}_{L^r} \norm{ d^\delta \nab u}_{L^q} \ls \norm{w}_1 \norm{\nab \dt \bar{\eta}}_1 \norm{\nab^2 \dt \bar{\eta}}_{s-1} \norm{\nab u}_{W^1_\delta} \\
 \ls \norm{w}_1 \norm{\dt \eta}_{3/2} \norm{\dt \eta}_{s+1/2} \norm{u}_{W^2_\delta} \ls \norm{w}_1 \sqrt{\E} \sqrt{\D} \sqrt{\E} = \norm{w}_1 \E \sqrt{\D}.
\end{multline}

\textbf{TERM: $\mu \diva \sg_{\dt^2 \A} u$.}  We estimate 
\begin{equation}
 \abs{\int_\Omega J w \cdot (\mu \diva \sg_{\dt^2 \A} u)} \ls \int_\Omega \abs{w} \abs{\nab \dt^2 \bar{\eta}} \abs{\nab^2 u} + \int_\Omega \abs{w} \abs{\nab^2 \dt^2 \bar{\eta}} \abs{\nab u}  =: I + II.
\end{equation}
For $I$ we choose $2 < r < 2/\delta$ and $q \in [1,\infty)$ such that $2/q + 1/r =1/2$.  Then 
\begin{multline}
 I \le \norm{w}_{L^q} \norm{\nab \dt^2 \bar{\eta}}_{L^q} \norm{d^{-\delta}}_{L^r} \norm{d^\delta \nab^2 u}_{L^2} \ls \norm{w}_1 \norm{\nab \dt^2 \bar{\eta}}_1 \norm{u }_{W^2_\delta} \\
 \ls \norm{w}_1 \norm{\dt^2 \eta}_{3/2} \norm{u}_{W^2_\delta} \ls \norm{w}_1 \sqrt{\D} \sqrt{\E} .
\end{multline}
For $II$ we choose $2 < r < 2/\delta$ and $q \in [1,\infty)$ such that $2/q + 1/r = 1/2$.  Then 
\begin{multline}
 II \le \norm{w}_{L^q} \norm{\nab^2 \dt^2 \bar{\eta}}_{L^2} \norm{d^{-\delta}}_{L^r} \norm{d^\delta \nab u}_{L^q} \ls \norm{w}_1 \norm{ \dt^2 \bar{\eta}}_2 \norm{\nab u}_{W^1_\delta} \\
 \ls \norm{w}_1 \norm{\dt^2 \eta}_{3/2} \norm{u}_{W^2_\delta} \ls \norm{w}_1 \sqrt{\D} \sqrt{\E}.
\end{multline}

\end{proof}

Next we handle the $F^4$ term.

\begin{prop}\label{ee_f4}
Let $F^4$ be given by \eqref{dt1_f4} or \eqref{dt2_f4}.  Then we have the estimate 
\begin{equation}\label{ee_f4_0}
\abs{ \int_{-\ell}^\ell w \cdot  F^4  } \ls \norm{w}_1 (\sqrt{\E} + \E) \sqrt{\D}
\end{equation}
for all $w \in H^1(\Omega)$.
\end{prop}
\begin{proof}
Again we will only prove the result in the more complicated case when $F^4$ is given by \eqref{dt2_f4}, i.e.
\begin{multline}\label{ee_f4_1}
 F^4 = 2\mu \sg_{\dt \A} \dt u \N + \mu \sg_{\dt^2 \A} u \N + \mu \sg_{\dt \A} u \dt \N\\
+\left[  2g \dt \eta  - 2\sigma \p_1 \left(\frac{\p_1 \dt \eta}{(1+\abs{\p_1 \zeta_0}^2)^{3/2}} + \dt[ \R(\p_1 \zeta_0,\p_1 \eta)] \right)  -2 S_{\A}(\dt p,\dt u) \right]  \dt \N 
\\
+ \left[ g\eta  - \sigma \p_1 \left(\frac{\p_1 \eta}{(1+\abs{\p_1 \zeta_0}^2)^{3/2}} +  \R(\p_1 \zeta_0,\p_1 \eta) \right)  - S_{\A}(p,u) \right]  \dt^2 \N.
\end{multline}
The case when $F^4$ is given by \eqref{dt1_f4} is handled by a similar and simpler argument.  We will examine each of the terms in \eqref{ee_f4_1} separately.  The estimate \eqref{ee_f4_0}  follow by combining the subsequent estimates of each term.

In what follows we always let $p$, $q$, and $r$ be given by 
\begin{equation}
 p = \frac{3+\delta}{2+2\delta}, q = \frac{6+2\delta}{1-\delta}, \text{ and } r = \frac{9+3\delta}{1-\delta}
\end{equation}
which implies that 
\begin{equation}
 \frac{1}{p} + \frac{3}{r} = 1 \text{ and }  \frac{1}{p} + \frac{2}{q} = 1.
\end{equation}

\textbf{TERM: $ 2\mu \sg_{\dt \A} \dt u \N$.}  We estimate 
\begin{multline}
\abs{\int_{-\ell}^\ell 2\mu w \cdot (\sg_{\dt \A} \dt u)(\N) }
 \ls  \norm{w}_{L^q(\Sigma)}   \norm{\dt \A}_{L^q(\Sigma)} \norm{\nab \dt u}_{L^p(\Sigma)}  \norm{\N}_{L^\infty} \\
\ls \norm{w}_{H^{1/2}(\Sigma)}   \norm{\dt \bar{\eta}}_{H^{3/2}(\Sigma)} \norm{\dt u}_{W^2_\delta} 
\ls \norm{w}_{1}   \norm{\dt \eta}_{3/2} \norm{\dt u}_{W^2_\delta} 
\ls \norm{w}_{1}  \sqrt{\E} \sqrt{\D} . 
\end{multline}

\textbf{TERM: $ \mu \sg_{\dt^2 \A} u \N$.} We estimate 
\begin{multline}
\abs{ \int_{-\ell}^\ell w \cdot \mu \sg_{\dt^2 \A} u \N } 
\ls \norm{w}_{L^{q}(\Sigma)}  \norm{\N}_{L^\infty(\Sigma)} \norm{\dt^2 \A}_{L^q(\Sigma)}  \norm{\nab u}_{L^{p}(\Sigma)} \\
\ls\norm{w}_{H^{1/2}(\Sigma)}  \norm{\dt^2 \eta}_{3/2}  \norm{u}_{W^{2}_\delta} 
\ls \norm{w}_{1}  \norm{\dt^2 \eta}_{3/2} \norm{u}_{W^2_\delta} \ls \norm{w}_{1}  \sqrt{\D}\sqrt{\E} .
\end{multline}


\textbf{TERM: $\mu \sg_{\dt \A} u \dt \N$.}  We estimate 
\begin{multline}
\abs{ \int_{-\ell}^\ell w \cdot \mu \sg_{\dt \A} u \dt \N  }
 \ls  \norm{w}_{L^r(\Sigma)}   \norm{\dt \A}_{L^r(\Sigma)} \norm{\nab u}_{L^p(\Sigma)}  \norm{\dt \N}_{L^r} \\
\ls \norm{w}_{H^{1/2}(\Sigma)}   \ns{\dt \eta}_{3/2} \norm{u}_{W^2_\delta} 
\ls \norm{w}_{1}  \ns{\dt \eta}_{3/2} \norm{u}_{W^2_\delta} 
\ls \norm{w}_{1}  \E \sqrt{\D}. 
\end{multline}

\textbf{TERM: $\left[  2g \dt \eta  - 2\sigma \p_1 \left(\frac{\p_1 \dt \eta}{(1+\abs{\p_1 \zeta_0}^2)^{3/2}} \right)  \right]  \dt \N $.} We estimate
\begin{multline}
\abs{ \int_{-\ell}^\ell w \cdot \left[  2g \dt \eta  - 2\sigma \p_1 \left(\frac{\p_1 \dt \eta}{(1+\abs{\p_1 \zeta_0}^2)^{3/2}}  \right)  \right]  \dt \N  }
\ls \norm{w}_{L^q(\Sigma)}\left( \norm{\dt \eta}_{L^p} +\norm{\p_1^2 \dt \eta}_{L^p} \right) \norm{\p_1 \dt \eta}_{L^q} \\
\ls \norm{w}_{H^{1/2}(\Sigma)} \norm{\dt \eta}_{W^{5/2}_\delta}  \norm{\dt \eta}_{3/2}
\ls \norm{w}_{1} \norm{\dt \eta}_{W^{5/2}_\delta}  \norm{\dt \eta}_{3/2} \ls \norm{w}_1 \sqrt{\D} \sqrt{\E} .
\end{multline}

\textbf{TERM: $\p_1 \dt[ \R(\p_1 \zeta_0,\p_1 \eta)]    \dt \N$.} To begin we expand the term via
\begin{multline}
 \p_1 \dt[ \R(\p_1 \zeta_0,\p_1 \eta)]  = \p_1 [\p_z \R(\p_1 \zeta_0,\p_1 \eta) \p_1 \dt \eta]  =  \frac{\p_z \R(\p_1 \zeta_0,\p_1 \eta)}{\p_1 \eta} \p_1 \eta \p_1^2 \dt \eta  + \p_z^2 \R(\p_1 \zeta_0,\p_1 \eta) \p_1^2 \eta \p_1 \dt \eta \\
 + \frac{ \p_z \p_y \R(\p_1 \zeta_0,\p_1 \eta)}{\p_1 \eta} \p_1 \eta \p_1^2 \zeta_0 \p_1 \dt \eta.
\end{multline}
This allows us to estimate 
\begin{multline}
\abs{ \int_{-\ell}^\ell w \cdot  \p_1 [\dt[ \R(\p_1 \zeta_0,\p_1 \eta)] ] \dt \N }
\\
 \ls \norm{w}_{L^{r}(\Sigma)} \norm{\p_1 \dt \eta}_{L^{r}} \left( \norm{\p_1 \eta}_{L^{r}} \norm{\p_1^2 \dt \eta}_{L^{p}}  + \norm{\p_1 \dt \eta}_{L^{r}}\norm{\p_1^2 \eta}_{L^{p}} + \norm{\p_1 \eta}_{L^{r}}\norm{\p_1 \dt \eta}_{L^{p}}     \right) \\
\ls  \norm{w}_{H^{1/2}(\Sigma)} \norm{\p_1 \dt \eta}_{1/2} \left( \norm{\p_1 \eta}_{1/2} \norm{ \dt \eta}_{W^{5/2}_\delta}  + \norm{\p_1 \dt \eta}_{1/2}\norm{\eta}_{W^{5/2}_\delta} + \norm{\p_1 \eta}_{1/2}\norm{\p_1 \dt \eta}_{1/2}     \right)\\
\ls  \norm{w}_{1} \norm{\dt \eta}_{3/2} \left( \norm{\eta}_{3/2} \norm{\dt \eta}_{W^{5/2}_\delta}  + \norm{\dt \eta}_{3/2}\norm{\eta}_{W^{5/2}_\delta} + \norm{\eta}_{3/2}\norm{\dt \eta}_{3/2} \right) \\
\ls \norm{w}_1 \sqrt{\E} \left(\sqrt{\E} \sqrt{\D} + \sqrt{\E} \sqrt{\D} + \sqrt{\E} \sqrt{\D} \right) \ls \norm{w}_1 \E \sqrt{\D}.
\end{multline}

\textbf{TERM: $ -2 S_{\A}(\dt p,\dt u)  \dt \N$.} We estimate 
\begin{multline}
\abs{ \int_{-\ell}^\ell -2 w \cdot  S_{\A}(\dt p,\dt u)  \dt \N }
 \ls \norm{w}_{L^q(\Sigma)} \norm{\A}_{L^\infty} \left(\norm{\dt p}_{L^p(\Sigma)} + \norm{\nab \dt u}_{L^p(\Sigma)} \right) \norm{\dt \p_1 \eta}_{L^q} \\
\ls \norm{w}_{H^{1/2}(\Sigma)} \left( \norm{\dt p}_{\oW^1_\delta} + \norm{\dt u}_{W^2_\delta} \right) \norm{\dt \p_1 \eta}_{1/2} \ls \norm{w}_1 \left( \norm{\dt p}_{\oW^1_\delta} + \norm{\dt u}_{W^2_\delta} \right) \norm{\dt \eta}_{3/2} \\
\ls \norm{w}_1 \sqrt{\D} \sqrt{\E}.
\end{multline}

\textbf{TERM: $\left[ g\eta  - \sigma \p_1 \left(\frac{\p_1 \eta}{(1+\abs{\p_1 \zeta_0}^2)^{3/2}}   \right)  \right]  \dt^2 \N$.} We estimate 
\begin{multline}
\abs{ \int_{-\ell}^\ell w \cdot \left[ g\eta  - \sigma \p_1 \left(\frac{\p_1 \eta}{(1+\abs{\p_1 \zeta_0}^2)^{3/2}}   \right)  \right]  \dt^2 \N  }
 \ls \norm{w}_{L^{q}(\Sigma)} \norm{\p_1 \dt^2 \eta}_{L^{q}} \left(\norm{\eta}_{L^p} + \norm{\p_1^2 \eta}_{L^p}  \right) \\
\ls \norm{w}_{H^{1/2}(\Sigma)} \norm{\p_1 \dt^2 \eta}_{1/2} \norm{\eta}_{W^{5/2}_\delta}   
\ls \norm{w}_{1} \norm{\dt^2 \eta}_{3/2} \norm{\eta}_{W^{5/2}_\delta}
 \ls \norm{w}_1 \sqrt{\D} \sqrt{\E}.
\end{multline}

\textbf{TERM: $ \p_1[  \R(\p_1 \zeta_0,\p_1 \eta) ]     \dt^2 \N$.}  We expand
\begin{equation}
 \p_1 [ \R(\p_1 \zeta_0,\p_1 \eta)]    =  \frac{\p_z \R(\p_1 \zeta_0,\p_1 \eta)}{\p_1 \eta} \p_1 \eta \p_1^2 \eta  
 + \frac{ \p_y \R(\p_1 \zeta_0,\p_1 \eta)}{(\p_1 \eta)^2}\p_1^2 \zeta_0 (\p_1 \eta)^2 .
\end{equation}
This allows us to  estimate 
\begin{multline}
\abs{ \int_{-\ell}^\ell w \cdot \p_1[  \R(\p_1 \zeta_0,\p_1 \eta) ]     \dt^2 \N  }
 \ls  \norm{w}_{L^{r}(\Sigma)} \norm{\p_1 \dt^2 \eta}_{L^{r}} \left(\norm{\p_1 \eta}_{L^r} \norm{\p_1^2 \eta }_{L^p} + \norm{\p_1 \eta}_{L^r} \norm{\p_1 \eta }_{L^p}  \right) \\
 \ls  \norm{w}_{1} \norm{\dt^2 \eta}_{3/2} \left(\norm{\eta}_{3/2} \norm{\eta }_{W^{5/2}_\delta} + \ns{\eta}_{3/2}  \right) \ls \norm{w}_1 \sqrt{\D}   (\sqrt{\E} \sqrt{\E} + \sqrt{\E} \sqrt{\E}) \ls \norm{w}_1  \E \sqrt{\D}.
\end{multline}

\textbf{TERM: $- S_{\A}(p,u)   \dt^2 \N$.}  We  estimate 
\begin{multline}
\abs{ - \int_{-\ell}^\ell w \cdot S_{\A}(p,u)   \dt^2 \N }
\ls  \norm{w}_{L^{q}(\Sigma)} \norm{\p_1 \dt^2 \eta}_{L^{q}}  \left( \norm{p}_{L^p(\Sigma)} + \norm{\A}_{L^\infty} \norm{\nab u}_{L^p(\Sigma)}   \right) \\
\ls  \norm{w}_{1} \norm{\dt^2 \eta}_{3/2}  \left( \norm{p}_{W^1_\delta} + \norm{u}_{W^2_\delta}   \right) \ls \norm{w}_1 \sqrt{\D}  \sqrt{\E}.
\end{multline}

\end{proof}

Now we handle the $F^5$ term.

\begin{prop}\label{ee_f5}
Let $F^5$ be given by \eqref{dt1_f5} or \eqref{dt2_f5}.  Then we have the estimate
\begin{equation}\label{ee_f5_0}
 \abs{ \int_{\Sigma_s}  J(w \cdot \tau)F^5 } \ls \norm{w}_1 \sqrt{\E}\sqrt{\D} 
\end{equation}
for every $w \in H^1(\Omega)$.
\end{prop}
\begin{proof}
 
Once more we will only prove the result in the harder case, when 
\begin{equation}
 F^5 = 2 \mu \sg_{\dt \A} \dt u \nu \cdot \tau + \mu \sg_{\dt^2 \A} u \nu \cdot \tau.
\end{equation}
The simpler case of \eqref{dt1_f5} follows from a similar argument.

Let $p$ and $q$  be given by 
\begin{equation}
 p = \frac{3+\delta}{2+2\delta} \text{ and }  q = \frac{6+2\delta}{1-\delta}
\end{equation}
which implies that 
\begin{equation}
  \frac{1}{p} + \frac{2}{q} = 1.
\end{equation}

To handle the first term in $F^5$ we bound
\begin{multline}
\abs{ - \int_{\Sigma_s}  J(w \cdot \tau)(2 \mu \sg_{\dt \A} \dt u \nu \cdot \tau)  }
 \ls  \norm{w}_{L^q(\Sigma_s)}   \norm{\dt \A}_{L^p(\Sigma_s)} \norm{\nab \dt u}_{L^p(\Sigma_s)}  \norm{J }_{L^\infty} \\
\ls \norm{w}_{H^{1/2}(\Sigma_s)}   \norm{\dt \bar{\eta}}_{H^{3/2}(\Sigma_s)} \norm{\dt u}_{W^2_\delta} 
\ls \norm{w}_{1}   \norm{\dt \eta}_{3/2} \norm{\dt u}_{W^2_\delta} 
\ls \norm{w}_{1}   \sqrt{\E} \sqrt{\D}. 
\end{multline}
For the second we estimate 
\begin{multline}
\abs{ - \int_{\Sigma_s}  J(w \cdot \tau)( \mu \sg_{\dt^2 \A} u \nu \cdot \tau) }
\ls  \norm{J }_{L^\infty(\Sigma_s)} \norm{w}_{L^q(\Sigma_s)} \norm{\dt^2 \A}_{L^q(\Sigma_s)}  \norm{\nab u}_{L^p(\Sigma_s)} \\
\ls \norm{w}_{H^{1/2}(\Sigma_s)} \norm{\dt^2 \eta}_{3/2}  \norm{u}_{W^2_\delta} \\
\ls \norm{w}_{1} \norm{\dt^2 \eta}_{3/2}  \norm{u}_{W^2_\delta} \ls \norm{w}_1 \sqrt{\D} \sqrt{\E}.
\end{multline}
The estimate \eqref{ee_f5_0} follows immediately from these two bounds.

\end{proof}

We now combine the previous analysis into a single result.

\begin{thm}\label{ee_v_est}
Define the functional $H^1(\Omega) \ni w\mapsto \br{\mathcal{F},w} \in \mathbb{R}$ via 
\begin{equation}
\br{\mathcal{F},w} = \int_\Omega F^1 \cdot w J 
-  \int_{-\ell}^\ell  F^4 \cdot w 
- \int_{\Sigma_s}  J (w \cdot \tau)F^5, 
\end{equation} 
where $F^1,F^4,F^5$ are defined by either by \eqref{dt1_f1}, \eqref{dt1_f4}, and \eqref{dt1_f5} or else by \eqref{dt2_f1}, \eqref{dt2_f4}, and \eqref{dt2_f5}.  Then 
\begin{equation}
\abs{\br{\mathcal{F},w}} \ls  \norm{w}_1  (\E+ \sqrt{\E})\sqrt{\D}
\end{equation}
for all $w \in H^1(\Omega)$.
\end{thm}
\begin{proof}
We simply combine Propositions \ref{ee_f1}, \ref{ee_f4}, and \ref{ee_f5}.  
\end{proof}

We will need the following variant to use in conjunction with Theorem \ref{pressure_est}.

\begin{thm}\label{ee_v_est_pressure}
Define the functional $H^1(\Omega) \ni w\mapsto \br{\mathcal{F},w} \in \mathbb{R}$ via 
\begin{equation}
\br{\mathcal{F},w} = \int_\Omega F^1 \cdot w J 
- \int_{\Sigma_s}  J (w \cdot \tau)F^5, 
\end{equation} 
where $F^1,F^4,F^5$ are defined by either by \eqref{dt1_f1}, \eqref{dt1_f4}, and \eqref{dt1_f5} or else by \eqref{dt2_f1}, \eqref{dt2_f4}, and \eqref{dt2_f5}.  Then 
\begin{equation}
\abs{\br{\mathcal{F},w}} \ls  \norm{w}_1  (\E+ \sqrt{\E})\sqrt{\D}
\end{equation}
for all $w \in H^1(\Omega)$.
\end{thm}
\begin{proof}
We simply combine Propositions \ref{ee_f1} and \ref{ee_f5}.  
\end{proof}

\subsection{Generic nonlinear estimates: $F^3$ term}

Theorem \ref{ee_v_est} will be useful for applying Theorem \ref{linear_energy}, but it will also play a role in the estimates of Theorem \ref{xi_est}.  We now record a quick estimate involving $F^3$ that will also be useful there.

\begin{thm}\label{ee_f3_half}
Let 
\begin{equation}
 F^3 = \p_z \R(\p_1 \zeta_0,\p_1 \eta)  \p_1 \dt^2 \eta  + \p_z^2 \R(\p_1 \zeta_0,\p_1 \eta) (\p_1 \dt \eta)^2.
\end{equation}
Then
\begin{equation}
 \norm{F^3}_{1/2} \ls \sqrt{\E} \sqrt{\D}.
\end{equation}
\end{thm}
\begin{proof}
Since $s-1/2 >1/2$ we may estimate 
\begin{equation}
 \norm{\p_z \R(\p_1 \zeta_0,\p_1 \eta)  \p_1 \dt^2 \eta }_{1/2} \ls \norm{\p_1 \eta}_{s-1/2} \norm{\p_1 \dt^2 \eta}_{1/2} \ls \norm{\eta}_{s+1/2} \norm{\dt^2 \eta}_{3/s} \ls \sqrt{\E} \sqrt{\D}.
\end{equation}
Similarly, 
\begin{equation}
 \norm{\p_z^2 \R(\p_1 \zeta_0,\p_1 \eta) (\p_1 \dt \eta)^2}_{1/2} \ls  \norm{\p_z^2 \R(\p_1 \zeta_0,\p_1 \eta) (\p_1 \dt \eta)^2}_{s-1/2} \ls \ns{\dt \eta}_{3/2} \ls \sqrt{\E} \sqrt{\D}.
\end{equation}
\end{proof}

\subsection{Generic functional estimates: pressure term}

On the right side of \eqref{linear_energy_0} we find an interaction term of the form 
\begin{equation}
 \int_\Omega J \psi F^2
\end{equation}
for $\psi \in L^2(\Omega)$.  We now provide an estimate for this functional.

\begin{thm}\label{ee_p_est}
Let $F^2$ be given by either \eqref{dt1_f2} or \eqref{dt2_f2}.  Then 
\begin{equation}\label{ee_p_est_0}
 \abs{ \int_\Omega J \psi F^2} \ls \norm{\psi}_{L^2} \sqrt{\D} \sqrt{\E}
\end{equation}
for every $\psi \in L^2(\Omega)$.
\end{thm}
\begin{proof}
We will only prove the result in the harder case \eqref{dt2_f2}, i.e. when  
\begin{equation}\label{ee_p_est_1}
 F^2 = -\diverge_{\dt^2 \A} u - 2\diverge_{\dt \A}\dt u.
\end{equation}
The easier case \eqref{dt1_f2} follows from a simpler argument.

To handle the first term in \eqref{ee_p_est_1} we choose $2 < r < 2/\delta$ and $q\in[1,\infty)$ such that $2/q + 1/r =1/2$.  This allows us to estimate  
\begin{multline}
 \abs{\int_\Omega J \psi (-\diverge_{\dt^2 \A} u )}\ls \int_\Omega \abs{\psi} \abs{\nab \dt^2 \bar{\eta}} \abs{\nab u} \ls \norm{\psi}_{L^2} \norm{\nab \dt^2 \bar{\eta}}_{L^q} \norm{d^{-\delta}}_{L^r} \norm{d^\delta \nab u}_{L^q} 
\\
\ls \norm{\psi}_{L^2} \norm{\nab \dt^2 \bar{\eta}}_{1} \norm{\nab u}_{W^1_\delta} \ls \norm{\psi}_{L^2} \norm{\dt^2 \eta}_{3/2} \norm{u}_{W^2_\delta} \ls \norm{\psi}_{L^2} \sqrt{\D} \sqrt{\E}.
\end{multline}
To handle the second term in \eqref{ee_p_est_1} we choose  $2 < r < 2/\delta$ and $q\in[1,\infty)$ such that $2/q + 1/r =1/2$.  Then 
\begin{multline}
 \abs{\int_\Omega J \psi(- 2\diverge_{\dt \A}\dt u)} \ls \int_\Omega \abs{\psi} \abs{\nab \dt \bar{\eta}} \abs{\nab \dt u} \ls \norm{\psi}_{L^2} \norm{\nab \dt \bar{\eta}}_{L^q} \norm{d^{-\delta}}_{L^r} \norm{d^\delta \nab \dt u}_{L^q} 
 \\
\ls \norm{\psi}_{L^2} \norm{\nab \dt \bar{\eta}}_1 \norm{\nab \dt u}_{W^1_\delta} \ls \norm{\psi}_{L^2} \norm{\dt \eta}_{3/2} \norm{\dt u}_{W^2_\delta} \ls \norm{\psi}_{L^2} \sqrt{\E} \sqrt{\D}. 
\end{multline}
The estimate \eqref{ee_p_est_0} then follows by combining these.

\end{proof}

\subsection{Special functional estimates: velocity term}

On the right side of \eqref{linear_energy_0} we encounter the terms 
\begin{equation}
-  \int_{-\ell}^\ell \sigma  F^3   \p_1 (\dt^2 u \cdot \N)  \text{ and } -  \int_{-\ell}^\ell \sigma  F^3   \p_1 (\dt u \cdot \N) 
\end{equation}
with $F^3$ defined by \eqref{dt2_f3} in the first case and \eqref{dt1_f3} in the second case.  We do not have the luxury of estimating these as generic functionals and instead must exploit the special structure shared between $F^3$ and $\dt^j u \cdot \N$.   

We will only present the analysis for the first term, which is harder to control due to the second-order temporal derivative.  The analysis for the second term follows from similar, easier estimates.  In the second-order case we have that
\begin{equation}
 F^3 = \dt^2 [ \R(\p_1 \zeta_0,\p_1 \eta)] = \p_z \R(\p_1 \zeta_0,\p_1 \eta)  \p_1 \dt^2 \eta  + \p_z^2 \R(\p_1 \zeta_0,\p_1 \eta) (\p_1 \dt \eta)^2.
\end{equation}
For the purposes of these estimates we will write
\begin{equation}
 \dt^2 u \cdot \N = \dt^3 \eta - F^6.  
\end{equation}
We may then decompose 
\begin{multline}\label{ee_f3_decomp}
-  \int_{-\ell}^\ell \sigma  F^3   \p_1 (\dt^2 u \cdot \N)  = 
-  \int_{-\ell}^\ell \sigma \p_z \R(\p_1 \zeta_0,\p_1 \eta)  \p_1 \dt^2 \eta  \p_1 \dt^3 \eta
-  \int_{-\ell}^\ell \sigma \p_z \R(\p_1 \zeta_0,\p_1 \eta)  \p_1 \dt^2 \eta  \p_1 F^6 \\
-  \int_{-\ell}^\ell \sigma \p_z^2 \R(\p_1 \zeta_0,\p_1 \eta) (\p_1 \dt \eta)^2 \p_1 \dt^3 \eta 
-  \int_{-\ell}^\ell \sigma \p_z^2 \R(\p_1 \zeta_0,\p_1 \eta) (\p_1 \dt \eta)^2 \p_1 F^6 \\
= I + II + III + IV.
\end{multline}

We will estimate each of these terms separately.  We begin with $I$.

\begin{prop}\label{ee_f3_I}
Let $I$ be as given in \eqref{ee_f3_decomp}.  Then 
\begin{equation}
 \abs{I +  \frac{d}{dt} \int_{-\ell}^\ell \sigma \p_z \R(\p_1 \zeta_0,\p_1 \eta)  \frac{\abs{\p_1 \dt^2 \eta}^2}{2} } \ls \sqrt{\E} \D.
\end{equation}
Moreover, 
\begin{equation}
  \abs{ \int_{-\ell}^\ell \sigma \p_z \R(\p_1 \zeta_0,\p_1 \eta)  \frac{\abs{\p_1 \dt^2 \eta}^2}{2} } \ls \sqrt{\E} \seb.
\end{equation}
\end{prop}
\begin{proof}
We compute 
\begin{multline}
 I = -  \int_{-\ell}^\ell \sigma \p_z \R(\p_1 \zeta_0,\p_1 \eta) \dt \frac{\abs{\p_1 \dt^2 \eta}^2}{2} = -  \frac{d}{dt} \int_{-\ell}^\ell \sigma \p_z \R(\p_1 \zeta_0,\p_1 \eta)  \frac{\abs{\p_1 \dt^2 \eta}^2}{2} \\
 + \int_{-\ell}^\ell \sigma \p_z^2 \R(\p_1 \zeta_0,\p_1 \eta) \p_1 \dt \eta  \frac{\abs{\p_1 \dt^2 \eta}^2}{2}.
\end{multline}
We then estimate 
\begin{multline}
 \abs{\int_{-\ell}^\ell \sigma \p_z^2 \R(\p_1 \zeta_0,\p_1 \eta) \p_1 \dt \eta  \frac{\abs{\p_1 \dt^2 \eta}^2}{2}} \ls \norm{\p_1 \dt\eta}_{L^3} \ns{\p_1 \dt^2 \eta}_{L^3} \ls \norm{\p_1 \dt\eta}_{1/2} \ns{\p_1 \dt^2 \eta}_{1/2} \\
 \ls \norm{ \dt\eta}_{3/2} \ns{ \dt^2 \eta}_{3/2} \ls \sqrt{\E} \D.
\end{multline}
We can also estimate the term in the time derivative:
\begin{multline}
 \abs{ \int_{-\ell}^\ell \sigma \p_z \R(\p_1 \zeta_0,\p_1 \eta)  \frac{\abs{\p_1 \dt^2 \eta}^2}{2} } \ls  \int_{-\ell}^\ell \abs{\p_1 \eta}  \frac{\abs{\p_1 \dt^2 \eta}^2}{2} \\
 \ls \norm{\p_1 \eta}_{L^\infty} \ns{\p_1 \dt^2 \eta}_{0} \ls \norm{\eta}_{s+1/2} \ns{ \dt^2 \eta}_{1} \ls \sqrt{\E}\seb. 
\end{multline}
\end{proof}

Next we handle $II$.

\begin{prop}\label{ee_f3_II}
 Let $II$ be as given in \eqref{ee_f3_decomp}.  Then
\begin{equation}\label{ee_f3_II_0}
 \abs{II} \ls \E \D.
\end{equation}
\end{prop}
\begin{proof}

We may write $F^6$ as
\begin{equation}
 F^6 = -2 \dt u_1 \p_1 \dt \eta - u_1 \p_1 \dt^2 \eta.
\end{equation}
Using this, we may rewrite 
\begin{multline}
II =   \int_{-\ell}^\ell \sigma \frac{\p_z \R(\p_1 \zeta_0,\p_1 \eta)}{\p_1 \eta} \p_1 \eta  \p_1 \dt^2 \eta  2 \p_1  \dt u_1 \p_1 \dt \eta +  \int_{-\ell}^\ell \sigma \frac{\p_z \R(\p_1 \zeta_0,\p_1 \eta)}{\p_1 \eta} \p_1 \eta  \p_1 \dt^2 \eta  2   \dt u_1 \p_1^2 \dt \eta  \\
+ \int_{-\ell}^\ell \sigma \frac{\p_z \R(\p_1 \zeta_0,\p_1 \eta)}{\p_1 \eta} \p_1 \eta  \p_1 \dt^2 \eta  \p_1 u_1 \p_1 \dt^2 \eta 
+ \int_{-\ell}^\ell \sigma \p_z \R(\p_1 \zeta_0,\p_1 \eta)   \p_1 \dt^2 \eta  u_1 \p_1^2 \dt^2 \eta 
\\=: II_1 + II_2 + II_3 + II_4.
\end{multline}

Notice that Proposition \ref{weighted_embed} and the usual trace theory in $W^{1,p}(\Omega)$ allows us to estimate 
\begin{equation}
 \norm{\nab \dt u}_{L^p(\p \Omega)} \ls  \norm{\nab \dt u}_{W^{1,p}(\Omega)} \ls \norm{\nab \dt u}_{W^1_\delta} \ls \norm{\dt u}_{W^2_\delta} \ls \sqrt{\D}
\end{equation}
for any $1 \le p < 2/(1+\delta)$.  In particular, we may choose $p = (3+\delta)/(2+2\delta) \in [1,2/(1+\delta))$ and $q = (6+2\delta)/(1-\delta)$, which satisfy $1/p + 2/q =1$, in order to estimate 
\begin{multline}\label{ee_f3_II_1}
 \abs{II_1}  \ls \norm{\p_1 \eta}_{L^\infty} \norm{\p_1 \dt^2 \eta }_{L^{q}} 
 \norm{\p_1 \dt u}_{L^{p}(\Sigma)}  \norm{\p_1 \dt \eta}_{L^{q}} 
\ls \norm{\p_1 \eta}_{s-1/2} \norm{\p_1 \dt^2 \eta }_{1/2}   \norm{\dt u}_{W^2_\delta}  \norm{\p_1 \dt \eta}_{1/2} 
 \\
\ls \norm{\eta}_{s+1/2} \norm{\dt^2 \eta }_{3/2} \norm{\dt u}_{W^2_\delta}  \norm{\dt \eta}_{3/2}  \ls \sqrt{\E} \sqrt{\D} \sqrt{\D} \sqrt{\E}  \ls \E \D.
\end{multline}

For $II_2$ we use Proposition \ref{weighted_trace} to bound
\begin{equation}
 \norm{\p_1^2 \dt \eta}_{L^p} \ls \norm{\p_1^2 \dt \eta}_{W^{1/2}_\delta} \ls \norm{\dt \eta}_{W^{5/2}_\delta} \ls \sqrt{\D}
\end{equation}
for any $1 \le p < 2/(1+\delta)$.  Choosing the same $p$ and $q$ as for $II_1$ above, we then estimate 
\begin{multline}\label{ee_f3_II_2}
 \abs{II_2} \ls  \norm{\p_1 \eta}_{L^\infty} \norm{\p_1 \dt^2 \eta }_{L^{q}}  \norm{\dt u_1}_{L^q(\Sigma)} \norm{\p_1^2 \dt \eta}_{L^p} \ls \norm{\p_1 \eta}_{s-1/2} \norm{\p_1 \dt^2 \eta }_{1/2}  \norm{\dt u}_{H^{1/2}(\Sigma)}  \norm{ \dt \eta}_{W^{5/2}_\delta}
\\
\ls  \norm{\eta}_{s+1/2} \norm{\dt^2 \eta }_{3/2} \norm{\dt u}_1 \norm{ \dt \eta}_{W^{5/2}_\delta} \ls \sqrt{\E} \sqrt{\D} \sqrt{\E} \sqrt{\D} = \E \D.
\end{multline}

Again using Proposition \ref{weighted_embed} and trace theory we may bound 
\begin{equation}
 \norm{\nab  u}_{L^p(\p \Omega)} \ls  \norm{\nab  u}_{W^{1,p}(\Omega)} \ls \norm{\nab  u}_{W^1_\delta} \ls \norm{u}_{W^2_\delta} \ls \sqrt{\D}
\end{equation}
for $1 \le p < 2/(1+\delta)$.  Arguing as with $II_1$ for the same $p,q$ we then may estimate
\begin{equation}\label{ee_f3_II_3}
 \abs{II_3}  \ls \norm{\p_1 \eta}_{L^\infty} \norm{\p_1 u_1}_{L^{p}} \ns{\p_1 \dt^2 \eta}_{L^{q}} 
 \ls \norm{\eta}_{s+1/2} \norm{u}_{W^2_\delta} \ns{ \dt^2 \eta}_{3/2} \ls \E \D.
\end{equation}
Since $u_1 =0$ at the endpoints we may similarly estimate
\begin{multline}\label{ee_f3_II_4}
 \abs{II_4} = \abs{\int_{-\ell}^\ell \sigma \p_z \R(\p_1 \zeta_0,\p_1 \eta)  u_1  \p_1 \frac{\abs{ \p_1 \dt^2 \eta}^2}{2} } = 
 \abs{- \int_{-\ell}^\ell \sigma \p_1 \left[ \p_z \R(\p_1 \zeta_0,\p_1 \eta)  u_1 \right]  \frac{\abs{ \p_1 \dt^2 \eta}^2}{2} }
 \\
= \abs{ \int_{-\ell}^\ell  \sigma \left[  \frac{\p_z \R(\p_1 \zeta_0,\p_1 \eta)}{\p_1 \eta} \p_1 \eta \p_1  u_1 +  \p_z^2 \R(\p_1 \zeta_0,\p_1 \eta) \p_1^2 \eta u_1 + \frac{\p_y \p_z \R(\p_1 \zeta_0,\p_1 \eta)\p_1^2 \zeta_0}{\p_1\eta} \p_1 \eta   u_1\right]  \frac{\abs{ \p_1 \dt^2 \eta}^2}{2} }
\\
\ls \left( \norm{\p_1 \eta}_{L^\infty} \norm{\p_1 u_1}_{L^{p}(\Sigma)}  +  \norm{\p_1^2 \eta}_{L^{p}} \norm{u_1}_{L^\infty} + \norm{\p_1 \eta}_{L^{p}} \norm{u_1}_{L^\infty}  \right)  \ns{\p_1 \dt^2 \eta}_{L^{q}} 
\\
\ls \left( \norm{\eta}_{s+1/2} \norm{u}_{W^2_\delta}  +  \norm{\p_1^2 \eta}_{W^{1/2}_\delta} \norm{u_1}_{s} + \norm{\eta}_{s+1/2} \norm{u_1}_{s}  \right)  \ns{\dt^2 \eta}_{3/2} \\
\ls (\sqrt{\E} \sqrt{\E} + \sqrt{\E} \sqrt{\E} + \sqrt{\E} \sqrt{\E} )\D \ls \E \D.
\end{multline}

The estimate \eqref{ee_f3_II_0} then follows by combining \eqref{ee_f3_II_1}, \eqref{ee_f3_II_2}, \eqref{ee_f3_II_3}, and \eqref{ee_f3_II_4}.

\end{proof}

Next we handle $III$.

\begin{prop}\label{ee_f3_III}
Let $III$ be as given in \eqref{ee_f3_decomp}.  Then 
\begin{equation}
 \abs{III + \frac{d}{dt} \int_{-\ell}^\ell \sigma \p_z^2 \R(\p_1 \zeta_0,\p_1 \eta) (\p_1 \dt \eta)^2 \p_1 \dt^2 \eta}
\ls (\sqrt{\E} + \E)\D
\end{equation}
and 
\begin{equation}
  \abs{\int_{-\ell}^\ell \sigma \p_z^2 \R(\p_1 \zeta_0,\p_1 \eta) (\p_1 \dt \eta)^2 \p_1 \dt^2 \eta } \ls  \sqrt{\E}\seb.
\end{equation}
\end{prop}
\begin{proof}
To handle $III$ we cannot get away with integrating by parts spatially (the resulting term needs too many dissipation terms at the endpoints since $\p_1 \dt \eta$ is only in $H^{1/2}$ in the energy).  Instead we pull out a time derivative: 
\begin{multline}
III = -  \int_{-\ell}^\ell \sigma \p_z^2 \R(\p_1 \zeta_0,\p_1 \eta) (\p_1 \dt \eta)^2 \p_1 \dt^3 \eta  =  -  \frac{d}{dt} \int_{-\ell}^\ell \sigma \p_z^2 \R(\p_1 \zeta_0,\p_1 \eta) (\p_1 \dt \eta)^2 \p_1 \dt^2 \eta \\
 + \int_{-\ell}^\ell \sigma \p_z^3 \R(\p_1 \zeta_0,\p_1 \eta) (\p_1 \dt \eta)^3 \p_1 \dt^2 \eta 
 + \int_{-\ell}^\ell \sigma \p_z^2 \R(\p_1 \zeta_0,\p_1 \eta) 2 \p_1 \dt \eta  \abs{\p_1 \dt^2 \eta}^2.
\end{multline}
We then estimate
\begin{multline}
 \abs{ \int_{-\ell}^\ell \sigma \p_z^3 \R(\p_1 \zeta_0,\p_1 \eta) (\p_1 \dt \eta)^3 \p_1 \dt^2 \eta } \ls  \norm{\p_1 \dt \eta}_{L^4}^3 \norm{\p_1 \dt^2 \eta}_{L^4}   \ls  \norm{\p_1 \dt \eta}_{1/2}^3  \norm{\p_1 \dt^2 \eta}_{1/2} 
  \\  \ls  \norm{ \dt \eta}_{3/2}^3  \norm{ \dt^2 \eta}_{3/2} 
  \ls  \E \sqrt{\D} \sqrt{\D} = \E \D 
\end{multline}
and
\begin{equation}
 \abs{\int_{-\ell}^\ell \sigma \p_z^2 \R(\p_1 \zeta_0,\p_1 \eta) 2 \p_1 \dt \eta  \abs{\p_1 \dt^2 \eta}^2} \ls
\norm{\p_1 \dt \eta}_{L^3} \norm{\p_1 \dt^2 \eta}_{L^3}^2 
 \ls  
\norm{\dt \eta}_{3/2} \ns{\dt^2 \eta}_{3/2} \ls \sqrt{\E} \D.
\end{equation}

Finally, we want to show that the term with the time derivative is actually controlled only by the energy.  To this end we first note that the Sobolev embeddings and interpolation imply that if $\psi \in H^{3/2}((-\ell,\ell))$ then 
\begin{equation}
 \norm{\p_1 \psi}_{L^4} \ls \norm{\p_1 \psi}_{H^{1/4}} \ls \norm{\psi}_{H^{5/4}} \ls \norm{\psi}_{H^1}^{1/2} \norm{\psi}_{H^{3/2}}^{1/2}.
\end{equation}
Using this, we may estimate
\begin{equation}
 \abs{\int_{-\ell}^\ell \sigma \p_z^2 \R(\p_1 \zeta_0,\p_1 \eta) (\p_1 \dt \eta)^2 \p_1 \dt^2 \eta } \ls \ns{\p_1 \dt \eta}_{L^4} \norm{\p_1 \dt^2 \eta}_{L^2} \ls \norm{\dt \eta}_{3/2} \norm{\dt \eta}_{1}  \norm{\dt^2 \eta}_{1} \ls \sqrt{\E} \seb. 
\end{equation}

\end{proof}

Finally, we handle the term $IV$.

\begin{prop}\label{ee_f3_IV}
Let $II$ be as given in \eqref{ee_f3_decomp}.  Then
\begin{equation}\label{ee_f3_IV_0}
 \abs{IV} \ls (\E + \E^{3/2})\D.
\end{equation}
\end{prop}
\begin{proof}

We now use the expression for $F^6= -2 \dt u_1 \p_1 \dt \eta - u_1 \p_1 \dt^2 \eta.$ to write 
\begin{multline}
IV =   \int_{-\ell}^\ell \sigma \p_z^2 \R(\p_1 \zeta_0,\p_1 \eta) (\p_1 \dt \eta)^2    \p_1 (2 \dt u_1 \p_1 \dt \eta) \\
+  \int_{-\ell}^\ell \sigma \p_z^2 \R(\p_1 \zeta_0,\p_1 \eta) (\p_1 \dt \eta)^2  \p_1(u_1 \p_1 \dt^2 \eta) = IV_1 + IV_2.
\end{multline}
We first argue as with $II_1$ and $II_2$ in Proposition \ref{ee_f3_II} (using the same $p,q$) to estimate 
\begin{multline}\label{ee_f3_IV_1}
 \abs{IV_1} = \abs{ \int_{-\ell}^\ell \sigma \p_z^2 \R(\p_1 \zeta_0,\p_1 \eta) (\p_1 \dt \eta)^2   2 \p_1  \dt u_1 \p_1 \dt \eta +  \int_{-\ell}^\ell \sigma \p_z^2 \R(\p_1 \zeta_0,\p_1 \eta) (\p_1 \dt \eta)^2    2 \dt u_1 \p_1^2 \dt \eta }
 \\
\ls \norm{\p_1 \dt \eta}_{L^\infty} \norm{\p_1 \dt \eta }_{L^{q}} 
\left( \norm{\p_1 \dt u}_{L^{p}(\Sigma)}  \norm{\p_1 \dt \eta}_{L^{q}} 
+   \norm{\dt u}_{L^{q}(\Sigma)}  \norm{\p_1^2 \dt \eta}_{L^{p}} \right) \\
\ls \norm{\p_1 \dt \eta}_{s-1/2} \norm{\p_1 \dt \eta }_{1/2} \left( \norm{\dt u}_{W^2_\delta}  \norm{\p_1 \dt \eta}_{1/2} +   \norm{\dt u}_{H^{1/2}(\Sigma)}  \norm{\dt \eta}_{W^{5/2}_\delta} \right)
 \\
\ls \norm{\dt \eta}_{s+1/2} \norm{ \dt \eta }_{3/2}\left( \norm{\dt u}_{W^2_\delta}  \norm{\dt \eta}_{3/2} +   \norm{\dt u}_{1}  \norm{\dt \eta}_{W^{5/2}_\delta} \right) \\
\ls \sqrt{\D} \sqrt{\E} \left(\sqrt{\D} \sqrt{\E} + \sqrt{\E} \sqrt{\D}   \right) \ls \E \D.
\end{multline}
Next we expand
\begin{multline}
IV_2 = \int_{-\ell}^\ell \sigma \p_z^2 \R(\p_1 \zeta_0,\p_1 \eta) (\p_1 \dt \eta)^2  \p_1 u_1 \p_1 \dt^2 \eta 
+ \int_{-\ell}^\ell \sigma \p_z^2 \R(\p_1 \zeta_0,\p_1 \eta) (\p_1 \dt \eta)^2  u_1 \p_1^2 \dt^2 \eta  =: IV_{3} + IV_4.
\end{multline}

To handle the term $IV_3$ we let $p=(3+\delta)/(2+2\delta)$ as above but now choose $r = (9+3\delta)/(1-\delta)$ so that $1/p + 3/r =1$, which allows us to estimate
\begin{equation}\label{ee_f3_IV_3}
 \abs{IV_3} \ls  \ns{\p_1 \dt \eta }_{L^{r}} 
 \norm{\p_1 u}_{L^{p}(\Sigma)}  \norm{\p_1 \dt^2 \eta}_{L^{r}}  
\ls  \ns{ \dt \eta }_{3/2}  \norm{u}_{W^2_\delta}  \norm{\dt^2 \eta}_{3/2} \ls \E \sqrt{\D} \sqrt{\D} = \E \D. 
\end{equation}

For the term $IV_4$ we first use the fact that  $u_1 =0$ at the endpoints to write
\begin{multline}
  IV_4 = -\int_{-\ell}^\ell \sigma \left[ \p_z^2 \R(\p_1 \zeta_0,\p_1 \eta) (\p_1 \dt \eta)^2  \p_1 u_1 + 2 \p_z^2 \R(\p_1 \zeta_0,\p_1 \eta) \p_1 \dt \eta \p_1^2 \dt \eta u_1 \right] \p_1 \dt^2 \eta  \\
-\int_{-\ell}^\ell \sigma \left[ \p_z^3 \R(\p_1 \zeta_0,\p_1 \eta) \p_1^2 \eta  + \p_y \p_z^2 \R(\p_1 \zeta_0,\p_1 \eta)  \p_1^2 \zeta_0 \right] (\p_1 \dt \eta)^2  u_1   \p_1 \dt^2 \eta =: IV_5 + IV_6.
\end{multline}
Then with $p$ and $r$ as used above for $IV_3$ we bound
\begin{multline}\label{ee_f3_IV_5}
 \abs{IV_5} \ls \left( \ns{\p_1 \dt \eta}_{L^r} \norm{\p_1 u_1}_{L^p(\Sigma)}  + \norm{\p_1 \dt \eta}_{L^r} \norm{\p_1^2 \dt \eta}_{L^p} \norm{u}_{L^r(\Sigma)}  \right) \norm{\p_1 \dt^2 \eta}_{L^r} \\
\ls \left( \ns{\dt \eta}_{3/2} \norm{u}_{W^2_\delta}  + \norm{\dt \eta}_{3/2} \norm{\dt \eta}_{W^{5/2}_\delta} \norm{u}_{s}  \right) \norm{\dt^2 \eta}_{3/2} \ls \left(\E \sqrt{\D} + \sqrt{\E} \sqrt{\D} \sqrt{\E} \right) \sqrt{\D} \ls \E \D,
\end{multline}
and
\begin{multline}\label{ee_f3_IV_6}
 \abs{IV_6} \ls \left(\norm{\p_1^2 \eta}_{L^p} + 1 \right) \ns{\p_1 \dt \eta}_{L^r} \norm{u}_{L^\infty} \norm{\p_1 \dt^2 \eta}_{L^r} \ls \left(\norm{ \eta}_{W^{5/2}_\delta} + 1 \right) \ns{ \dt \eta}_{3/2} \norm{u}_{s} \norm{ \dt^2 \eta}_{3/2} \\
\ls  \sqrt{\D}  \E \sqrt{\E} \sqrt{\D} + \E \sqrt{\D} \sqrt{\D} = (\E + \E^{3/2}) \D . 
\end{multline}

The estimate \eqref{ee_f3_IV_0} then follows by combining \eqref{ee_f3_IV_1}, \eqref{ee_f3_IV_3}, \eqref{ee_f3_IV_5}, and \eqref{ee_f3_IV_6}.

\end{proof}

Now that we have controlled $I$--$IV$ in \eqref{ee_f3_decomp} we can record a unified estimate.

\begin{thm}\label{ee_f3_dt2}
Let $F^3$ be given by \eqref{dt2_f3}.  Then 
\begin{multline}
\abs{ -  \int_{-\ell}^\ell \sigma  F^3   \p_1 (\dt^2 u \cdot \N) +  \frac{d}{dt} \int_{-\ell}^\ell \left[ \sigma \p_z \R(\p_1 \zeta_0,\p_1 \eta)  \frac{\abs{\p_1 \dt^2 \eta}^2}{2} + \sigma \p_z^2 \R(\p_1 \zeta_0,\p_1 \eta) (\p_1 \dt \eta)^2 \p_1 \dt^2 \eta \right] } 
\\
\ls (\sqrt{\E} + \E + \E^{3/2})\D.
\end{multline}
Moreover, 
\begin{equation}
 \abs{\int_{-\ell}^\ell \left[ \sigma \p_z \R(\p_1 \zeta_0,\p_1 \eta)  \frac{\abs{\p_1 \dt^2 \eta}^2}{2} + \sigma \p_z^2 \R(\p_1 \zeta_0,\p_1 \eta) (\p_1 \dt \eta)^2 \p_1 \dt^2 \eta \right]} \ls \sqrt{\E} \seb.
\end{equation}
\end{thm}
\begin{proof}
 We simply combine \eqref{ee_f3_decomp} with Propositions \ref{ee_f3_I}, \ref{ee_f3_II}, \ref{ee_f3_III}, and , \ref{ee_f3_IV}.
\end{proof}

A similar, but simpler, result holds for the once time-differentiated problem.  We will record it without proof.

\begin{thm}\label{ee_f3_dt}
Let $F^3$ be given by \eqref{dt1_f3}.  Then 
\begin{equation}
\abs{ -  \int_{-\ell}^\ell \sigma  F^3   \p_1 (\dt^2 u \cdot \N)  } 
\ls (\sqrt{\E} + \E)\D.
\end{equation}
\end{thm}

\subsection{Special functional estimates: free surface term}

On the right side of \eqref{linear_energy_0} we encounter the terms
\begin{equation}
   \int_{-\ell}^\ell   g \dt^2 \eta F^6 +  \sigma \frac{\p_1 \dt^2 \eta \p_1 F^6}{(1+\abs{\p_1 \zeta_0}^2)^{3/2}}  \text{ and }    \int_{-\ell}^\ell   g \dt \eta F^6 +  \sigma \frac{\p_1 \dt \eta \p_1 F^6}{(1+\abs{\p_1 \zeta_0}^2)^{3/2}} 
\end{equation}
where $F^6$ is given by \eqref{dt2_f6} for the first term and \eqref{dt1_f6} for the second term.  We have the following estimate.

\begin{thm}\label{ee_f6}
We have the estimate
\begin{equation}\label{ee_f6_01}
 \abs{   \int_{-\ell}^\ell   g \dt^2 \eta F^6 +  \sigma \frac{\p_1 \dt^2 \eta \p_1 F^6}{(1+\abs{\p_1 \zeta_0}^2)^{3/2}}} \ls \sqrt{\E} \D
\end{equation}
when $F^6$ is given by \eqref{dt2_f6}, and we have the estimate
\begin{equation}\label{ee_f6_02}
 \abs{   \int_{-\ell}^\ell   g \dt \eta F^6 +  \sigma \frac{\p_1 \dt \eta \p_1 F^6}{(1+\abs{\p_1 \zeta_0}^2)^{3/2}}} \ls \sqrt{\E} \D
\end{equation}
when $F^6$ is given by \eqref{dt1_f6}.
\end{thm}
\begin{proof}
We will again only prove the result in the more difficult case, which corresponds to \eqref{ee_f6_01}.  The estimate \eqref{ee_f6_02} follows from a similar argument.  In the first case we may write
\begin{equation}
 F^6 =  -2 \dt u_1 \p_1 \dt \eta - u_1 \p_1 \dt^2 \eta
\end{equation}
and
\begin{equation}
   \int_{-\ell}^\ell   g \dt^2 \eta F^6 +  \sigma \frac{\p_1 \dt^2 \eta \p_1 F^6}{(1+\abs{\p_1 \zeta_0}^2)^{3/2}} = I + II.
\end{equation}
We will estimate $I$ and $II$ separately; combining these then leads to the bound \eqref{ee_f6_01}.

We estimate the term $I$ via
\begin{multline}
 \abs{I} \ls \norm{\dt^2 \eta}_{L^3} \left( \norm{\dt u_1}_{L^3(\Sigma)} \norm{\p_1 \dt \eta}_{L^3} +  \norm{u_1}_{L^3(\Sigma)} \norm{\p_1 \dt^2 \eta}_{L^3} \right) \\
 \ls \norm{\dt^2 \eta}_{1/2} \left( \norm{\dt u_1}_{H^{1/2}(\Sigma)} \norm{\p_1 \dt \eta}_{1/2} +  \norm{u_1}_{H^{1/2}(\Sigma)} \norm{\p_1 \dt^2 \eta}_{1/2} \right) \\
 \ls \norm{\dt^2 \eta}_{1/2} \left( \norm{\dt u_1}_{1} \norm{\dt \eta}_{3/2} +  \norm{u_1}_{1} \norm{\dt^2 \eta}_{3/2} \right) \ls \sqrt{\E}(\sqrt{\D} \sqrt{\D} + \sqrt{\D} \sqrt{D}) = \sqrt{\E} \D.
\end{multline}

To estimate the term $II$ we let $p$ and $q$  be given by 
\begin{equation}
 p = \frac{3+\delta}{2+2\delta} \text{ and }  q = \frac{6+2\delta}{1-\delta}
\end{equation}
which implies that 
\begin{equation}
  \frac{1}{p} + \frac{2}{q} = 1.
\end{equation}
We then expand
\begin{equation}
 II =    -\int_{-\ell}^\ell   \sigma \frac{\p_1 \dt^2 \eta  }{(1+\abs{\p_1 \zeta_0}^2)^{3/2}} 2\p_1 \dt u_1 \p_1 \dt \eta - \int_{-\ell}^\ell   \sigma \frac{\p_1 \dt^2 \eta  }{(1+\abs{\p_1 \zeta_0}^2)^{3/2}} u_1 \p_1^2 \dt^2 \eta = II_1 + II_2.
\end{equation}
We estimate $II_1$ via 
\begin{multline}
 \abs{II_1} \ls \norm{\p_1 \dt^2 \eta}_{L^q} \norm{\nab \dt u}_{L^p(\Sigma)} \norm{\p_1 \dt \eta}_{L^q} \ls \norm{\dt^2 \eta}_{3/2} \norm{\dt u}_{W^2_\delta} \norm{\dt \eta}_{3/2} \ls \sqrt{\D} \sqrt{\D} \sqrt{\E} = \sqrt{\E} \D. 
\end{multline}

Then since $u_1$ vanishes at the endpoints we have that
\begin{equation}
 II_2 = - \int_{-\ell}^\ell   \sigma \frac{u_1  }{(1+\abs{\p_1 \zeta_0}^2)^{3/2}} \p_1 \frac{\abs{ \p_1 \dt^2 \eta }^2}{2} = \int_{-\ell}^\ell   \sigma \p_1 \left(\frac{u_1  }{(1+\abs{\p_1 \zeta_0}^2)^{3/2}} \right)  \frac{\abs{ \p_1 \dt^2 \eta }^2}{2} 
\end{equation}
and so
\begin{equation}
 \abs{II_2}\ls \left(\norm{u}_{L^{p}(\Sigma)} +  \norm{\nab  u}_{L^{p}(\Sigma)} \right) \ns{\p_1 \dt^2 \eta}_{L^{q}} \ls \norm{u}_{W^2_\delta}  \ns{\p_1 \dt^2 \eta}_{1/2} 
\ls  \norm{u}_{W^2_\delta} \ns{\dt^2 \eta}_{3/2} \ls \sqrt{\E} \D.
\end{equation}

\end{proof}

\subsection{Special functional estimates: $F^7$ term}

On the right side of \eqref{linear_energy_0} we also encounter the term $ [v\cdot \N,F^7]_\ell.$ We estimate this now.

\begin{thm}\label{ee_f7}
We have the estimate
\begin{equation}\label{ee_f7_01}
\abs{ [\dt^2 u\cdot \N,F^7]_\ell} \ls \sqrt{\E} \D
\end{equation}
when $F^7$ is given by \eqref{dt2_f7}, and we have the estimate
\begin{equation}\label{ee_f7_02}
\abs{ [\dt u \cdot \N,F^7]_\ell} \ls \sqrt{\E} \D
\end{equation}
when $F^7$ is given by \eqref{dt1_f7}.
\end{thm}
\begin{proof}
We will only prove \eqref{ee_f7_01}.  In this case we bound 
\begin{equation}
 \abs{F^7} \ls \abs{\swh'(\dt \eta)} \abs{\dt^3 \eta} + \abs{\swh''(\dt \eta)} \abs{\dt^2 \eta}^2.
\end{equation}
Since $S =  \norm{\dt \eta}_{C^0} \ls \norm{\dt \eta}_1 \ls \E \ls 1,$ we may estimate 
\begin{equation}
 \abs{\swh'(z)} = \frac{1}{\kappa}\abs{ \int_0^z \sw''(r)dr}\ls \abs{z} \text{ for } z \in [-S,S].
\end{equation}
Then we may use the equations for $\dt^j \eta$ and trace theory to bound
\begin{equation}
 \abs{F^7} \ls \abs{\dt \eta} \abs{\dt^3 \eta} + \abs{ \dt^2 \eta}^2 \ls \norm{\dt \eta}_1 \abs{\dt^2 u \cdot \N} + \norm{\dt^2 \eta}_1 \abs{\dt u \cdot \N}. 
\end{equation}
Consequently, 
\begin{equation}
 \abs{  [\dt^2 u\cdot \N,F^7]_\ell } \ls [\dt^2 u \cdot \N]_\ell \left(\norm{\dt \eta}_1 \abs{\dt^2 u \cdot \N} + \norm{\dt^2 \eta}_1 \abs{\dt u \cdot \N} \right)  \ls \sqrt{\E} \D.
\end{equation}
\end{proof}

\subsection{Special zeroth order terms}

Here we record a couple simple estimates that we will use in conjunction with Corollary \ref{basic_energy}.  We begin with the $\Q$ term.

\begin{thm}\label{ee_Q}
 Let $\Q(y,z)$ be the smooth function defined by \eqref{Q_def}.  Then 
\begin{equation}
 \abs{\int_{-\ell}^\ell \sigma \Q(\p_1 \zeta_0, \p_1 \eta)  } \ls \norm{\eta}_{s+1/2} \ns{\eta}_1 \ls \sqrt{\E} \ns{\eta}_1
\end{equation}
\end{thm}
\begin{proof}
According to Proposition \ref{R_prop} we have that 
\begin{equation}
 \abs{\int_{-\ell}^\ell \sigma \Q(\p_1 \zeta_0, \p_1 \eta)  } \ls \int_{-\ell}^\ell \abs{\p_1 \eta}^3 \ls \norm{\p_1 \eta}_{L^\infty} \ns{\eta}_1 \ls \norm{\p_1 \eta}_{s-1/2} \ns{\eta}_1 \ls \norm{\eta}_{s+1/2} \ns{\eta}_1.
\end{equation}
This implies the desired estimate.
\end{proof}

Next we handle the $\sw$ term appearing in Corollary \ref{basic_energy}.

\begin{thm}\label{ee_W}
We have that 
\begin{equation}\label{ee_W_0}
\abs{  [u\cdot \N,\swh(\dt \eta)]_\ell } \ls \norm{\dt \eta}_{1} \bs{u\cdot \N}.
\end{equation}
\end{thm}
\begin{proof}
The definition of $\swh \in C^2$ in \eqref{V_pert} easily shows that $\abs{\swh(z)} \ls z^2$.  Since $\dt \eta = u\cdot \N$ at $\pm \ell$ we have that  
\begin{equation}
 \abs{  [u\cdot \N,\swh(\dt \eta)]_\ell } \ls \sum_{a=\pm 1} \kappa \abs{u\cdot \N(a\ell,t)}^2 \abs{\dt \eta(a \ell,t)}.  
\end{equation}
The estimate \eqref{ee_W} then follows from this and the $1-D$ trace estimate $\abs{\dt \eta (\pm \ell,t)} \ls \norm{\dt \eta}_1$.
\end{proof}

\section{Terms in the elliptic estimates }\label{sec_nlin_ell}

Our scheme of a priori estimates will employ the elliptic estimates of Theorem \ref{A_stokes_stress_solve}.  In order for this to be useful we must estimate the various terms appearing on the right side of \eqref{A_stokes_stress_0} when the $G^i$ terms are determined by the once temporally-differentiated problem and by the non-differentiated problem,  The former is far more delicate, and so we will focus our efforts on these.  The latter can be handled with similar and simpler arguments and are thus omitted.

Throughout this entire section we will assume that  $\omega \in (0,\pi)$ is the angle formed by $\zeta_0$ at the corners of $\Omega$,  $\delta_\omega \in [0,1)$ is given by \eqref{crit_wt}, and $\delta \in (\delta_\omega,1)$.  This determines explicitly the choice of $\delta$ appearing in the definitions of $\E$ and $\D$ in \eqref{ed_def_1}--\eqref{ed_def_5}.  We will assume throughout the entirety of this section that $\eta$ is given and satisfies 
\begin{equation} 
 \sup_{0 \le t \le T} \left(  \seb(t) +  \ns{\eta(t)}_{W^{5/2}_\delta(\Omega)} + \ns{\dt \eta(t)}_{H^{3/2}((-\ell,\ell))}  \right) \le \gamma < 1,
\end{equation}
where $\gamma \in (0,1)$ is as in Lemma \ref{eta_small}. For the sake of brevity we will not explicitly state this in each result's hypotheses.

\subsection{The time differentiated problem }

We want to apply Theorem \ref{A_stokes_stress_solve} to the time differentiated problem.  In this case we have
\begin{equation}
 G^1 = F^1, G^2 = F^2, G^3_- = 0, G^3_+ = \dt^2 \eta - F^6,
\end{equation}
and
\begin{equation}
 G^4_- = F^5, G^4_+ = F^4 \cdot \mathcal{T} /\abs{\mathcal{T}}^2, G^5 = F^4 \cdot \N /\abs{\N}^2, G^6 = F^3,
 G^7 = \kappa (\dt^2 \eta + F^7) \pm \sigma F^3,
\end{equation}
where the $F^i$ terms are given in Appendix \ref{fi_dt1}.  Theorem \ref{A_stokes_stress_solve} then dictates that we must control
\begin{multline}
\ns{F^1}_{W^0_{\delta} } +  \ns{F^2}_{W^1_{\delta}} + \ns{\dt^2 \eta - F^6}_{W^{3/2}_{\delta} }  
 +  \ns{F^5}_{W^{1/2}_{\delta}} \\
 +  \ns{F^4 \cdot \mathcal{T} /\abs{\mathcal{T}}^2}_{W^{1/2}_{\delta}}
 + \ns{F^4 \cdot \N /\abs{\N}^2}_{W^{1/2}_{\delta}} + \ns{\p_1 F^3}_{W^{1/2}_{\delta}} + [ \kappa \dt^2 \eta + \kappa F^7 \pm \sigma F^3]_\ell^2.
\end{multline}

We begin by estimating the $F^1$ term.

\begin{prop}\label{we_f1}
Let $F^1$ be given by \eqref{dt1_f1}. We have the estimate 
\begin{equation}
 \ns{F^1}_{W^0_{\delta}} \ls \D( \E + \E^2).
\end{equation}
\end{prop}
\begin{proof}
We have that 
\begin{equation}
 F^1 = - \diverge_{\dt \A} S_\A(p,u) + \mu \diva \sg_{\dt \A} u,
\end{equation}
and we will estimate each term separately.

For the first we choose $q \in [1,\infty)$ such that $2/q + (2-s)/2 =1/2$ in order to estimate
\begin{multline}
 \ns{- \diverge_{\dt \A} S_\A(p,u)}_{W^0_{\delta}} 
 \ls \ns{\dt \A}_{L^\infty} \ns{\nabla p}_{W^0_{\delta}} + \ns{\dt \A}_{L^\infty} \ns{\A}_{L^\infty} \ns{u}_{W^2_{\delta}} \\
 + \ns{\dt \A}_{L^{q} } \ns{\nab \A}_{L^{2/(2-s)}} \ns{d^\delta \nab u}_{L^{q}} 
 \ls \ns{\dt \eta}_{s+1/2} \ns{p}_{W^1_{\delta}} + \ns{\dt \eta}_{s+1/2}  \ns{u}_{W^2_{\delta}} \\
 +  \ns{\dt \eta}_{3/2} \ns{\eta}_{s+1/2} \ns{u}_{W^2_{\delta}}   
 \ls \D \E + \D \E  + \D \E^2 \ls \D(\E + \E^2).
\end{multline}

For the second estimate
\begin{multline}
 \ns{\mu \diva \sg_{\dt \A} u}_{W^0_{\delta}} \ls \ns{\dt \nab \A}_{L^{2/(2-s)}} \ns{d^\delta \nab u}_{L^{2/(s-1)}} + \ns{\dt \A}_{L^\infty} \ns{u}_{W^2_{\delta}} \\ 
\ls \ns{\dt \eta}_{s+1/2} \ns{u}_{W^2_\delta} + \ns{\dt \eta}_{s+1/2} \ns{u}_{W^2_{\delta}} \ls \D \E + \D \E \ls \D \E.
\end{multline}

\end{proof}

Next we estimate the $F^2$ term.

\begin{prop}\label{we_f2}
Let $F^2$ be given by \eqref{dt1_f2}.  We have that 
\begin{equation}
 \ns{F^2}_{W^1_{\delta}} \ls \E \D.
\end{equation}
\end{prop}
\begin{proof}
Since $ F^2 = -\diverge_{\dt \A} u$ we only have one term to estimate.  We bound
\begin{multline}
 \ns{-\diverge_{\dt \A} u}_{W^1_{\delta}} \ls \ns{\dt \A}_{L^4} \ns{d^\delta \nab u}_{L^4} + \ns{\dt \nabla \A}_{L^{2/(2-s)}} \ns{d^\delta \nab u}_{L^{2/(s-1)}} + \ns{\dt \A}_{L^\infty} \ns{\nab^2 u}_{W^0_{\delta}} \\
\ls \ns{\dt \eta}_{3/2} \ns{u}_{W^2_\delta} + \ns{\dt \eta}_{s+1/2} \ns{u}_{W^2_\delta} + \ns{\dt \eta}_{s+1/2} \ns{u}_{W^2_{\delta}} \ls \D \E + \D \E + \D \E \ls \E \D. 
\end{multline}
\end{proof}

Next we estimate the first $F^3$ term.

\begin{prop}\label{we_f3_top}
Let $F^3$ be given by \eqref{dt1_f3}.  We have that
\begin{equation}
 \ns{\p_1 F^3}_{W^{1/2}_{\delta}} \ls \D \E.
\end{equation}
\end{prop}
\begin{proof}
We have that 
\begin{equation}
 \p_1 F^3 = \p_y \p_z \R(\p_1 \zeta_0,\p_1 \eta) \p_1^2 \zeta_0 \p_1 \dt \eta + \p_z^2 \R(\p_1 \zeta_0,\p_1 \eta) \p_1^2 \eta \p_1 \dt \eta + \p_z  \R(\p_1 \zeta_0,\p_1 \eta) \p_1^2 \dt \eta =: I + II + III.
\end{equation}
To control these terms we will use Proposition \ref{weight_prod_half}, which is possible because $s-1/2 > 1/2$.  We estimate 
\begin{equation}
 \ns{I}_{W^{1/2}_{\delta}} \ls \ns{\p_1 \dt  \eta}_{W^{1/2}_{\delta}}   \ns{\p_y \p_z \R(\p_1 \zeta_0,\p_1 \eta) \p_1^2 \zeta_0}_{s-1/2} \ls \D \E,
\end{equation}
\begin{equation}
 \ns{II}_{W^{1/2}_{\delta}} \ls \ns{\p_1^2 \eta}_{W^{1/2}_{\delta}}   \ns{\p_z^2 \R(\p_1 \zeta_0,\p_1 \eta) \p_1 \dt \eta}_{s-1/2} \ls \E \D,
\end{equation}
and
\begin{equation}
 \ns{III}_{W^{1/2}_{\delta}} \ls  \ns{\p_1^2 \dt \eta }_{W^{1/2}_{\delta}} \ns{\p_z \R(\p_1 \zeta_0,\p_1 \eta)}_{s-3/2} \ls \D \E.
\end{equation}

\end{proof}

Next we estimate the term with $F^3$ and $F^7$ at the endpoints.

\begin{prop}\label{we_f3_end}
Let $F^3$ be given by \eqref{dt1_f3} and $F^7$ be given by \eqref{dt1_f7}.  We have the estimate
\begin{equation}
 [\kappa \dt^2 \eta + \kappa F^7 \pm \sigma F^3]_\ell^2 \ls \sdb+ \D \E.
\end{equation}
\end{prop}
\begin{proof}
We automatically have
\begin{equation}
 \bs{\dt^2 \eta} = \bs{\dt u \cdot \N} \le \sdb.
\end{equation}
Next we control
\begin{equation}
 F^3 = \p_z \R(\p_1 \zeta_0,\p_1 \eta) \dt \p_1 \eta 
\end{equation}
at $\pm \ell$.  Proposition \ref{R_prop} implies that $\abs{\p_z \R(\p_1 \zeta_0,\p_1 \eta)} \ls \abs{\p_1 \eta}$, so we reduce to controlling $\abs{\p_1 \eta \dt \p_1 \eta}$ at the endpoints.  We estimate
\begin{equation}
 \abs{\p_1 \eta(\pm \ell,t)}^2 \ls \ns{\p_1 \eta}_{s-1/2} \ls \E
 \text{ and } 
 \abs{\dt \p_1 \eta(\pm \ell,t)}^2 \ls \ns{\dt \p_1 \eta}_{s-1/2} \ls \D.
\end{equation}
Thus 
\begin{equation}
 \bs{\pm \sigma F^3} \ls \E \D.
\end{equation}
Finally, we turn to the $F^7$ term.  In this case we may argue as in Theorem \ref{ee_f7} to bound $\abs{F^7} \ls \abs{\dt \eta} \abs{\dt^2 \eta}$, from which we deduce that 
\begin{equation}
 \bs{F^7} \ls \ns{\dt \eta}_1 \bs{\dt u \cdot \N} \ls \E \D.
\end{equation}

\end{proof}

Next we handle the $F^4$ term.

\begin{prop}\label{we_f4}
Let $F^4$ be given by \eqref{dt1_f4}. We have that 
\begin{equation}
 \ns{F^4}_{W^{1/2}_{\delta}}  \ls \D(\E + \E^2).
\end{equation}
\end{prop}
\begin{proof}
 
We have that 
\begin{multline}
  \ns{F^4 \cdot \mathcal{T} /\abs{\mathcal{T}}^2}_{W^{1/2}_{\delta}}
 + \ns{F^4 \cdot \N /\abs{\N}^2}_{W^{1/2}_{\delta}} \ls  \ns{F^4}_{W^{1/2}_{\delta}}(1 + \ns{\eta}_{C^1}) 
 \\
  \ls \ns{F^4}_{W^{1/2}_{\delta}}(1 + \ns{\eta}_{s+1/2}) \ls \ns{F^4}_{W^{1/2}_{\delta}}(1+ \E),
\end{multline}
and so it suffices to estimate $\ns{F^4}_{W^{1/2}_{\delta}}$.  It is written
\begin{equation}
  F^4 = \mu \sg_{\dt \A} u \N +\left[ g\eta  - \sigma \p_1 \left(\frac{\p_1 \eta}{(1+\abs{\p_1 \zeta_0}^2)^{3/2}} +  \R(\p_1 \zeta_0,\p_1 \eta) \right)  - S_{\A}(p,u) \right]  \dt \N.
\end{equation}
We will handle each term separately.

For the first we estimate 
\begin{multline}
 \ns{\mu \sg_{\dt \A} u \N}_{W^{1/2}_{\delta}} \ls  \ns{\sg_{\dt \A} u (e_2 - e_1 \p_1 \bar{\eta})}_{W^{1}_{\delta}} \ls \ns{u}_{W^{2}_{\delta}} \ns{\dt \A}_{L^\infty} \ns{e_2 - e_1 \p_1 \bar{\eta}}_{L^\infty} \\
 +  \ns{d^\delta \nab u}_{L^{2/(s-1)}} \left( \ns{\dt \nab \A}_{L^{2/(2-s)}} \ns{e_2 - e_1 \p_1 \bar{\eta}}_{L^\infty} +  \ns{\dt  \A}_{L^\infty} \ns{\nab^2 \bar{\eta}}_{L^{2/(2-s)}}  \right)
 \\
\ls \ns{u}_{W^{2}_{\delta}} \ns{\dt \eta}_{2} + \ns{u}_{W^{2}_{\delta}}  \left( \ns{\dt \eta}_{s+1/2} + \ns{\dt \eta}_{s+1/2} \ns{\eta}_{s+1/2}  \right) \ls \E \D + \E(\D + \D \E) \ls \D(\E + \E^2).
\end{multline}
For the second we  estimate 
\begin{multline}
 \ns{-S_{\A}(p,u) \dt \N}_{W^{1/2}_{\delta}}  \ls  \ns{S_{\A}(p,u) \dt \p_1 \bar{\eta}}_{W^{1}_{\delta}(\Omega)} \ls \left( \ns{p}_{W^1_{\delta}} + \ns{u}_{W^2_{\delta}}   \right) \ns{\dt \p_1 \bar{\eta} }_{L^\infty} \\
 + \left( \ns{d^\delta p}_{L^{2/(s-1)} } + \ns{d^\delta \nab u}_{L^{2/(s-1)} }   \right) \left( \ns{ \dt \nab^2 \bar{\eta}}_{L^{2/(s-2)}} + \ns{\nab \A}_{L^{2/(s-2)}} \ns{\dt \p_1 \bar{\eta}}_{L^\infty}  \right)
 \\
\ls \E \ns{\dt \eta}_{s+1/2} + \E\left( \ns{\dt \eta}_{s+1/2} + \ns{\eta}_{s+1/2} \ns{\dt \eta}_{s+1/2}  \right) \ls \E \D + \E( \D + \E \D) \ls \D(\E + \E^2).
\end{multline}
For the remaining terms we use Proposition \ref{weight_prod_half} to bound
\begin{multline}
 \ns{ \left[ g\eta  - \sigma \p_1 \left(\frac{\p_1 \eta}{(1+\abs{\p_1 \zeta_0}^2)^{3/2}} +  \R(\p_1 \zeta_0,\p_1 \eta) \right)\right] \dt \N}_{W^{1/2}_{\delta}}  \ls  \ns{\eta}_{W^{5/2}_{\delta}}  \ns{\dt \p_1 \eta}_{s-1/2}
 \\
 \ls  \ns{\eta}_{W^{5/2}_{\delta}} \ns{\dt \eta}_{s+1/2} \ls \E \D,
\end{multline}
which is possible because $s-1/2 > 1/2$.  
 
\end{proof}

The $F^5$ terms is handled next.

\begin{prop}\label{we_f5}
Let $F^5$ be given by \eqref{dt1_f5}.  We have that 
\begin{equation}
 \ns{F^5}_{W^{1/2}_{\delta}} \ls \D\E.
\end{equation}
\end{prop}
\begin{proof}
The fact that $1 < s$ allows us to  use Proposition \ref{weight_prod_half} to estimate
\begin{equation}
 \ns{ \mu \sg_{\dt \A} u \nu \cdot \tau}_{W^{1/2}_{\delta}} \ls \ns{ \nab u}_{W^{1/2}_{\delta}} \ns{ \dt \A}_{s-1/2} \ls \ns{u}_{W^2_\delta} \ns{\dt \eta}_{s+1/2} \ls \E \D. 
\end{equation}

\end{proof}

Next we consider the $F^6$ term.

\begin{prop}\label{we_f6}
Let $F^6$ be given by \eqref{dt1_f6}. We have that 
\begin{equation}
 \ns{\dt^2 \eta - F^6}_{W^{3/2}_{\delta}} \ls \ns{\dt^2 \eta}_{3/2} + \E\D.
\end{equation}
\end{prop}
\begin{proof}
We have that $F^6 = u_1 \p_1 \dt \eta$.  We then use Proposition \ref{weight_prod_half} to estimate 
\begin{equation}
 \ns{\dt^2 \eta - F^6}_{W^{3/2}_{\delta}} \ls \ns{\dt^2 \eta}_{3/2} + \ns{u_1 \p_1 \dt \eta}_{W^{3/2}_{\delta}} \ls \ns{\dt^2 \eta}_{3/2} + \ns{u}_{W^2_{\delta}} \ns{\dt \eta}_{W^{5/2}_{\delta}} \ls \ns{\dt^2 \eta}_{3/2} + \E \D .
\end{equation}
\end{proof}

Finally, we combine the above propositions into a single estimate.

\begin{thm}\label{elliptic_est_dt}
Let $\omega \in (0,\pi)$ be the angle formed by $\zeta_0$ at the corners of $\Omega$,  $\delta_\omega \in [0,1)$ be given by \eqref{crit_wt}, and $\delta \in (\delta_\omega,1)$.  Suppose that  $\ns{\eta}_{W^{5/2}_{\delta} } < \gamma$, where $\gamma$ is as in Theorem \ref{A_stokes_om_iso}.  Then we have the inclusions $(\dt u,\dt p,\dt \eta) \in W^2_{\delta}(\Omega) \times  \oW^1_{\delta}(\Omega) \times W^{5/2}_{\delta}$ as well as the estimate
\begin{equation} 
 \ns{\dt u}_{W^2_{\delta}} + \ns{\dt p}_{\oW^1_{\delta}} +  \ns{\dt \eta}_{W^{5/2}_{\delta}} \ls  \sdb +  \D(\E + \E^2).
\end{equation}
\end{thm}
\begin{proof}
Propositions \ref{we_f1}--\ref{we_f6} guarantee that we may estimate all the terms appearing on the right side of \eqref{A_stokes_stress_0} by $\sdb + \D(\E +\E^2)$.  The result then follows from Theorem \ref{A_stokes_stress_solve}.
\end{proof}

\subsection{The problem without time derivatives}

We now want to apply Theorem \ref{A_stokes_stress_solve} to the basic problem without time derivatives.  In this case we have $G^1 =0$, $G^2 =0$, $G^3_-=0$, $G^3_+ = \dt \eta$, $G^4_\pm =0$, $G^5 =0$, $G^6 = \R(\p_1 \zeta_0,\p_1 \eta)$, and $G^7_\pm = \kappa \dt \eta(\pm \ell,t) + \kappa \swh(\dt \eta(\pm \ell,t)) \pm \R(\p_1 \zeta_0,\p_1 \eta)$.  Consequently we must estimate 
\begin{equation} 
 \ns{\dt \eta}_{W^{3/2}_{\delta}} + \ns{\p_1 [\R(\p_1 \zeta_0,\p_1 \eta)] }_{W^{1/2}_{\delta}} + [\kappa \dt \eta + + \kappa \swh(\dt \eta) \pm \R(\p_1 \zeta_0,\p_1 \eta)]_\ell^2 \ls \sdb + \D( \E + \E^2).
\end{equation}
The estimates of these terms follow from argument similar to those given for the time differentiated problem and are thus omitted.

\begin{thm}\label{elliptic_est_no_dt}
Let $\omega \in (0,\pi)$ be the angle formed by $\zeta_0$ at the corners of $\Omega$,  $\delta_\omega \in [0,1)$ be given by \eqref{crit_wt}, and $\delta \in (\delta_\omega,1)$.  Suppose that  $\ns{\eta}_{W^{5/2}_{\delta} } < \gamma$, where $\gamma$ is as in Theorem \ref{A_stokes_om_iso}.  Then we have the inclusions $(u,p,\eta) \in W^2_{\delta}(\Omega) \times  \oW^1_{\delta}(\Omega) \times W^{5/2}_{\delta}$ as well as the estimate
\begin{equation} 
 \ns{u}_{W^2_{\delta}} + \ns{ p}_{\oW^1_{\delta}} +  \ns{\eta}_{W^{5/2}_{\delta}} \ls  
 \sdb + \D(\E + \E^2).
\end{equation}
\end{thm}

\section{Main results}\label{sec_apriori}

Here we record the main results of the paper: the a priori estimates for solutions to \eqref{geometric}, and our small data global well-posedness and decay result.

\subsection{A priori estimates }

In order to deduce our a priori estimates we must first introduce some variants the energy and dissipation.  We define 
\begin{equation}\label{fed_def_1}
 \mathfrak{E} = \sum_{j=0}^2 \int_{-\ell}^\ell \frac{g}{2} \abs{\dt^j \eta}^2 + \frac{\sigma}{2(1+\abs{\p_1 \zeta_0}^2)^{3/2}} \abs{\p_1 \dt^j \eta}^2,
\end{equation}
\begin{equation}\label{fed_def_2}
 \mathfrak{D} = \sum_{j=0}^2 \left( \frac{\mu}{2} \int_\Omega \abs{\sga \dt^j u}^2 J 
+\int_{\Sigma_s} \beta J \abs{\dt^j u \cdot \tau}^2 + \bs{\dt^j u\cdot \N} \right)
\end{equation}
and 
\begin{equation}\label{fed_def_3}
 \mathfrak{F} =   \int_{-\ell}^\ell \left[\sigma \Q(\p_1 \zeta_0,\p_1 \eta) +  \sigma \p_z \R(\p_1 \zeta_0,\p_1 \eta)  \frac{\abs{\p_1 \dt^2 \eta}^2}{2} + \sigma \p_z^2 \R(\p_1 \zeta_0,\p_1 \eta) (\p_1 \dt \eta)^2 \p_1 \dt^2 \eta \right].
\end{equation}


These terms are clearly related to certain energy and dissipation terms that we have previously defined.  We state these relations now.

\begin{prop}\label{ap_fed_ests}
Let $\mathfrak{E},$ $\mathfrak{D},$ and $\mathfrak{F}$ be as defined in \eqref{fed_def_1}--\eqref{fed_def_3}.  There exists a universal constant $\gamma >0$ such that if 
\begin{equation}
 \sup_{0 \le t \le T} \E(t)  \le \gamma,
\end{equation}
then
\begin{equation}\label{ap_fed_ests_01} 
 \mathfrak{E} \ls \seb \ls \mathfrak{E} \text{ and } \mathfrak{D} \ls \sdbb \ls \mathfrak{D},
\end{equation}
where $\seb$ and $\sdbb$ are defined in \eqref{ed_def_1}--\eqref{ed_def_5}, and also
\begin{equation}\label{ap_fed_ests_02}
\abs{ \mathfrak{F} } \le \hal \mathfrak{E}. 
\end{equation}
\end{prop}
\begin{proof}
The estimates in \eqref{ap_fed_ests_01} follows easily from Lemma \ref{eta_small} if we assume that $\gamma$ is as small as stated there.  Theorems \ref{ee_f3_dt2} and \ref{ee_Q} guarantee that 
\begin{equation}
\abs{ \mathfrak{F} } \ls \sqrt{\E} \seb \ls \sqrt{\E} \mathfrak{E}. 
\end{equation}
Consequently, if we further restrict $\gamma$ then we must have that $\abs{\mathfrak{F}} \le \hal \mathfrak{E}$, which is \eqref{ap_fed_ests_02}.
\end{proof}

The reason we have introduced $\mathfrak{E}$, $\mathfrak{D}$, and $\mathfrak{F}$ is that they appear naturally in an  energy-dissipation inequality that we may derive from Theorem \ref{linear_energy}.   This inequality, which we now state, forms the core of our a priori estimates.

\begin{thm}\label{ap_diff_ineq}
Let $\omega \in (0,\pi)$ be the angle formed by $\zeta_0$ at the corners of $\Omega$,  $\delta_\omega \in [0,1)$ be given by \eqref{crit_wt}, and $\delta \in (\delta_\omega,1)$.  There exists a universal constant $\gamma >0$ such that if 
\begin{equation}\label{ap_diff_ineq_00}
 \sup_{0 \le t \le T} \E(t)  + \int_0^T \D(t) dt \le \gamma,
\end{equation} 
then there exists a universal constant $C >0$ such that
\begin{equation}\label{ap_diff_ineq_0}
 \frac{d}{dt} \left( \mathfrak{E} + \mathfrak{F}\right) +  C \D \le 0.
\end{equation} 
\end{thm}
\begin{proof}
Let $0<\gamma < 1$ be as small as in Proposition \ref{ap_fed_ests}, and hence as small as in Lemma \ref{eta_small}. 
 
To begin we apply Theorem \ref{linear_energy} twice: to the once and twice time differentiated problems.  This is possible thanks to \eqref{ap_diff_ineq_00}.  We sum the resulting inequalities with the result of Corollary \ref{basic_energy} and then apply the estimates of Theorems \ref{ee_v_est}, \ref{ee_p_est}, \ref{ee_f3_dt2}, \ref{ee_f3_dt},  \ref{ee_f6},  \ref{ee_f7}, and \ref{ee_W} to deduce that  
\begin{equation}\label{ap_diff_ineq_1}
 \frac{d}{dt} \left( \mathfrak{E}  + \mathfrak{F} \right) + \mathfrak{D} \ls \sqrt{\E} \D,
\end{equation}
where $\mathfrak{E}$, $\mathfrak{D}$, and $\mathfrak{F}$ are as defined by \eqref{fed_def_1}--\eqref{fed_def_3}.  We know from Proposition \ref{ap_fed_ests} that  
\begin{equation}\label{ap_diff_ineq_2}
\mathfrak{D} \ls \sdbb \ls \mathfrak{D}.
\end{equation}

Next we combine Theorem \ref{pressure_est} with the estimate of Theorem \ref{ee_v_est_pressure} to see that 
\begin{equation}
 \ns{p}_0 + \ns{\dt p}_0 +\ns{\dt^2 p}_0 \ls \sdbb + \sqrt{\E} \D.
\end{equation}
Similarly, Theorem \ref{xi_est} and the estimates of Theorems \ref{ee_v_est} and \ref{ee_f3_half} show that
\begin{equation}
 \ns{\eta}_{3/2} + \ns{\dt \eta}_{3/2} +\ns{\dt^2 \eta}_{3/2} \ls \sdbb + \sqrt{\E} \D.
\end{equation}
Consequently, we may combine with \eqref{ap_diff_ineq_2} to see that
\begin{equation}\label{ap_diff_ineq_4}
 \sdb \ls \mathfrak{D} + \sqrt{\E} \D.
\end{equation}

Theorems \ref{elliptic_est_dt} and \ref{elliptic_est_no_dt} then imply that if $\gamma \le \hat{\gamma}$ (where here we write $\hat{\gamma}$ for the smallness parameter used in the theorems) then
\begin{equation}
  \ns{u}_{W^2_{\delta}} + \ns{p}_{\oW^1_{\delta}} +  \ns{\eta}_{W^{5/2}_{\delta}} + 
  \ns{\dt u}_{W^2_{\delta}} + \ns{\dt p}_{\oW^1_{\delta}} +  \ns{\dt \eta}_{W^{5/2}_{\delta}} \ls  \sdb +  \D(\E + \E^2).
\end{equation}
Since $\dt^3 \eta = \dt^2 (u \cdot \N)$ we then also have that 
\begin{equation}
 \ns{\dt^3 \eta}_{W^{1/2}_\delta} \ls  \sdb +  \D(\E + \E^2).
\end{equation}
Combining these with \eqref{ap_diff_ineq_4} then shows that 
\begin{equation}\label{ap_diff_ineq_5}
 \D \ls \mathfrak{D} + \sqrt{\E} \D. 
\end{equation}

Plugging \eqref{ap_diff_ineq_5} into \eqref{ap_diff_ineq_1} then shows that 
\begin{equation}
 \frac{d}{dt} \left( \mathfrak{E} + \mathfrak{F}\right) + 2 C \D \ls \sqrt{\E} \D 
\end{equation}
for some universal constant $C>0$.  By further restricting $\gamma$ we may absorb the term on the right onto the left, which yields the estimate
\begin{equation}
 \frac{d}{dt} \left( \mathfrak{E} + \mathfrak{F}\right) +  C \D \le 0.
\end{equation} 
This is \eqref{ap_diff_ineq_0}. 
\end{proof}

With theorem \ref{ap_diff_ineq} we can now complete the proof of our a priori estimates.  These will consist of two estimates: a decay estimate and a higher-order bound.  We begin with the proof of the decay estimate.

\begin{thm}\label{ap_decay}
Let $\omega \in (0,\pi)$ be the angle formed by $\zeta_0$ at the corners of $\Omega$,  $\delta_\omega \in [0,1)$ be given by \eqref{crit_wt}, and $\delta \in (\delta_\omega,1)$.   There exists a universal constant $\gamma >0$ such that if 
\begin{equation}
 \sup_{0 \le t \le T} \E(t) + \int_0^T \D(t) dt \le \gamma,
\end{equation}
then there exists a universal constant $\lambda >0$ such that
\begin{equation}\label{ap_decay_0}
 \sup_{ 0 \le t \le T} e^{\lambda t} \left[ \seb(t) +  \ns{u(t)}_{1} + \ns{u(t)\cdot \tau}_{L^2(\Sigma_s)} + \bs{u\cdot \N(t)} +\ns{p(t)}_0\right]    \ls \seb(0).
\end{equation}
\end{thm}
\begin{proof}
 
Let $\gamma$ be as small as in Theorem \ref{ap_diff_ineq}.  The theorem then provides for the existence of a universal constant $C>0$ such that
\begin{equation}\label{ap_decay_1}
 \frac{d}{dt} \left( \mathfrak{E} + \mathfrak{F}\right) +  C \D \le 0.
\end{equation}
We know from Proposition \ref{ap_fed_ests} that  
\begin{equation}\label{ap_decay_2}
 \mathfrak{E} \ls \seb \ls \mathfrak{E} \text{ and } 0 \le  \hal \mathfrak{E} \le \mathfrak{E} + \mathfrak{F} \le \frac{3}{2} \mathfrak{E}.
\end{equation}
On the other hand, it's clear that 
\begin{equation}\label{ap_decay_3}
 \mathfrak{E} \ls \D.
\end{equation}

We may thus combine \eqref{ap_decay_1}, \eqref{ap_decay_2}, and \eqref{ap_decay_3} to deduce that there exists a universal constant $\lambda >0$ such that 
\begin{equation}
 \frac{d}{dt} \left( \mathfrak{E} + \mathfrak{F}\right) + \lambda \left( \mathfrak{E} + \mathfrak{F}\right) \le 0.
\end{equation}
Upon integrating this differential inequality we find that 
\begin{equation}
 \hal \mathfrak{E}(t) \le \mathfrak{E}(t) + \mathfrak{F}(t) \le e^{-\lambda t} \left(\mathfrak{E}(0) + \mathfrak{F}(0)\right) \le \frac{3}{2} e^{-\lambda t} \mathfrak{E}(0)
\end{equation}
for all $t \in [0,T].$  Thus \eqref{ap_decay_2} tells us that 
\begin{equation}\label{ap_decay_6}
 \sup_{ 0 \le t \le T} e^{\lambda t} \seb(t) \ls \seb(0).
\end{equation}

To complete the proof of \eqref{ap_decay_0} we first use Lemma \ref{geometric_evolution} on \eqref{geometric}, which means that $F^i=0$ except for $i=3,7$, in which case $F^3 = \R(\p_1 \zeta_0,\p_1 \eta)$ and $F^7 = \swh(\dt \eta)$.  The lemma tells us that 
\begin{equation}\label{ap_decay_10}
 \frac{\mu}{2} \int_\Omega \abs{\sga  u}^2 J 
+\int_{\Sigma_s} \lambda J \abs{ u \cdot \tau}^2 + \bs{ u\cdot \N} = -\ip{\eta}{\dt \eta}_{1,\sigma} - \int_{-\ell}^\ell \sigma \R(\p_1 \zeta_0,\p_1 \eta) \p_1 \dt \eta - [u\cdot \N, \swh(\dt \eta)]_\ell.
\end{equation}
We may write $\swh(z) = \frac{1}{\kappa} \int_0^z (z-r) \sw''(r) dr$, which allows us to argue as in the proof of Theorem \ref{ee_W} to deduce that $\abs{\swh(\dt \eta)} \ls \abs{\dt \eta}^2$.  From this and \eqref{ap_decay_10} we  immediately deduce, using Proposition \ref{R_prop} and the fact that $\dt \eta = u\cdot \N$ at the endpoints, that
\begin{equation}\label{ap_decay_7}
 \ns{u}_{1} + \ns{u\cdot \tau}_{L^2(\Sigma_s)} + \bs{u\cdot \N} \ls (1 + \sqrt{\E}) \norm{\eta}_{1} \norm{\dt \eta}_1  + \sqrt{\E} \ns{\dt \eta}_1 \ls \seb.
\end{equation}
Theorem \ref{pressure_est} provides the estimate 
\begin{equation}\label{ap_decay_8}
 \ns{p}_0 \ls \ns{u}_1 \ls \seb.
\end{equation}
Then \eqref{ap_decay_0} follows by combining \eqref{ap_decay_6}, \eqref{ap_decay_7}, and \eqref{ap_decay_8}.

\end{proof}

Next we complete our a priori estimate by proving the higher-order bound.

\begin{thm}\label{ap_bound}
Let $\omega \in (0,\pi)$ be the angle formed by $\zeta_0$ at the corners of $\Omega$,  $\delta_\omega \in [0,1)$ be given by \eqref{crit_wt}, and $\delta \in (\delta_\omega,1)$.   There exists a universal constant $\gamma >0$ such that if 
\begin{equation}
 \sup_{0 \le t \le T} \E(t) + \int_0^T \D(t) dt \le \gamma,
\end{equation}
then 
\begin{equation}\label{ap_bound_0}
\sup_{0 \le t \le T} \E(t) + \int_0^T \D(t) dt \ls \E(0).
\end{equation}
\end{thm}
\begin{proof}
Let $\gamma < 1$ be as small as in Theorem \ref{ap_diff_ineq}.  Then there is a universal constant $C>0$ such that
\begin{equation}\label{ap_bound_1}
 \frac{d}{dt} \left( \mathfrak{E} + \mathfrak{F}\right) +  C \D \le 0.
\end{equation}
Again, we know from Proposition \ref{ap_fed_ests} that  
\begin{equation}
 \mathfrak{E} \ls \seb \ls \mathfrak{E} \text{ and } 0 \le  \hal \mathfrak{E} \le \mathfrak{E} + \mathfrak{F} \le \frac{3}{2} \mathfrak{E}.
\end{equation}
This allows us to integrate \eqref{ap_bound_1} to deduce that 
\begin{equation}
 \hal \mathfrak{E}(t) + C \int_0^t \D(s) ds \le \mathfrak{E}(0) + \mathfrak{F}(0) \le \frac{3}{2} \mathfrak{E}(0),
\end{equation}
and hence that 
\begin{equation}\label{ap_bound_2}
 \seb(t) + \int_0^t \D(s) ds \ls \seb(0)
\end{equation}
for all $t\in [0,T]$.

Now, if $X$ is a real Hilbert space and $f \in H^1([0,T];X)$ then 
\begin{equation}
 \frac{d}{dt} \ns{f(t)}_X = 2(f(t),\dt f(t))_X,
\end{equation}
and upon integrating and applying Cauchy-Schwarz we find that
\begin{equation}
 \ns{f(t)}_X \le \ns{f(0)}_X + \int_0^t \left(\ns{f(s)}_X + \ns{\dt f(s)}_X\right) ds.
\end{equation}
We use this estimate to bound
\begin{multline}\label{ap_bound_3}
 \ns{\eta(t)}_{W^{5/2}_\delta} +  \ns{\dt \eta(t)}_{3/2} +  \ns{u(t)}_{W^{2}_\delta} +  \ns{\dt u(t)}_{W^{2}_\delta} +  \ns{p(t)}_{W^{1}_\delta} +  \ns{\dt p(t)}_{W^{1}_\delta} \\
\le  \ns{\eta(0)}_{W^{5/2}_\delta} +  \ns{\dt \eta(0)}_{3/2} +  \ns{u(0)}_{W^{2}_\delta} +  \ns{\dt u(0)}_{W^{2}_\delta} 
\\
+  \ns{p(0)}_{W^{1}_\delta} +  \ns{\dt p(0)}_{W^{1}_\delta} + \int_0^t \D(s) ds.
\end{multline}
Then we combine \eqref{ap_bound_2} and \eqref{ap_bound_3} to deduce that 
\begin{equation} 
 \E(t) + \int_0^t \D(s) ds \ls \E(0)
\end{equation}
for all $t\in [0,T]$, which then easily implies \eqref{ap_bound_0}.
 
\end{proof}

\subsection{Global existence and decay }

We now state our main result on the global existence and decay of solutions.

\begin{thm}\label{gwp}
Let $\omega \in (0,\pi)$ be the angle formed by $\zeta_0$ at the corners of $\Omega$,  $\delta_\omega \in [0,1)$ be given by \eqref{crit_wt}, and $\delta \in (\delta_\omega,1)$.  There exists a universal smallness parameter $\gamma >0$ such that if 
\begin{equation}
 \E(0) \le \gamma,
\end{equation}
then there exists a unique global solution triple $(u,p,\eta)$ such that 
\begin{equation}
 \sup_{t \ge 0} \left( \E(t) + e^{\lambda t} \left[ \seb(t) +  \ns{u(t)}_{1} + \ns{u(t)\cdot \tau}_{L^2(\Sigma_s)} + \bs{u\cdot \N(t)} + \ns{p(t)}_0 \right]    \right) + \int_0^\infty \D(t) dt \le C \E(0),
\end{equation}
where $\lambda,C >0$ are  universal constants.
\end{thm}
\begin{proof}
The result follows from coupling Theorems \ref{ap_bound} and \ref{ap_decay} with the local existence theory and a standard continuation argument.
\end{proof}

\appendix

\section{Recording the nonlinearities }\label{app_nonlin_form}

The governing equations for $(u,p,\eta)$ are \eqref{geometric}, where $\R$ is defined by \eqref{R_def}.  When we apply $\dt$ and $\dt^2$ to this system we get 
\begin{equation}\label{time_diff} 
 \begin{cases}
  \diva\Sa(q,v)  =F^1 & \text{in } \Omega \\
 \diva v = F^2 & \text{in }\Omega \\
 \Sa(q,v) \N = g\xi \N -\sigma \p_1\left( \frac{\p_1 \xi}{(1+\abs{\p_1 \zeta_0}^2)^{3/2}} + F^3\right) \N  + F^4& \text{on } \Sigma \\
 (\Sa(q,v)\nu - \beta v)\cdot \tau =F^5 &\text{on }\Sigma_s \\
 v\cdot \nu =0 &\text{on }\Sigma_s \\
 \dt \xi = v \cdot \N  + F^6& \text{on } \Sigma \\
 \kappa \dt \xi(\pm \ell,t) = \mp \sigma \left(\frac{\p_1 \xi}{(1+\abs{\p_1 \zeta_0}^2)^{3/2}} + F^3 \right)(\pm \ell,t) - \kappa F^7(\pm\ell,t)
 \end{cases}
\end{equation}
for $v = \dt^j u$, $\xi = \dt^j \eta$, and $q = \dt^j p$.  We now identify the form of the forcing terms that appear in these equations.

\subsection{Terms when $\dt$ is applied }\label{fi_dt1}

First, we record the forcing terms appearing in \eqref{time_diff} when $\dt$ is applied.  
\begin{equation}\label{dt1_f1}
 F^1 = - \diverge_{\dt \A} S_\A(p,u) + \mu \diva \sg_{\dt \A} u
\end{equation}
\begin{equation}\label{dt1_f2}
 F^2 = -\diverge_{\dt \A} u
\end{equation}
\begin{equation}\label{dt1_f3}
 F^3 = \dt [  \R(\p_1 \zeta_0,\p_1 \eta)]
\end{equation}
\begin{equation}\label{dt1_f4}
 F^4 = \mu \sg_{\dt \A} u \N +\left[ g\eta  - \sigma \p_1 \left(\frac{\p_1 \eta}{(1+\abs{\p_1 \zeta_0}^2)^{3/2}} +  \R(\p_1 \zeta_0,\p_1 \eta) \right)  - S_{\A}(p,u) \right]  \dt \N 
\end{equation}
\begin{equation}\label{dt1_f5}
 F^5 = \mu \sg_{\dt \A} u \nu \cdot \tau
\end{equation}
\begin{equation}\label{dt1_f6}
 F^6 = u \cdot \dt \N = -u_1 \p_1 \dt \eta.
\end{equation}
\begin{equation}\label{dt1_f7}
 F^7 = \swh'(\dt \eta) \dt^2 \eta.
\end{equation}
Note that a key feature of $F^6$ is that it vanishes at $\pm \ell$ since $u_1$ vanishes there.

\subsection{Terms when $\dt^2$ is applied }\label{fi_dt2}

Here we record the forcing terms appearing in \eqref{time_diff} $\dt^2$ is applied.  
\begin{multline}\label{dt2_f1}
 F^1 = - 2\diverge_{\dt \A} S_\A(\dt p,\dt u) + 2\mu \diva \sg_{\dt \A} \dt u  \\
- \diverge_{\dt^2 \A} S_\A(p,u) + 2 \mu \diverge_{\dt \A} \sg_{\dt \A} u + \mu \diva \sg_{\dt^2 \A} u
\end{multline}
\begin{equation}\label{dt2_f2}
 F^2 = -\diverge_{\dt^2 \A} u - 2\diverge_{\dt \A}\dt u
\end{equation}
\begin{equation}\label{dt2_f3}
 F^3 = \dt^2 [ \R(\p_1 \zeta_0,\p_1 \eta)]
\end{equation}
\begin{multline}\label{dt2_f4}
 F^4 = 2\mu \sg_{\dt \A} \dt u \N + \mu \sg_{\dt^2 \A} u \N + \mu \sg_{\dt \A} u \dt \N\\
+\left[  2g \dt \eta  - 2\sigma \p_1 \left(\frac{\p_1 \dt \eta}{(1+\abs{\p_1 \zeta_0}^2)^{3/2}} + \dt[ \R(\p_1 \zeta_0,\p_1 \eta)] \right)  -2 S_{\A}(\dt p,\dt u) \right]  \dt \N 
\\
+ \left[ g\eta  - \sigma \p_1 \left(\frac{\p_1 \eta}{(1+\abs{\p_1 \zeta_0}^2)^{3/2}} +  \R(\p_1 \zeta_0,\p_1 \eta) \right)  - S_{\A}(p,u) \right]  \dt^2 \N
\end{multline}
\begin{equation}\label{dt2_f5}
 F^5 = 2 \mu \sg_{\dt \A} \dt u \nu \cdot \tau + \mu \sg_{\dt^2 \A} u \nu \cdot \tau
\end{equation}
\begin{equation}\label{dt2_f6}
 F^6 = 2 \dt u \cdot \dt \N + u \cdot \dt^2 \N = -2 \dt u_1 \p_1 \dt \eta - u_1 \p_1 \dt^2 \eta.
\end{equation}
\begin{equation}\label{dt2_f7}
 F^7 = \swh'(\dt \eta) \dt^3 \eta + \swh''(\dt \eta) (\dt^2 \eta)^2.
\end{equation}
Again a key feature of $F^6$ is that it vanishes at $\pm \ell$ since $u_1$ and $\dt u_1$ vanish there.

\section{Estimates for $\R$ }\label{app_r}

Recall that $\R$ is given by \eqref{R_def}.  The following result records the essential estimates of $\R$.  We omit the proof for the sake of brevity.
\begin{prop}\label{R_prop}
The mapping $\R \in C^\infty(\Rn{2})$ defined by \eqref{R_def} obeys the following estimates.
\begin{multline}
 \sup_{(y,z) \in \Rn{2}} \left[ \abs{\frac{1}{z^3}\int_0^z \R(y,s)ds} + \abs{\frac{\R(y,z)}{z^2}} + \abs{\frac{\p_z \R(y,z)}{z}} + \abs{\frac{\p_y \R(y,z)}{z^2}}  \right.
\\
\left.
 + \abs{ \p_z^2 \R(y,z)} + \abs{\frac{\p_y^2 \R(y,z)}{z^2}} + \abs{\frac{\p_z \p_y \R(y,z)}{z}}
 + \abs{ \p_z^3 \R(y,z)} + \abs{ \p_z^2 \p_y \R(y,z)} \right]   < \infty.
\end{multline}
\end{prop}

\section{Weighted Sobolev spaces}\label{app_weight}

We begin our discussion of weighted Sobolev spaces by recalling Hardy's inequality.

\begin{lem}[Hardy's inequality]\label{hardy_ineq}
Assume that $\alpha >1$ and $p \in [1,\infty)$. Then
\begin{equation}\label{hardy_ineq_02}
 \left(\int_0^\infty y^{p(\alpha-1) -1} \left(\int_y^\infty \abs{\varphi(z)} dz \right)^{p} dy \right)^{1/p} \le   \frac{1}{ (\alpha-1)} \left(\int_0^\infty \abs{\varphi(z)}^p z^{\alpha p-1} dz  \right)^{1/p}.
\end{equation}
\end{lem}
\begin{proof}
This is one of the well-known Hardy inequalities.  A proof may be found, for instance, in \cite{h_l_p}.
\end{proof}

Next we record a useful application of Hardy's inequality.

\begin{prop}\label{hardy_sobolev}
Let $K \subset \Rn{n}$ be an open cone and suppose that $\delta  > 1 - n/2$ and $r \in [2,\infty)$.  Then there exists a constant $C(n,\delta) >0$ such that 
\begin{equation}\label{hardy_sobolev_01}
 \left(\int_K \abs{x}^{2(\delta-1)} \abs{\psi(x)}^2 dx   \right)^{1/2} \le C(n,\delta) \left( \int_K \abs{x}^{2\delta} \abs{\nab \psi(x)}^2 dx \right)^{1/2}
\end{equation}
for all $\psi \in C^1_c(\bar{K})$.
\end{prop}
\begin{proof}
Writing $(s,\theta)$ for $s>0$ and $\theta \in  K \cap \mathbb{S}^{n-1}$, we then have that 
\begin{equation}
 \int_K \abs{x}^{2(\delta-1)} \abs{\psi(x)}^2 dx = \int_{ K \cap \mathbb{S}^{n-1}} \int_0^\infty s^{2(\delta-1) + n -1} \abs{\psi(s,\theta)}^r ds d\theta
\end{equation}
and 
\begin{equation}
 \int_K \abs{x}^{2\delta} \abs{\nab \psi(x)}^2 dx \ge  \int_K \abs{x}^{2\delta} \abs{x \abs{x}^{-1} \cdot \nab \psi(x)}^2 dx = \int_{ K \cap \mathbb{S}^{n-1}} \int_0^\infty s^{2\delta + n -1} \abs{\p_s \psi(s,\theta)}^2 ds d\theta.
\end{equation}

To prove \eqref{hardy_sobolev_01} it thus suffices to show that 
\begin{equation}\label{hardy_sobolev_1}
 \int_0^\infty s^{2(\delta-1) + n -1} \abs{\psi(s,\theta)}^2 ds \le \int_0^\infty s^{2\delta + n -1} \abs{\p_s \psi(s,\theta)}^2 ds  
\end{equation}
for each $\theta \in  K \cap \mathbb{S}^{n-1}$.  To prove this we set $\varphi(s,\theta) = \p_s \psi(s,\theta)$ and note that 
\begin{equation}
 \abs{\psi(s,\theta)} =  \abs{ \int_s^\infty \p_s \psi(z,\theta)dz} \le \int_s^\infty \abs{\varphi(z,\theta)}dz,
\end{equation}
which in turn means that 
\begin{equation}
 \int_0^\infty s^{2(\delta-1) + n -1} \abs{\psi(s,\theta)}^2 ds \le \int_0^\infty y^{2(\alpha-1) -1} \left(\int_y^\infty \abs{\varphi(z,\theta)} dz \right)^{2} dy.
\end{equation}
Applying  Lemma \ref{hardy_ineq} with $\alpha = \delta + n/2 > 1$ and $p=2$ to estimate the right side, we then find that \eqref{hardy_sobolev_1} holds.

\end{proof}

Next we define various weighted Sobolev spaces on the equilibrium domain $\Omega$. We write 
\begin{equation}\label{corner_def}
M = \{(-\ell,\zeta_0(\ell)), (\ell, \zeta_0(\ell))\} 
\end{equation}
for the pair of corner points of $\Omega$.   For $0 < \delta < 1$ and $k \in \N$ we let $W^k_\delta(\Omega)$ denote the space of functions such that $\ns{f}_{W^k_\delta} < \infty$, where
\begin{equation}
\ns{f}_{W^k_\delta} = \sum_{\abs{\alpha}\le k} \int_\Omega \dist(x,M)^{2\delta} \abs{\p^\alpha f(x)}^2 dx.
\end{equation}

A consequence of Hardy's inequality (see for example Lemma 7.1.1 in \cite{kmr_1}) is that we have the continuous embeddings
\begin{equation}\label{hardy_embed}
 W^1_\delta(\Omega) \hookrightarrow H^0(\Omega),  W^2_\delta(\Omega) \hookrightarrow H^1(\Omega), \text{ and } H^1(\Omega) \hookrightarrow W^0_{-\delta}(\Omega)
\end{equation}
when $0 < \delta < 1$.  We define the trace spaces $W^{k-1/2}_\delta(\p \Omega)$ in the obvious way.  It can be shown (see for example Section 9.1.2 of \cite{kmr_3}) that 
\begin{equation}\label{hardy_embed_2}
\abs{\int_{\p \Omega} f (v\cdot \tau)  } \ls \norm{f}_{W^{1/2}_\delta(\p \Omega)} \norm{v}_{H^1(\Omega)}
\end{equation}
for all $f \in W^{1/2}_\delta(\Omega)$ and $v \in H^1(\Omega)$ when $0 < \delta < 1$.  Finally, we will also need the spaces 
\begin{equation}
\oW^k_\delta(\Omega) = \{u \in W^k_\delta(\Omega) \st \int_\Omega u =0\}
\end{equation}
for $k \ge 1$.

Next we record a useful embedding.

\begin{prop}\label{weigted_nest}
Let $k \in \N$ and $\delta_1,\delta_2 \in \R$ with $\delta_1 < \delta_2$.  Then we have that 
\begin{equation}
 W^k_{\delta_1}(\Omega) \hookrightarrow W^k_{\delta_2}(\Omega).
\end{equation}
\end{prop}
\begin{proof}
 This follows from the trivial estimate $\dist(x,M)^{2\delta_2} \ls \dist(x,M)^{2\delta_1}$ in $\Omega$. 
\end{proof}

We will also need the following embedding result.

\begin{prop}\label{weighted_embed}
 If $k \in \mathbb{N}$ and $0 < \delta < 1$, then
\begin{equation}\label{weighted_embed_01}
 W^k_\delta(\Omega) \hookrightarrow W^{k,q}(\Omega),
\end{equation}
where
\begin{equation}\label{weighted_embed_02}
 1 \le q < \frac{2}{1+\delta}.
\end{equation}
In particular, 
\begin{equation}\label{weighted_embed_03}
W^1_\delta(\Omega) \hookrightarrow L^p(\Omega) 
\end{equation}
when 
\begin{equation}
 1 \le p < \frac{2}{\delta}.
\end{equation}
\end{prop}
\begin{proof}
By H\"older's inequality, for $1 \le q < 2$ we have that 
\begin{equation}
 \int_{\Omega} \abs{f}^q \le \left( \int_\Omega \abs{f}^2 \dist(\cdot,M)^{2\delta}  \right)^{q/2} \left(\int_\Omega \frac{1}{\dist(\cdot,M)^{2\delta q/(2-q) } }\right)^{1-q/2} 
\end{equation}
and the latter integral is finite if and only if 
\begin{equation}
 \frac{2 \delta q}{2-q} -1 < 1,
\end{equation}
which is equivalent to \eqref{weighted_embed_02}.  Thus 
\begin{equation}
 \norm{f}_{W^{k,q}} \ls \norm{f}_{W^k_\delta}
\end{equation}
for every $f \in W^k_\delta(\Omega)$, and so we have \eqref{weighted_embed_01}.   The embedding \eqref{weighted_embed_03}  follows from the using the first with $k=1$ and the usual Sobolev embedding $W^{1,q}(\Omega) \hookrightarrow L^p(\Omega)$.
\end{proof}

Next we show a boundary embedding.

\begin{prop}\label{weighted_trace}
Suppose that $0 < \delta < 1$.  Then $W^{1/2}_\delta(\p \Omega) \hookrightarrow L^q(\p\Omega)$, where
\begin{equation}
 1 \le q < \frac{2}{1+\delta}.
\end{equation}
\end{prop}
\begin{proof}
Suppose that $f \in W^{1/2}_\delta(\p \Omega)$.  Then there exists $F \in W^{1}_\delta(\Omega)$ such that $F = f$ on $\p \Omega$ and $\norm{F}_{W^1_\delta} \le 2 \norm{f}_{W^{1/2}_\delta}.$  According to Proposition \ref{weighted_embed} we have that $F \in W^{1,q}(\Omega)$.  The  usual trace theory then implies that 
\begin{equation}
 \norm{f}_{L^q(\p\Omega)} = \norm{F \vert_{\p \Omega}}_{L^q(\p \Omega)} \ls \norm{F}_{W^{1,q}(\Omega)} \ls \norm{F}_{W^1_\delta(\Omega)} \ls \norm{f}_{W^{1/2}_\delta}.
\end{equation}
This yields the desired embedding.
\end{proof}

Next we show another consequence of Hardy's inequality.

\begin{prop}\label{weighted_sobolev_embed}
If $0 < \delta < 1$, then
\begin{equation}
 W^1_\delta(\Omega) \hookrightarrow W^{0}_{\delta-1}(\Omega).
\end{equation}
\end{prop}
\begin{proof}
 This follows from the Hardy inequality.  A proof may be found in Lemma 7.1.1 of \cite{kmr_1}, for instance.
\end{proof}

Next we study some other weighted estimates.

\begin{thm}\label{weighted_sobolev_thm}
Suppose that $0 < \delta < 1$.  Let $M$ be the corner set given by \eqref{corner_def}.  Then for each $q \in [1,\infty)$ we have that
\begin{equation}\label{weighted_sobolev_thm_0}
 \norm{\dist(\cdot,M)^\delta f}_{L^q} \ls \norm{f}_{W^1_\delta}
\end{equation}
for all $f \in W^1_\delta(\Omega)$.
\end{thm}
\begin{proof}
Note first that $\dist(\cdot,M)$ is Lipschitz, differentiable a.e., and satisfies $\abs{\nab \dist(\cdot,M)} =1$ a.e.   For $f \in W^1_\delta(\Omega)$ we compute 
\begin{equation}
 \nab(\dist(\cdot,M)^\delta f) = \dist(\cdot,M)^\delta \nab f + \delta \dist(\cdot,M)^{\delta-1} \nab \dist(\cdot, M) f
\end{equation}
and hence Proposition \ref{weighted_sobolev_embed} allows us to estimate
\begin{equation}
 \norm{\nab (\dist(\cdot,M)^\delta f)}_{L^2} \ls \norm{\dist(\cdot,M)^\delta \nab f}_{L^2} +  \norm{ \dist(\cdot,M)^{\delta-1} f}_{L^2} \ls \norm{f}_{W^1_\delta}.
\end{equation}
Since we automatically have $\norm{\dist(\cdot,M)^\delta f}_{L^2} \le \norm{f}_{W^1_\delta}$, we then conclude that 
\begin{equation}
 \norm{\dist(\cdot,M)^\delta f}_{H^1} \ls \norm{f}_{W^1_\delta}.
\end{equation}
Consequently, the usual Sobolev embedding in $\Omega \subset \Rn{2}$ imply \eqref{weighted_sobolev_thm_0}.
\end{proof}

\section{Product estimates}\label{app_prods}

We now prove a product estimate.

\begin{lem}\label{int_prod}
Let $\Omega \subset \Rn{2}$.  Suppose that $f \in H^r(\Omega)$ for $r \in (0,1)$ and $g \in H^1(\Omega)$.  Then $fg \in H^\sigma$ for every $\sigma \in (0,r)$, and 
\begin{equation}
 \norm{fg}_{\sigma} \le C(r,\sigma) \norm{f}_r \norm{g}_1.
\end{equation}
 
\end{lem}
\begin{proof}
Define the linear operator $T$ via $Tf = fg$.  Then by H\"{o}lder's inequality and the Sobolev embedding, we have
\begin{equation}
 \norm{fg}_{L^p} \le \norm{f}_{L^2} \norm{g}_{L^q} \ls \norm{f}_{H^0} \norm{g}_{H^1}
\end{equation}
for
\begin{equation}\label{inp_1}
 \frac{1}{p} = \frac{1}{2} + \frac{1}{q} \text{ and } 2 \le q < \infty.
\end{equation}
From this we see that $T: H^0 \to L^p$ is a bounded linear operator for any $1 \le p < 2$ and
\begin{equation}
 \norm{T}_{\L(H^0,L^p)} \ls \norm{g}_1.
\end{equation}

Similarly, 
\begin{multline}
 \norm{Tf}_{W^{1,p}} \ls \norm{fg}_{L^p} + \norm{f \nab g}_{L^p} + \norm{\nab f g}_{L^p} \\
 \ls  \norm{f}_{H^0} \norm{g}_{H^1} + \norm{f}_{L^q} \norm{\nab g}_{L^2} + \norm{\nab f}_{L^2} \norm{g}_{L^q} \\
\ls \norm{f}_{H^0} \norm{g}_{H^1} + \norm{f}_{H^1} \norm{ g}_{H^1} + \norm{ f}_{H^1} \norm{g}_{H^1} 
\ls \norm{f}_{H^1} \norm{g}_{H^1}.
\end{multline}
This means that $T: H^1 \to W^{1,p}$ is a bounded linear operator for any $1 \le p < 2$ and
\begin{equation}
 \norm{T}_{\L(H^1,W^{1,p})} \ls \norm{g}_1.
\end{equation}

Now, the usual interpolation theory implies that
\begin{equation}
 T : [H^0,H^1]_r \to [L^p, W^{1,p}]_r,
\end{equation}
which means that (see Adams and Fournier \cite{adams_fournier})
\begin{equation}
 T : H^r \to W^{r,p}
\end{equation}
is a bounded linear operator with $\norm{T f}_{W^{r,p}} \ls \norm{f}_{H^r} \norm{g}_{H^1}$.

Now we use the embedding
\begin{equation}
 W^{r,p} \hookrightarrow H^{r + 1 - 2/p}
\end{equation}
to deduce that
\begin{equation}
 \norm{fg}_{r+1-2/p} \ls \norm{f}_{H^r} \norm{g}_{H^1}.
\end{equation}
Returning to \eqref{inp_1}, we have that
\begin{equation}
 \sigma := r + 1 - \frac{2}{p} = r + 1 - \left( 1 + \frac{2}{q}\right) = r - \frac{2}{q}
\end{equation}
can take on any value in $(0,r)$ by choosing appropriate $q \in (2/r ,\infty)$.
\end{proof}

\begin{remark}
 If $\Omega \subset \Rn{3}$ then the same argument works with $\sigma = r - (n-2)/n$.  In dimension $n\ge 4$ we fail to get an embedding into $H^\sigma$ for $\sigma \ge 0$.
\end{remark}

Next we state a product estimate in weighted spaces on $\Sigma$

\begin{prop}\label{weight_prod_1}
Suppose that $f \in W^1_\delta(\Omega)$ for $0 < \delta < 1$ and that $g \in H^{1+\kappa}(\Omega)$ for $0 < \kappa < 1$.  Then  $f g \in W^1_\delta(\Omega)$ and 
\begin{equation}
 \norm{fg}_{W^1_\delta} \ls \norm{f}_{W^1_\delta} \norm{g}_{1+\kappa}.
\end{equation}
\end{prop}
\begin{proof}

Since $\Omega \subset \Rn{2}$ we have the Sobolev embedding $H^{1+\kappa}(\Omega) \hookrightarrow L^\infty(\Omega)$, and so we have that 
\begin{equation}
 \int_\Omega \dist(\cdot,M)^{2\delta} \left[ \abs{f}^2  \abs{g}^2 + \abs{\nab f}^2 \abs{g}^2 \right]  \ls  \ns{f}_{W^1_\delta} \ns{g}_{L^\infty} \ls  \ns{f}_{W^1_\delta} \ns{g}_{1+\kappa}.
\end{equation}
Consequently, in order to prove the desired estimate it suffices to prove that 
\begin{equation}\label{weight_prod_1_1}
 \int_\Omega \dist(\cdot,M)^{2\delta} \abs{f}^2 \abs{\nab g}^2 \ls \ns{f}_{W^1_\delta} \ns{g}_{1+\kappa}.
\end{equation}

First note that $\nab g \in H^\kappa(\Omega) \hookrightarrow L^{2/(1-\kappa)}(\Omega)$.  We may then use H\"older's inequality to estimate
\begin{multline}
 \int_\Omega \dist(\cdot,M)^{2\delta} \abs{f}^2 \abs{\nab g}^2  \le \left(\int_\Omega \dist(\cdot,M)^{2\delta/\kappa} \abs{f}^{2/\kappa}  \right)^\kappa \left(\int_\Omega \abs{\nab g}^{2/(1-\kappa)} \right)^{1-\kappa} \\
 \ls \ns{\dist(\cdot,M)^\delta f}_{L^{2/\kappa}}  \ns{g}_{1+\kappa}.
\end{multline}
Now, $2/\kappa \in (2,\infty)$ and $0 < \delta <1$, so Theorem \ref{weighted_sobolev_thm}  implies that $\norm{\dist(\cdot,M)^\delta f}_{L^{2/\kappa}} \ls \norm{f}_{W^1_\delta}$, and thus we deduce that  \eqref{weight_prod_1_1} holds.

\end{proof}

Next we record a boundary result.

\begin{prop}\label{weight_prod_half}
Let $f\in W^{1/2}_\delta(\Sigma)$ for $0 < \delta < 1$, and let $g \in H^{1/2 + \kappa}(\Sigma)$ for $\kappa \in (0,1)$.  Then $fg \in W^{1/2}_\delta(\Sigma)$ and
\begin{equation}
 \norm{fg}_{W^{1/2}_\delta} \ls \norm{f}_{W^{1/2}_\delta} \norm{g}_{1/2 + \kappa}.
\end{equation}
\end{prop}
\begin{proof}
Using trace theory, we may find $F \in W^1_\delta(\Omega)$ and $G \in H^{1+\kappa}(\Omega)$ such that $F = f$ and $G = g$ on $\Sigma$ and 
\begin{equation}
 \norm{F}_{W^1_\delta} \ls \norm{f}_{W^{1/2}_\delta} \text{ and } \norm{G}_{1+\kappa} \ls \norm{g}_{1/2+\kappa}.
\end{equation}
Applying Proposition \ref{weight_prod_1} to $F$ and $G$ shows that $FG \in W^1_\delta(\Omega)$ and that $\norm{FG}_{W^1_\delta} \ls \norm{F}_{W^1_\delta} \norm{G}_{1+\kappa}$.  Thus
\begin{equation}
 \norm{fg}_{W^{1/2}_\delta} \le \norm{FG}_{W^1_\delta} \ls \norm{F}_{W^1_\delta} \norm{G}_{1+\kappa} \ls \norm{f}_{W^{1/2}_\delta} \norm{g}_{1/2+s}.
\end{equation}

\end{proof}

\subsection{Coefficient estimates}\label{app_coeff}

Here we are concerned with how the size of $\eta$ can control the ``geometric'' terms that appear in the equations.
\begin{lem}\label{eta_small}
Let $0 < \delta < 1$.  There exists a universal $0 < \gamma < 1$ so that if $\ns{\eta}_{W^{5/2}_\delta}(\Sigma) \le \gamma$, then
\begin{equation}\label{es_01}
\begin{split}
 & \norm{J-1}_{L^\infty(\Omega)} +\norm{A}_{L^\infty(\Omega)} \le \hal, \\
 & \norm{\N-1}_{L^\infty(\Gamma)} + \norm{K-1}_{L^\infty(\Gamma)} \le \hal, \text{ and }  \\
 & \norm{K}_{L^\infty(\Omega)} + \norm{\mathcal{A}}_{L^\infty(\Omega)} \ls 1.
 \end{split}
\end{equation}
Also, the map $\Phi$ defined by \eqref{mapping_def} is a diffeomorphism.
\end{lem}
\begin{proof}
These follow from the weighted embeddings proved in Appendix \ref{app_weight} and the usual Sobolev product estimates.  We refer to Lemma 2.4 of \cite{gt_hor} for more details in the case of a $3D$ domain.
\end{proof}

\section{Equilibrium surface}\label{app_surf}

In this appendix we collect some well-known facts about the equilibrium capillary surface problem \eqref{zeta0_eqn}.

\subsection{Uniqueness}

We begin by proving that there is at most one solution to \eqref{zeta0_eqn}.  This is proven in greater generality in Theorem 5.1 of \cite{finn}, but we record the simple $1-D$ proof here for the reader's convenience.

\begin{thm}\label{eq_surf_unique}
There exists at most one solution to the problem
\begin{equation} 
\begin{cases}
g \zeta - \sigma \mathcal{H}(\zeta) = P & \text{in }(-\ell,\ell) \\
\frac{\zeta'}{ \sqrt{1+(\zeta')^2} }(\pm \ell) = \pm \frac{\jg}{\sigma}.
\end{cases}
\end{equation}
\end{thm}
\begin{proof}
Suppose that $\zeta_1,\zeta_2$ are two solutions.  We subtract the equations for $\zeta_2$ from the equations for $\zeta_1$, multiply by $\zeta_1 -\zeta_2$, and integrate by parts over $(-\ell,\ell)$ to deduce that 
\begin{multline}\label{eq_surf_unique_1}
  \int_{-\ell}^\ell g\abs{\zeta_1 - \zeta_2}^2 + \sigma \left(\frac{\zeta_1'}{ \sqrt{1+(\zeta_1')^2} } - \frac{\zeta_2'}{ \sqrt{1+(\zeta_2')^2} } \right) (\zeta_1' - \zeta_2')  \\
 = \left.\sigma \left(\frac{\zeta_1'}{ \sqrt{1+(\zeta_1')^2} } - \frac{\zeta_2'}{ \sqrt{1+(\zeta_2')^2} } \right) (\zeta_1 - \zeta_2) \right\vert_{- \ell}^\ell =0.
\end{multline}
Set $f(z) = z/\sqrt{1+z^2}$ and $\varphi(t) = (x-y)(f(y+t(x-y)) - f(y))$ for fixed pair $x,y \in \Rn{}$.  Then 
\begin{equation}
 \varphi(0) = 0 \text{ and } \varphi'(t) = \frac{\abs{x-y}^2}{(1 + \abs{y + t(x-y)}^2)^{3/2}} \ge 0,
\end{equation}
and so 
\begin{equation}\label{eq_surf_unique_pos}
 (x-y)(f(y)-f(x)) = \varphi(1) = \varphi(0) + \int_0^1 \varphi'(t) dt \ge 0
\end{equation}
for any $x,y \in \Rn{}$.  Applying this inequality in \eqref{eq_surf_unique_1} then shows that
\begin{equation}
 \int_{-\ell}^\ell g \abs{\zeta_1 -\zeta_2}^2 \le 0
\end{equation}
and hence $\zeta_1 = \zeta_2$.
\end{proof}

\subsection{Existence}

We now prove the existence of solutions to \eqref{zeta0_eqn}.  Consider the function $h:(1,\infty) \to (0,\infty)$ given by 
\begin{equation}
h(r) :=  \int_{0}^{\arcsin(\abs{\jg}/\sigma) } \frac{\cos(\psi)}{\sqrt{r-\cos(\psi)}}d\psi.
\end{equation}
It's easy to show that the mapping $h$ is decreasing and satisfies 
\begin{equation}
\lim_{r \to 1} h(r) = \infty \text{ and } \lim_{r \to \infty} h(r) = 0. 
\end{equation}
From this we deduce that there exists a unique 
\begin{equation}
 C = C(g,\sigma,\abs{\jg},\ell) \in (1,\infty)
\end{equation}
such that $h(C) =  \ell  \sqrt{\frac{2g}{\sigma}} \in (0,\infty)$, which is equivalent to  
\begin{equation}\label{C_cond}
 \ell = \sqrt{\frac{\sigma}{2g}} \int_{0}^{\arcsin(\abs{\jg}/\sigma) } \frac{\cos(\psi)}{\sqrt{C-\cos(\psi)}}d\psi,
\end{equation}

With the constant $C = C(g,\sigma,\abs{\jg},\ell) \in (1,\infty) $ determined by \eqref{C_cond} we define the mapping $\Xi :[\arcsin(-\abs{\jg}/\sigma),\arcsin(\abs{\jg}/\sigma)] \to [-\ell,\ell]$ via 
\begin{equation}\label{surf_xi_def}
 \Xi(z) = \sqrt{\frac{\sigma}{2g}} \int_{0}^{z} \frac{\cos(\psi)}{\sqrt{C-\cos(\psi)}}d\psi.
\end{equation}
It's easy to see that $\Xi$ is a smooth increasing diffeomorphism that is an odd function on $[-\ell,\ell]$.  We use $\Xi$ to construct the equilibrium capillary surface.

\begin{thm}\label{chi_thm}
Suppose that $ \jg/\sigma \in (-1,1).$ Define $\chi:[-\ell,\ell]\to \Rn{}$ by 
\begin{equation}
 \chi(x) = \sgn(\jg) \sqrt{\frac{2\sigma}{g}} \sqrt{C(g,\sigma,\abs{\jg},\ell) - \cos(\Xi^{-1}(x))}
\end{equation}
with the understanding that $\chi$ is simply $0$ when sgn$(\jg) = \jg =0$ (in this case we cannot evaluate $\Xi$, but we don't need to).  Then $\chi$ is smooth on $[-\ell,\ell]$ and an even function.  Moreover, $\chi$ is the unique solution to 
\begin{equation} 
\begin{cases}
g \chi - \sigma \mathcal{H}(\chi) = 0& \text{in }(-\ell,\ell) \\
\frac{\chi'}{ \sqrt{1+(\chi')^2} }(\pm \ell) = \pm \frac{\jg}{\sigma}.
\end{cases}
\end{equation} 
\end{thm}
\begin{proof}
It's obvious that $\chi$ is smooth and even.  A direct computation shows that $\chi$ satisfies 
\begin{equation}\label{chi_thm_1}
 \frac{\chi'}{\sqrt{1+(\chi')^2}} = \sgn(\jg) \sin(\Xi^{-1}(x)) \text{ and } \frac{1}{\sqrt{1+(\chi')^2}} = \cos(\Xi^{-1}(x)) = C(g,\sigma,\abs{\jg},\ell) - \frac{g}{2\sigma} \chi^2.
\end{equation}
Upon evaluating the first equation in \eqref{chi_thm_1} at $x = \pm \ell$ we find that 
\begin{equation}
 \frac{\chi'}{ \sqrt{1+(\chi')^2} }(\pm \ell) = \sgn(\jg) \sin(\arcsin(\pm \abs{\jg}/\sigma))  = \pm \frac{\jg}{\sigma}.
\end{equation}
The second equation in \eqref{chi_thm_1} shows that 
\begin{equation}
 g \chi \chi' + \sigma\left(\frac{1}{\sqrt{1+(\chi')^2}} \right)' =0 \text{ and hence } 
 g \chi - \sigma \left(\frac{\chi'}{\sqrt{1+(\chi')^2}} \right)'=0.
\end{equation}
Uniqueness follows from Theorem \ref{eq_surf_unique}.
\end{proof}

Theorem \ref{chi_thm} only constructs a special equilibrium capillary surface for which the pressure vanishes.  However, we can use it to recover arbitrary solutions.

\begin{thm}\label{eq_surf_thm}
Suppose that $\jg/\sigma \in (-1,1)$.  Let $\chi: [-\ell,\ell] \to \Rn{}$ be the function given in Theorem \ref{chi_thm}, and set 
\begin{equation}
 M_{min} = \int_{-\ell}^\ell \left( \chi(x) - \min_{(-\ell,\ell)} \chi\right) dx \ge 0. 
\end{equation}
Then for each $M_{top} > M_{min}$ there exists a unique, smooth, even function $\zeta_0 :[-\ell,\ell] \to (0,\infty)$ and a constant $P_0 \in \Rn{}$ such that 
\begin{equation} 
\begin{cases}
g \zeta_0 - \sigma \mathcal{H}(\zeta_0) = P_0 & \text{in }(-\ell,\ell) \\
\frac{\zeta_0'}{ \sqrt{1+(\zeta_0')^2} }(\pm \ell) = \pm \frac{\jg}{\sigma}.
\end{cases}
\end{equation}
Moreover, 
\begin{equation}
 M_{top} = \int_{-\ell}^\ell \zeta_0(x) dx \text{ and } P_0 =  \frac{g M_{0} -2 \jg}{2\ell}.
\end{equation}
\end{thm}
\begin{proof}
We simply set $\zeta_0 = \chi - \min_{(-\ell,\ell)} \chi + h$ for $h \in (0,\infty)$ determined by 
$h = (M_{top}- M_{min})/(2\ell).$ The stated results then follow from Theorem \ref{chi_thm} and simple calculations. 
\end{proof}

\subsection{Variational characterization}

Although we have not used the calculus of variations to construct $\zeta_0$ we can still show that $\zeta_0$ satisfies a minimization principle.

\begin{thm}\label{surf_min}
 Let $\mathscr{I}$ be the energy functional defined by \eqref{zeta0_energy}.  Let $\zeta_0$, $M_{top}$, and $P_0$ be as in Theorem \ref{eq_surf_thm}.  If $\eta \in W^{1,1}((-\ell,\ell))$ is such that $ \int_{-\ell}^\ell \eta = M_{top}$, then $\mathscr{I}(\zeta_0) \le \mathscr{I}(\eta)$.
\end{thm}
\begin{proof}
 Let $\psi \in  W^{1,1}((-\ell,\ell)) $ be such that $\int_{-\ell}^\ell \psi =0$.  Then 
\begin{equation}
 \int_{-\ell}^\ell (\zeta_0 + t \psi) = M_{top} \text{ for all }t \in \Rn{}.
\end{equation}
We then compute, 
\begin{multline}
 \frac{d}{dt} \mathscr{I}(\zeta_0 + t \psi) = \int_{-\ell}^\ell g \zeta_0 \psi + \sigma \frac{\zeta_0' \psi'}{\sqrt{1+(\zeta_0')^2}} -\jg(\psi(\ell) + \psi(-\ell) ) \\
 + \sigma \int_{-\ell}^\ell \left(\frac{\zeta_0' + t \psi'}{\sqrt{1+\abs{\zeta_0' + t \psi'}^2}} -  \frac{\zeta_0' }{\sqrt{1+(\zeta_0')^2}} \right) \frac{(\zeta_0' + t \psi' - \zeta_0')}{t}.
\end{multline}
The ODE satisfied by  $\zeta_0$ allows us to integrate by parts to deduce that 
\begin{equation}
 \int_{-\ell}^\ell g \zeta_0 \psi + \sigma \frac{\zeta_0' \psi'}{\sqrt{1+(\zeta_0')^2}} -\jg(\psi(\ell) + \psi(-\ell) ) = \int_{-\ell}^\ell P_0 \psi =0.
\end{equation}
On the other hand, \eqref{eq_surf_unique_pos} tells us that 
\begin{equation}
  \sigma \int_{-\ell}^\ell \left(\frac{\zeta_0' + t \psi'}{\sqrt{1+\abs{\zeta_0' + t \psi'}^2}} -  \frac{\zeta_0' }{\sqrt{1+(\zeta_0')^2}} \right) \frac{(\zeta_0' + t \psi' - \zeta_0')}{t} \ge 0 \text{ for all }t \in \Rn{}.
\end{equation}
Thus 
\begin{equation}
  \frac{d}{dt} \mathscr{I}(\zeta_0 + t \psi) \ge 0 \text{ for all }t \in \Rn{},
\end{equation}
which in particular means that 
\begin{equation}\label{surf_min_1}
 \mathscr{I}(\zeta_0 + \psi) = \mathscr{I}(\zeta_0) + \int_0^1  \frac{d}{dt} \mathscr{I}(\zeta_0 + t \psi)  dt \ge \mathscr{I}(\zeta_0).
\end{equation}

Now, if $\eta \in W^{1,1}((-\ell,\ell))$ is such that $ \int_{-\ell}^\ell \eta = M_{top}$, then we may apply \eqref{surf_min_1} to  $\psi = \zeta_0 - \eta$ to deduce that 
\begin{equation}
 \mathscr{I}(\eta) = \mathscr{I}(\zeta_0 +\psi) \ge \mathscr{I}(\zeta_0).
\end{equation}

\end{proof}

\textbf{Acknowledgments:}  We would like to thank Weinan E for bringing this problem to our attention and for fruitful discussions.  We would also like to extend our thanks to the Beijing International Center for Mathematical Research for hosting us.


\end{document}